\documentclass[11pt,twoside]{amsart}

\usepackage{amsmath}
\usepackage{xspace}
\usepackage{amssymb}
\usepackage{bbm}
\usepackage{graphicx}
\usepackage{url}
\usepackage{pifont}
\usepackage{enumerate}
\usepackage{tikz}
\usepackage{verbatim}
\usepackage{mathtools}

\numberwithin{equation}{section}

\newcommand{\mi}{\bbi\xspace}

\DeclareMathSymbol{\varnothing}{\mathord}{AMSb}{"3F}

\DeclareMathOperator{\RE}{Re}
\DeclareMathOperator{\IM}{Im}

\newcommand{\bbC}{\mathbb{C}}

\newcommand{\bbN}{\mathbb{N}}

\newcommand{\bbR}{\mathbb{R}}

\newcommand{\N}{\mathbb{N}}
\newcommand{\R}{\mathbb{R}}
\newcommand{\C}{\mathbb{C}}
\newcommand{\one}{\mathbbm{1}}

\newcommand{\calP}{\mathcal{P}}

\newcommand{\AND}{\quad\text{and}\quad}

\newcommand{\diff}{\mathrm{d}}
\newcommand{\tr}{\mathrm{tr}\,}
\newcommand{\Res}{\mathrm{Res}}

\newcommand{\Liouville}{\mathrm{L}}

\theoremstyle{plain}
\newtheorem{theorem}{Theorem}[section]
\newtheorem*{theorem*}{Theorem}

\newtheorem*{corollary*}{Corollary}
\newtheorem{proposition}[theorem]{Proposition}
\newtheorem*{proposition*}{Proposition}
\newtheorem{lemma}[theorem]{Lemma}
\newtheorem*{lemma*}{Lemma}
\newtheorem{example}[theorem]{Example}
\newtheorem*{example*}{Example}

\newtheorem*{definition*}{Definition}

\newtheorem*{notation*}{Notation}
\newtheorem{remark}[theorem]{Remark}
\newtheorem*{remark*}{Remark}

\newcommand{\bbi}{\mathbbm{i}}

\title[Blowing up sequences of constant mean curvature tori 
in $\mathbb{R}^3$ to minimal surfaces]{Blowing up sequences of constant mean curvature tori 
  in $\mathbb{R}^3$ \\ to minimal surfaces}

\author[E. Carberry]{Emma Carberry}
\email{emma.carberry@sydney.edu.au}
\address{School of Mathematics and Statistics\\University of Sydney\\Australia}

\author[S. Klein]{Sebastian Klein}
\email{math@sebastian-klein.de}
\address{S.~Klein, Mathematics Chair III\\
Universit\"at Mannheim\\
D-68131 Mannheim, Germany}

\author[M. Schmidt]{Martin Ulrich Schmidt}
\email{schmidt@math.uni-mannheim.de}
\address{M.~Schmidt, Mathematics Chair III\\
Universit\"at Mannheim\\
D-68131 Mannheim, Germany}

\thanks{\today.}

\begin{document}

\begin{abstract}
	This paper is motivated by the question of whether a sequence of solutions of a given integrable system can be blown up to obtain a solution of a different integrable system in the limit. We study a specific example of this phenomenon. Namely, we describe a blow-up for immersed constant mean curvature (cmc) planes of finite type with unbounded principal curvatures and derive sufficient conditions under which this blow-up converges to a minimal surface immersion. Passing to the respective Gauss-Codazzi equations,  we are blowing up a sequence of solutions to the sinh-Gordon integrable system to obtain a solution to Liouville's equation, whose integrable system will turn out to be closely related to the Korteweg-de Vries integrable system. Our most important tool for this investigation is the algebraic-geometric correspondence that was established by Pinkall/Sterling and by Hitchin for cmc planes of finite type, which include all cmc tori. 
\end{abstract}

\maketitle

\textbf{Acknowledgement.} Sebastian Klein was funded by the Deutsche Forschungsgemeinschaft, Grant 414903103.

\section{Introduction}
\label{Se:intro}
It is well-known that the Gauss-Codazzi equations for constant mean curvature (cmc) tori reduce to the sinh-Gordon equation. The discovery by Pinkall/Sterling \cite{PS:89} and Hitchin \cite{Hi} of an algebraic-geometric correspondence between cmc tori and algebraic data led to the understanding that the sinh-Gordon equation is an integrable system. These algebraic data include potentials, polynomial Killing fields and spectral curves. We shall study sequences of cmc tori with exploding principal curvatures as a paradigm of the following questions. (1)~By blowing up a sequence of solutions to an integrable system, can one pass to a different integrable system in the limit? (2) Can we determine how the blow-ups of the algebraic and geometric data are related? (3)~If so, then by simultaneously also blowing up the corresponding algebraic data, can one obtain an algebraic geometric correspondence for the limiting data? There exist well-established techniques of blowing up for both sides of the algebraic-geometric correspondence, namely for solutions of partial differential equations \cite{brian-white} and for algebraic curves  \cite {Hartshorne}. In general, relating these two blow-ups appears to be a difficult problem. In this paper we shall blow up certain sequences of cmc tori in such a way as to obtain minimal surfaces in the limit, hence passing from the sinh-Gordon equation to Liouville's equation. Furthermore, in this case we will establish an explicit correspondence between blow-ups of the geometric and the algebraic data. Here, by geometric data we mean the cmc immersions and their conformal factors, which are the solutions of the sinh-Gordon equation. 

More precisely we consider sequences \,$(f_n)_{n\in\mathbb{N}}$\, of smooth cmc immersions of tori into Euclidean 3-space \,$\bbR^3$\, with fixed mean curvature \,$H > 0$\, such that at least one principal curvature explodes. We shall see that this condition on the principal curvatures is necessary so that we get a different integrable system in the limit. We then investigate under which circumstances we may blow up the sequence \,$ f_n $\,
such that  a subsequence of the blown-up \,$f_n$\, converges to a non-trivial surface immersion \,$\tilde{f}$\, into \,$\bbR^3$\,. By ``blowing up'' we mean rescaling both the parameter $ z $ of the plane  and the ambient space \,$\bbR^3$\,. For this purpose we choose a base point \,$z_0$\, and sequences \,$(r_n)_{n\in\mathbb{N}}$\, and \,$(h_n)_{n\in\mathbb{N}}$\, of positive real numbers, and hence introduce a rescaled parameter \,$\tilde{z}$\, and rescaled immersions \,$\tilde{f}_n$\, by
\begin{equation}
	z = z_{0}+r_n\,\tilde{z} \AND \tilde{f}_n(\tilde{z}) = h_n^{-1}\bigr( f_n(z_{0}+r_n\,\tilde{z})-f_n(z_{0})\bigr) \;. 
\end{equation}
In \cite[Theorem~7.3]{brian-white} geometric-analytic methods were used to find a base point \,$z_0$\, such that the \,$(\tilde{f}_n)$\, have a convergent subsequence. However, here we impose further conditions upon the convergence, that we will describe.
If both principal curvatures of the \,$f_n$\, are bounded, then after applying rigid motions in \,$\bbR^3$\, there exists a subsequence of the original \,$(f_n)$\, that converges to another cmc immersion into \,$\bbR^3$\,, whose conformal factor is again a solution of the sinh-Gordon equation. Because we are interested in the case where a blow-up limit produces a solution of a different integrable system, we are led to consider the hypothesis that at least one principal curvature of the \,$f_n$\, diverges. Since the mean curvature is fixed at \,$H$\,, the other principal curvature then also diverges with the opposite sign. Therefore no subsequence of the original sequence \,$(f_n)$\, can converge to a surface immersion in this case. It is therefore natural to consider the blow-up sequence \,$(\tilde{f}_n)$\, in this setting. We choose the scaling sequences \,$r_n$\, and \,$h_n$\, so that the conformal factor of \,$\tilde{f}_n$\, at \,$\tilde{z}=0$\, and the Hopf differential of \,$\tilde{f}_n$\, are bounded and bounded away from zero. This means equivalently that the principal curvatures of the blown up sequence are bounded and bounded away from zero. 
Under an additional condition on the growth behaviour of the conformal factor of \,$f_n$\, at the base point \,$z_0$\,, we will then prove that the limit of this blown up sequence is a surface immersion. The limiting surface will have non-zero principal curvatures of equal magnitude and opposite sign, and hence be minimal. 

We study in detail the example where the limiting minimal surface is a helicoid. This is equivalent to it being ruled, since the only other ruled minimal surfaces are planes, which are excluded here because the principal curvatures are non-zero. The limiting surface being ruled is equivalent to one of its families of asymptotic curves consisting of straight lines, i.e.~having vanishing curvature and torsion. To express this condition in terms of the original cmc immersions \,$f_n$\,, we first note that for \,$n$\, sufficiently large, the principal curvatures of \,$f_n$\, have opposite sign, and therefore the \,$f_n$\, have asymptotic curves. Then the blown-up limit is a helicoid if and only if for one family of asymptotic curves of the \,$f_n$\,, the quotients of both its curvature and its torsion by either principal curvature of the \,$f_n$\, tend to zero. 

Consideration of this example is motivated by the following idea, which will be the subject of further investigation. Since the ruled lines of the blowup have vanishing curvature and torsion, there must exist another blowup by a slower rate in which both curvature and torsion along the corresponding asymptotic curves remain bounded, and at least one of these quantities is bounded away from zero. In all other directions, the curvature will diverge in this slower blowup. It is natural to ask under which circumstances any sequence of nearby asymptotic curves of \,$f_n$\, which start in the neighbourhood of the original one with respect to the slower blow-up converges to the original limiting curve. In this case the whole immersions \,$f_n$\, converge in the slower blow-up to this regular curve. This might provide an example of a soul curve as originally conceived by Pinkall, cf.~\cite[p.~5]{Knoeppel}.





Our most important tool for the investigation of the blow-up of the \,$f_n$\, in the present paper is the algebraic-geometric correspondence mentioned above. We will simultaneously perform the blow-up on the geometric and the algebraic data. Thereby, we will see that on both sides of the algebraic-geometric correspondence the blown up data have a well-defined limit, and this will yield an algebraic-geometric correspondence for the limiting minimal surface. 

The spectral data for the immersions \,$f_n$\, have only finitely many degrees of freedom if the immersions \,$f_n$\, are of \emph{finite type}. Pinkall/Sterling and Hitchin have shown that torus immersions have this property. Our approach is based on the fact that the spectral data for each \,$f_n$\, can be described as a \emph{potential} \,$\zeta_n(\lambda)$\,. This is the \,$\lambda^{-1}$-multiple of
a $(2\times 2)$-matrix-valued polynomial of degree \,$d$\, in the spectral parameter \,$\lambda \in \bbC^\times$\, that depends only on the infinitesimal geometry of the surface at the single point \,$z_0$\,. The number \,$g=d-1$\, is called the \emph{spectral genus} of \,$f_n$\,, and the fact that \,$f_n$\, is of finite type precisely means that \,$g$\, is finite. In order to ensure that the blow-up limit of the potentials yields a potential of finite degree, we shall assume that the spectral genus of the \,$f_n$\, is bounded. Then the original sequence \,$f_n$\, has a subsequence with constant spectral genus \,$g$\,. Hence, in the sequel we shall assume that all \,$f_n$\, have the same spectral genus \,$g$\,. 

To blow up the potentials, we rescale the spectral parameter \,$\lambda$\, with a sequence \,$(\ell_n)$\, of positive real numbers, and then renormalise the rescaled potentials by another sequence \,$(s_n)$\, of positive real numbers. More explicitly, we introduce \,$\tilde{\lambda}$\, and \,$\tilde{\zeta}_n$\, by 
\begin{equation*}
	\lambda = \ell_n\,\tilde{\lambda} \AND \tilde{\zeta}_n(\tilde{\lambda}) = s_n\,\zeta_n(\ell_n\,\tilde{\lambda}) \; .
\end{equation*}
The matrix-valued polynomials \,$\tilde{\lambda}\,\tilde{\zeta}_n(\tilde{\lambda})$\, are of fixed degree \,$g+1$\,, and therefore elements of a finite-dimensional complex vector space. In such spaces, normalised sequences always have convergent subsequences. Hence we may assume that the polynomials \,$\tilde{\lambda}\,\tilde{\zeta}_n(\tilde{\lambda})$\, converge to a polynomial of degree \,$\leq g+1$\,. In general, the limiting polynomials will have only one non-zero coefficient and contain no geometric information. However, we will show that if one chooses \,$\ell_n = r_n^2$\,, we obtain non-trivial polynomials in the limit. It will turn out that these limiting polynomials yield spectral data which describe the minimal surface that is the blow-up limit of the \,$f_n$\,. Thereby we will establish an algebraic-geometric correspondence for minimal surfaces with constant non-vanishing Hopf differential. For this purpose we will transfer Pinkall/Sterling's iterative construction of potentials for cmc tori to minimal surfaces. Because the Gauss-Codazzi equation for minimal surfaces reduces to the Liouville equation, what we have obtained is an algebraic-geometric correspondence for the Liouville equation. We shall see that the corresponding algebraic data are closely related to the algebraic-geometric data for the Korteweg-de Vries (KdV) equation in a sense that will be made precise in Section~\ref{Se:KdV}.

We begin Section~\ref{Se:cmc} with an account of the spectral theory for cmc tori in \,$\R^3$\, tailored to our situation. We then prove in Proposition~\ref{P:cmc:magic-estimates-new} that the growth behaviour of the spectral divisors and of the conformal factor of the immersions, and hence of the potentials, is controlled by the absolute values of the branch points of the spectral curves. 
In Section~\ref{Se:KdV} we detail the
spectral theory for minimal surfaces in \,$\R^3$\,, which occur as the limit of our blown-up sequence of cmc tori.
The spectral theory for cmc tori is well-known, with the exception of the description of the growth behaviour of the potentials in Proposition~\ref{P:cmc:magic-estimates-new}. However, the adaption of polynomial Killing fields/potentials and of the Pinkall-Sterling iteration to minimal surfaces, and hence the algebraic-geometric correspondence for the Liouville equation,
do not appear to be explicitly described in the literature. 
In Section~\ref{Se:blowup1} we prove our main result (Theorem~\ref{T:blowup1:blowup1}), which gives sufficient conditions under which the blowups \,$\tilde{\zeta}_n$\, converge non-trivially. Then the blown-up immersions \,$\tilde{f}_n$\, converge to a minimal surface immersion. 

\section{The integrable system for constant mean curvature surfaces}
\label{Se:cmc}

This paper is concerned with surfaces and curves in \,$\R^3$\,. We denote the standard inner product of \,$\R^3$\, by \,$\langle \,\cdot\,,\,\cdot\,\rangle$\,.
In the context of Wirtinger derivatives \,$f_z, f_{\bar{z}}$\, of functions \,$f$\, mapping into \,$\R^3$\,, we denote by \,$\langle \,\cdot\,,\,\cdot\,\rangle$\, also the
\,$\C$-bilinear continuation of that inner product of \,$\R^3$\, to \,$\C^3$\,. 

{\bf Surface immersions into \,$\R^3$\,.}
We begin by considering a smooth immersion \,$f: X \to \R^3$\, of a 2-dimensional manifold \,$X$\, into \,$\R^3$\,. There exists the structure of a Riemann surface on \,$X$\,
such that the immersion \,$f$\, becomes conformal, and we always regard \,$X$\, as a Riemann surface in this way. We choose a holomorphic coordinate \,$z = x + \mi y$\, of \,$X$\,
and express the fundamental geometric quantities of \,$f$\, locally with respect to \,$z$\,.
Due to \,$f$\, being a conformal immersion with respect to \,$z$\, we have
$$ \|f_x\|^2 = \|f_y\|^2 = 2\langle f_z, f_{\bar z}\rangle > 0 \quad\text{and}\quad \langle f_x, f_y \rangle = 0 \;, $$
hence the  Riemannian metric on \,$X$\, induced by \,$f$\, is locally given by
$$ g = e^\omega\,(\diff x^2+\diff y^2) = e^\omega \,\diff z \, \diff \bar{z} \quad\text{with}\quad \omega = \ln( 2\langle f_z,f_{\bar z} \rangle) \; . $$
The smooth real-valued function \,$\omega$\, is called the \emph{conformal factor} of \,$f$\, (with respect to the coordinate \,$z$\,).
Let \,$N$\, be the positively oriented unit normal field for \,$f$\, (at least locally on the domain of the coordinate \,$z$\,). 
The \emph{mean curvature} \,$H$\, of \,$f$\, is one-half the trace of the shape operator \,$S=g^{-1} h$\, where \,$h$\, is the second fundamental form of \,$f$\,, and therefore given by
$$ H = e^{-\omega}\cdot \tfrac12\langle f_{xx}+f_{yy},N \rangle = \langle f_z, f_{\bar z} \rangle^{-1} \cdot \langle f_{z\bar{z}},N \rangle \; . $$
Moreover the \emph{Hopf differential} \,$Q\,\diff z^2$\, is the \,$\diff z^2$-component of the second fundamental form and therefore given by \,$Q = \langle f_{zz},N \rangle$\,. 
The zeros of \,$Q$\,, i.e.~the points where the second fundamental form of \,$f$\, is diagonal, are called \emph{umbilical points} of \,$f$\,. 
The integrability condition \,$f_{z\bar{z}} = f_{\bar{z}z}$\, for a surface immersion into \,$\R^3$\, is expressed by the equations of Gauss and Codazzi. With respect to the coordinate \,$z$\,, they take the form
\begin{align}
  \label{eq:cmc:gauss}
  2 \omega_{z\bar{z}} + H^2\,e^{\omega} - 4\,|Q|^2\,e^{-\omega} & = 0 \\
  \label{eq:cmc:codazzi}  
  Q_{\bar z} & = e^{\omega} H_z \; .
\end{align}

{\bf The extended frame and the connection form for cmc immersions.}
We now turn our attention to constant mean curvature (cmc) surface immersions \,$f$\, into \,$\R^3$\,, i.e.~to the case where the mean curvature function \,$H$\, is constant
and non-zero.
Then the Codazzi equation~\eqref{eq:cmc:codazzi} shows that the Hopf differential \,$Q\,\diff z^2$\, is holomorphic. Thus \,$f$\, is either totally umbilical (this happens only
if \,$f$\, parameterises part of a round sphere in \,$\R^3$\,), or else the umbilical points of \,$f$\, are discrete. In the latter case, around a non-umbilical point
the coordinate \,$z$\, can be chosen such that the function \,$Q$\, describing the Hopf differential \,$Q\,\diff z^2$\, is constant and non-zero; we will always
choose \,$z$\, in such a way in the sequel. In this setting the Codazzi equation~\eqref{eq:cmc:codazzi} reduces to \,$0=0$\,, so the Gauss equation~\eqref{eq:cmc:gauss}
is the sole condition of integrability for cmc immersions. Note that if we consider \,$H=\tfrac12$\, and \,$|Q|=\tfrac14$\,, then Equation~\eqref{eq:cmc:gauss}
reduces to the \emph{sinh-Gordon equation}
\begin{equation}
\label{eq:cmc:sinhgordon0}
\triangle \omega + \sinh(\omega) = 0 \; .
\end{equation}
For other constant, non-zero choices of \,$H$\, and \,$Q$\,, the Gauss equation~\eqref{eq:cmc:gauss} can be transformed into \eqref{eq:cmc:sinhgordon0} by reparameterisation. For this reason we take the liberty of applying the name \emph{sinh-Gordon equation} to \eqref{eq:cmc:gauss} whenever \,$H$\, and \,$Q$\, are constant and non-zero. 

We note that for any solution \,$(\omega, H, Q)$\, of the Gauss-Codazzi equations \eqref{eq:cmc:gauss}-\eqref{eq:cmc:codazzi} with constant \,$H$\, and any \,$\lambda \in S^1$\,, the triple \,$(\omega, H, \lambda^{-1}Q)$\, is another solution of the Gauss-Codazzi equations. Thus any cmc immersion \,$f$\, belongs to an \emph{associated family} \,$(f_\lambda)_{\lambda \in S^1}$\, of cmc immersions with the same conformal metric and the same mean curvature, but with the rotated Hopf differential \,$\lambda^{-1}\,Q\,\mathrm{d}z^2$\,. 

We now fix a frame for each of the immersions \,$f_\lambda$\,, such a family of frames is called an \emph{extended frame}. The usual choice of such an extended frame is \,$( \tfrac{f_x}{\|f_x\|}, \tfrac{f_y}{\|f_y\|}, N) \in \mathrm{SO}(3)$\,. However, with this choice we would obtain a more complicated Sym-Bobenko formula (compare Equation~\eqref{eq:cmc:sym-bobenko-complicated} to Equation~\eqref{eq:cmc:sym-bobenko}) which is not well-suited to the blow-up construction we shall perform. Instead we shall rotate the tangential part of that extended frame by the phase of \,$\lambda=e^{\mi \theta} \in S^1$\,, obtaining an extended frame \,$(u,v, N) \in \mathrm{SO}(3)$\,, where
$$ u = \cos(\theta)\,\tfrac{f_x}{\|f_x\|} + \sin(\theta)\,\tfrac{f_y}{\|f_y\|} \quad\text{and}\quad v = -\sin(\theta)\,\tfrac{f_x}{\|f_x\|} + \cos(\theta)\,\tfrac{f_y}{\|f_y\|} \; . $$

For the purpose of describing this extended frame and the Sym-Bobenko formula more efficiently, we identify \,$\mathbb{R}^3$\, as an oriented Euclidean space with \,$\mathfrak{su}(2)$\, via
$$ \Phi: (x_1,x_2,x_3) \in \R^3 \; \longleftrightarrow \; \tfrac{\mi}{2} \left( \begin{smallmatrix} x_3 & x_1+\mi x_2 \\ x_1-\mi x_2 & -x_3 \end{smallmatrix} \right) \in \mathfrak{su}(2) \; . $$
The inner product on \,$\mathfrak{su}(2)$\, that corresponds to the usual inner product on \,$\R^3$\, is
$$ \langle X, Y \rangle := -2\, \mathrm{tr}(XY) \quad\text{for \,$X,Y \in \mathfrak{su}(2)$\,} \; . $$
The cross product \,$\times$\, of \,$\mathbb{R}^3$\, corresponds under this identification to the Lie bracket (commutator) of elements of \,$\mathfrak{su}(2)$\,. Note that \,$\mathrm{SU}(2)$\, is the universal covering of \,$\mathrm{SO}(3)$\, via the two-fold covering map 
$$ \Psi: \mathrm{SU}(2) \to \mathrm{SO}(3),\; F \mapsto \bigr(v \mapsto \Phi(F\cdot\Phi^{-1}(v)\cdot F^{-1}) \bigr) \; . $$

We lift the extended frame \,$(u,v,N)$\, by the covering map \,$\Psi$\, to obtain a \,$\lambda$-dependent \,$\mathrm{SU}(2)$-valued map \,$F_\lambda$\,, which we will also call the \emph{extended frame} of \,$f_\lambda$\,. Because the kernel of \,$\Psi$\, is \,$\{\pm \one \} \subset \mathrm{SU}(2)$\,, \,$F_\lambda$\, is determined up to sign by this condition. To fix the sign of \,$F_\lambda$\,, we suppose that the coordinate system of \,$\mathbb{R}^3$\, is chosen such that the extended frame \,$(u,v,N)$\, equals the standard basis \,$(e_1,e_2,e_3)$\, at some base point \,$z_0$\,, and then require that \,$F_\lambda(z_0) = \one$\, holds. In more explicit terms, we have 
\begin{equation}
\label{eq:cmc:frame-SU2}
u = F \tfrac{\mi}{2} \left( \begin{smallmatrix} 0 & 1 \\ 1 & 0 \end{smallmatrix} \right) F^{-1} \;,\quad 
v = F \tfrac{\mi}{2} \left( \begin{smallmatrix} 0 & \mi \\ -\mi & 0 \end{smallmatrix} \right) F^{-1}  
\quad\text{and}\quad 
N = F \tfrac{\mi}{2} \left( \begin{smallmatrix} 1 & 0 \\ 0 & -1 \end{smallmatrix} \right) F^{-1} \; .
\end{equation}
The original basis vectors \,$\tfrac{f_x}{\|f_x\|}$\, and \,$\tfrac{f_y}{\|f_y\|}$\, are obtained from \,$u$\, and \,$v$\, by rotation by the angle \,$-\theta$\,, and therefore it follows from the preceding equations that we also have 
\begin{equation}
\label{eq:cmc:frame-SU2-fxfy}
\tfrac{f_x}{\|f_x\|} = F \tfrac{\mi}{2} \left( \begin{smallmatrix} 0 & \lambda^{-1} \\ \lambda & 0 \end{smallmatrix} \right) F^{-1}
\quad\text{and}\quad
\tfrac{f_y}{\|f_y\|} = F \tfrac{\mi}{2} \left( \begin{smallmatrix} 0 & \mi \lambda^{-1} \\ -\mi \lambda & 0 \end{smallmatrix} \right) F^{-1} \;. 
\end{equation}

The family \,$F_\lambda$\, corresponds to the family of Maurer-Cartan forms \,$\alpha_\lambda = F_\lambda^{-1}\,\mathrm{d}F_\lambda$\,. By an explicit calculation, one can show that 
\begin{equation}
  \label{eq:cmc:alpha}
  \alpha = \alpha_\lambda = \frac14 \begin{pmatrix} \omega_z & 2H\,e^{\omega/2}\,\lambda^{-1} \\ -4Q\,e^{-\omega/2} & -\omega_z \end{pmatrix} \diff z
  + \frac14 \begin{pmatrix} -\omega_{\bar z} & 4\bar{Q}\,e^{-\omega/2} \\ -2H\,e^{\omega/2}\,\lambda & \omega_{\bar z} \end{pmatrix} \diff \bar{z}
\end{equation}
holds. Note that these local 1-forms defined with respect to different coordinates \,$z$\,
patch together to global \,$\mathfrak{sl}(2,\C)$-valued smooth 1-forms on \,$X$\,.

Conversely, if data \,$(\omega,H,Q)$\, with constants \,$H$\, and \,$Q$\, are given, then we can define a family of connection forms \,$\alpha_\lambda$\, by Equation~\eqref{eq:cmc:alpha}. Here \,$\alpha_\lambda$\, is well-defined not only for \,$\lambda\in S^1$\, but for all \,$\lambda\in \C^{\times} $\,. We call \,$\lambda \in \C^{\times} $\, the \emph{spectral parameter}. An explicit calculation shows that the Maurer-Cartan equation
\,$\diff \alpha_\lambda + \tfrac12 [\alpha_\lambda,\alpha_\lambda]=0$\, holds for all \,$\lambda\in \C^{\times} $\, if and only if the sinh-Gordon equation~\eqref{eq:cmc:gauss} holds. It follows that if the data \,$(\omega,H,Q)$\, correspond to a cmc immersion and hence satisfy the sinh-Gordon equation~\eqref{eq:cmc:gauss},
then the initial value problem of the partial differential equation
$$ \diff F_\lambda = F_\lambda\,\alpha_\lambda \quad\text{with}\quad F_\lambda(z_0) = \one $$
has a unique solution \,$F_\lambda: X \to \mathrm{SL}(2,\C)$\,, which we again call the \emph{extended frame}. For \,$\lambda\in S^1$\, it coincides with the extended frame of the immersions \,$f_\lambda$\, defined above. 
Because \,$\alpha_\lambda$\, depends holomorphically on the spectral parameter \,$\lambda \in \C^{\times} $\,, also \,$F_\lambda$\, depends holomorphically on \,$\lambda$\,. 
Note that because \,$\omega$\, is real-valued
and \,$H$\, is real, we have
\begin{equation}
  \label{eq:cmc:alpha-F-reality}
  \alpha_{\bar{\lambda}^{-1}} = -\overline{\alpha_\lambda}^t \quad\text{and therefore}\quad F_{\bar{\lambda}^{-1}} = \overline{F_\lambda}^{t,-1} \; .
\end{equation}
In particular for \,$|\lambda|=1$\,, \,$\alpha_\lambda$\, is an \,$\mathfrak{su}(2)$-valued 1-form, and \,$F_\lambda:X \to \mathrm{SL}(2,\C)$\, takes values in \,$\mathrm{SU}(2)$\,. 

{\bf The Sym-Bobenko formula.}
The following \emph{Sym-Bobenko formula}, see \cite[Section~5]{bobenko}, shows how the immersions \,$f_\lambda$\, can be reconstructed from the extended frame \,$F_\lambda$\,.

\begin{proposition}
  \label{P:cmc:sym-bobenko}
  Let a real-valued, smooth function \,$\omega$\, and constants \,$H \in \R^+$\,, \,$Q \in \C^{\times} $\, be given so that these data satisfy the sinh-Gordon equation~\eqref{eq:cmc:gauss}, and let \,$F_\lambda$\, be the corresponding extended frame. 
  Further choose a \emph{Sym point} \,$\lambda_s \in S^1$\,. Then
  \begin{equation}
  \label{eq:cmc:sym-bobenko}
  f_{\lambda_s} = -\frac{1}{H} \,\mi\,\lambda\,\frac{\partial F_\lambda}{\partial \lambda}\,F_\lambda^{-1} \biggr|_{\lambda=\lambda_s}
  \end{equation}
  is an immersion \,$X \to \mathfrak{su}(2) \cong \R^3$\, with induced metric \,$e^\omega\,\diff z \,\diff \bar{z}$\,, constant mean curvature \,$H$\, and
  Hopf differential \,$\lambda_s^{-1} Q\,\diff z^2$\,. 
  The tangential directions of \,$f_{\lambda_s}$\, are given by
  \begin{equation}
    \label{eq:cmc:sym-bobenko:tangential}
    f_{\lambda_s,x} = e^{\omega/2}\,F_{\lambda_s} \tfrac{\mi}{2} \left( \begin{smallmatrix} 0 & \lambda_s^{-1} \\ \lambda_s & 0 \end{smallmatrix} \right) F_{\lambda_s}^{-1}
    \quad\text{and}\quad
    f_{\lambda_s,y} = e^{\omega/2}\,F_{\lambda_s} \tfrac{\mi}{2} \left( \begin{smallmatrix} 0 & \mi \lambda_s^{-1} \\ -\mi \lambda_s & 0 \end{smallmatrix} \right) F_{\lambda_s}^{-1} \;, 
  \end{equation}
  and a unit normal field for \,$f_{\lambda_s}$\,  is given by
  \begin{equation}
    \label{eq:cmc:sym-bobenko:normal}
    N_{\lambda_s} = F_{\lambda_s} \tfrac{\mi}{2} \left( \begin{smallmatrix} 1 & 0 \\ 0 & -1 \end{smallmatrix} \right) F_{\lambda_s}^{-1} \; .
  \end{equation}
\end{proposition}

Due to the fundamental existence and uniqueness theorem for surfaces,
the \,$f_{\lambda_s}$\, defined by the Sym-Bobenko formula \eqref{eq:cmc:sym-bobenko} coincide up to a rigid motion of \,$\mathbb{R}^3$\, with the associated family \,$(f_\lambda)_{\lambda\in S^1}$\, , 
where \,$f_{\lambda_s}$\, corresponds to the data \,$(\omega,H,\lambda_s^{-1}Q)$\,. In particular for \,$\lambda_s=1$\, we recover the original immersion \,$f=f_{\lambda=1}$\,. 


\begin{proof}
	We now omit the subscript \,${}_{\lambda}$\, from \,$F_\lambda$\,, \,$\alpha_\lambda$\, and the associated quantities. 
  We first note that because \,$\alpha$\, is \,$\mathfrak{su}(2)$-valued on the circle \,$S^1 \ni \lambda$\,,  \,$F$\, takes values in \,$\mathrm{SU}(2)$\, there.
  Thus \,$G = \tfrac{\partial F}{\partial \lambda}$\, is tangential to \,$\mathrm{SU}(2)$\, at \,$\lambda=\lambda_s$\,, and therefore \,$f$\, indeed maps into
  \,$T_{\one} \mathrm{SU}(2) \cong \mathfrak{su}(2)$\,. 
  By differentiating the equation \,$\mathrm{d}F = F\alpha$\, with respect to \,$\lambda$\,, one sees that \,$G$\,
  solves the partial differential equation \,$\mathrm{d}G = G\alpha + F\beta$\, with
  \begin{equation}
    \label{eq:cmc:sym-bobenko:beta}
    \beta = \frac{\partial \alpha}{\partial \lambda} = -\frac12 H \, e^{\omega/2} \begin{pmatrix} 0 & \lambda^{-2}\,\diff z \\ \diff \bar{z} & 0 \end{pmatrix} \; . 
  \end{equation}
  We now calculate
  $$ \mathrm{d}f = -\frac{\mi \lambda}{H} \left(  \diff G\, F^{-1} - G\,F^{-1}\,\diff F\,F^{-1} \right) = -\frac{\mi \lambda}{H} \left(  (G\alpha+F\beta) F^{-1} - G\,F^{-1}\,F\alpha\,F^{-1} \right)
  = -\frac{\mi \lambda}{H} F\,\beta\,F^{-1}\; . $$
  By inserting Equation~\eqref{eq:cmc:sym-bobenko:beta} we obtain
  \begin{equation}
    \label{eq:cmc:sym-bobenko:fzfbarz}
    f_z = \lambda^{-1}\,e^{\omega/2}\, F \tfrac{\mi}{2} \left( \begin{smallmatrix} 0 & 1 \\ 0 & 0 \end{smallmatrix} \right) F^{-1}
    \quad\text{and}\quad
    f_{\bar{z}} = \lambda\,e^{\omega/2}\,  F \tfrac{\mi}{2} \left( \begin{smallmatrix} 0 & 0 \\ 1 & 0 \end{smallmatrix} \right) F^{-1} \; .
  \end{equation}
  Via the equations \,$f_x = f_z + f_{\bar z}$\, and \,$f_y = \mi(f_z-f_{\bar z})$\,, Equations~\eqref{eq:cmc:sym-bobenko:tangential} follow, and then
  Equation~\eqref{eq:cmc:sym-bobenko:normal} follows from \,$N= \tfrac{f_x \times f_y}{\|f_x \times f_y\|}$\,. We also calculate
  $$ 2\langle f_z,f_{\bar z} \rangle = 2e^{\omega} \cdot \left\langle \tfrac{\mi}{2} \left( \begin{smallmatrix} 0 & 1 \\ 0 & 0 \end{smallmatrix} \right),
  \tfrac{\mi}{2} \left( \begin{smallmatrix} 0 & 0 \\ 1 & 0 \end{smallmatrix} \right) \right\rangle = e^\omega \;, $$
  whence it follows that \,$f$\, is a conformal immersion with the induced metric \,$e^\omega\,\diff z \,\diff \bar{z}$\,. 

  We moreover obtain from Equations~\eqref{eq:cmc:sym-bobenko:fzfbarz} and \eqref{eq:cmc:alpha}, where we write \,$\alpha_\lambda = U_\lambda\,\mathrm{d}z + V_\lambda\,\mathrm{d}\bar{z}$\,:
  \begin{align*}
    f_{zz} & = \tfrac12 \omega_z\,f_z + [FUF^{-1}, f_z] = \omega_z\,f_z + \lambda^{-1}\,Q\,N \\
    f_{z\bar{z}} & = \tfrac12 \omega_{\bar{z}}\,f_z + [FVF^{-1}, f_z] = \tfrac12 \,e^{\omega}\,H\,N \\
    f_{\bar{z}\bar{z}} & = \tfrac12 \omega_{\bar{z}}\,f_{\bar{z}} + [FVF^{-1}, f_{\bar{z}}] = \omega_{\bar{z}}\,f_{\bar{z}} + \lambda\,\bar{Q}\,N \; .
  \end{align*}
  Thus the Hopf differential of \,$f$\, is given by \,$\langle f_{zz},N \rangle \diff z^2 = \lambda^{-1}\,Q\,\diff z^2$\, and the mean curvature of \,$f$\, is
  $$ \langle f_z, f_{\bar z} \rangle^{-1} \cdot \langle f_{z\bar{z}},N \rangle = 2e^{-\omega} \cdot \tfrac12\,e^{\omega}\,H = H \; . $$
\end{proof}

Note that the usual extended frame \,$( \tfrac{f_x}{\|f_x\|}, \tfrac{f_y}{\|f_y\|}, N) \in \mathrm{SO}(3)$\, we mentioned above corresponds to the \,$\mathrm{SU}(2)$-valued map \,$F_{\textrm{usual},\lambda}$\, characterised by
\begin{equation}
	\label{eq:cmc:usual-frame-SU2}
	\tfrac{f_x}{\|f_x\|} = F_{\textrm{usual}} \tfrac{\mi}{2} \left( \begin{smallmatrix} 0 & 1 \\ 1 & 0 \end{smallmatrix} \right) F_{\textrm{usual}}^{-1} \;,\quad 
	\tfrac{f_y}{\|f_y\|} = F_{\textrm{usual}} \tfrac{\mi}{2} \left( \begin{smallmatrix} 0 & \mi \\ -\mi & 0 \end{smallmatrix} \right) F_{\textrm{usual}}^{-1}  
	\quad\text{and}\quad 
	N = F_{\textrm{usual}} \tfrac{\mi}{2} \left( \begin{smallmatrix} 1 & 0 \\ 0 & -1 \end{smallmatrix} \right) F_{\textrm{usual}}^{-1} \; .
\end{equation}
By comparing these equations with Equations~\eqref{eq:cmc:frame-SU2-fxfy} we see that our extended frame \,$F_\lambda$\, is related to \,$F_{\textrm{usual},\lambda}$\, by \,$F_\lambda = F_{\textrm{usual},\lambda}\cdot A^{-1}$\, where \,$A=\left( \begin{smallmatrix} \lambda^{-1/2} & 0 \\ 0 & \lambda^{1/2} \end{smallmatrix} \right)\in \mathrm{SU}(2)$\,. The extended frame \,$F_{\textrm{usual}}$\, gives rise to the connection form
\begin{equation}
\label{eq:cmc:alpha-usual}
\alpha_{\textrm{usual}} = F_{\textrm{usual}}^{-1}\,\mathrm{d}F_{\textrm{usual}} =  \frac14 \begin{pmatrix} -\omega_z & 4Q\,e^{-\omega/2}\,\lambda^{-1} \\ -2H\,e^{\omega/2} & \omega_z \end{pmatrix} \diff z
+ \frac14 \begin{pmatrix} \omega_{\bar z} & 2H\,e^{\omega/2} \\ -4\bar{Q}\,e^{-\omega/2}\,\lambda & -\omega_{\bar z} \end{pmatrix} \diff \bar{z}
\end{equation}
and to the alternate Sym-Bobenko formula also described in \cite[Section~5.1, Equation~(5.4)]{bobenko}
\begin{equation}
	\label{eq:cmc:sym-bobenko-usual}
	f_\lambda 
	= -\frac{1}{H}\left( F_{\textrm{usual}} \frac{\mi}{2}\begin{pmatrix} 1 & 0 \\ 0 & -1 \end{pmatrix} F_{\textrm{usual}}^{-1} + \mi \,\lambda\,\frac{\partial F_{\textrm{usual}}}{\partial \lambda}\,F_{\textrm{usual}}^{-1} \right) \; .
\end{equation}
In this paper we are interested in constructing convergent blow-ups of sequences of cmc immersions, which we describe via the Sym-Bobenko formula. It will turn out that the rates of blow-up of the two summands in the Sym-Bobenko formula~\eqref{eq:cmc:sym-bobenko-usual} are different. Therefore if we were to base our blow-up on \eqref{eq:cmc:sym-bobenko-usual}, to obtain convergence to a immersion we would need to apply a parallel translation to \,$f_\lambda$\, to counteract the effect of the additional term \,$F_{\textrm{usual}} \frac{\mi}{2}\left(\begin{smallmatrix} 1 & 0 \\ 0 & -1 \end{smallmatrix}\right) F_{\textrm{usual}}^{-1}$\,, where the length of the translation is unbounded. In the following remark we explain how this parallel translation can be interpreted in terms of the transition to the parallel cmc surface.  

\begin{remark}
  It is a classically well-known fact that for any constant mean curvature \,$(H\neq 0)$-surface, one of the two parallel surfaces in the distance \,$1/H$\,  is again a constant mean curvature surface. If the original cmc surface has the data \,$(\omega,H,Q)$\,, then the parallel cmc surface has the data \,$(\check{\omega}, \check{H}, \check{Q}) = (-\omega, 2|Q|, \tfrac{HQ}{2|Q|})$\,.

  In the situation described above, this fact is reflected in the following way:

	\begin{enumerate}
	\item 
	The connection form associated to the parallel cmc surface by Equation~\eqref{eq:cmc:alpha} is given by \,$\check{\alpha}_\lambda = -(g.\alpha_\lambda)^t$\,, where \,$g.\alpha_\lambda$\, denotes the gauge transformation of the connection form \,$\alpha_\lambda$\, associated to the original cmc surface with
	\begin{equation}
		\label{eq:cmc:dual-gauge-g}
		g = \begin{pmatrix} (\tfrac{Q}{|Q|}\lambda)^{-1/2} & 0 \\ 0 & (\tfrac{Q}{|Q|}\lambda)^{1/2} \end{pmatrix} \; .
	\end{equation}
	The extended frame of the parallel cmc surface is given by 
	\,$\check{F} = F^{t,-1}\,g^{-1}$\,, and the Sym-Bobenko formula \eqref{eq:cmc:sym-bobenko} for the parallel cmc surface is 
	  \begin{equation}
		\label{eq:cmc:sym-bobenko-complicated}
		\check{f} 
		= \frac{1}{\check{H}}\left( F \frac{\mi}{2}\begin{pmatrix} 1 & 0 \\ 0 & -1 \end{pmatrix} F^{-1} + \mi \,\lambda\,\frac{\partial F}{\partial \lambda}\,F^{-1} \right)^t \; .
	\end{equation}
	  By comparison with Proposition~\ref{P:cmc:sym-bobenko} we see that \,$-\tfrac{\check{H}}{H}\,\check{f}^t = f-\tfrac{1}{H}N$\, holds, and hence up to the
	rotation-reflection \,$X\mapsto -X^t$\, of \,$\R^3$\, and the scaling factor \,$\tfrac{\check{H}}{H} = 2\tfrac{|Q|}{H}$\,, \,$\check{f}$\, is the
	parallel surface of the cmc surface described by \,$f$\, that is also a cmc surface.
	
	\begin{proof}
		By substituting \,$(\omega, H, Q) \mapsto (\check{\omega}, \check{H}, \check{Q}) = (-\omega, 2|Q|, \tfrac{HQ}{2|Q|})$\, in Equation~\eqref{eq:cmc:alpha} one obtains for the connection form \,$\check{\alpha}_\lambda$\, associated to the parallel cmc surface
		\begin{equation}
		\label{eq:cmc:check-alpha}
		\check{\alpha}_\lambda = \frac14 \begin{pmatrix} -\omega_z & 4|Q|\,e^{-\omega/2}\,\lambda^{-1} \\ -2H\tfrac{Q}{|Q|}\,e^{\omega/2} & \omega_z \end{pmatrix} \diff z
		+ \frac14 \begin{pmatrix} \omega_{\bar z} & 2H\tfrac{\bar{Q}}{|Q|}\,e^{\omega/2} \\ -4|Q|\,e^{-\omega/2}\,\lambda & -\omega_{\bar z} \end{pmatrix} \diff \bar{z} \;.
		\end{equation}
		
		For general \,$g$\, that can vary in both \,$z$\, and \,$\lambda$\,, the \emph{gauge transformation} by \,$g$\, is defined by
		$$ g.\alpha = g^{-1}\,\alpha\,g + g^{-1}\diff g \quad\text{and}\quad g.F = F\,g \; , $$
		so that the partial differential equation \,$\diff (g.F) = (g.F)(g.\alpha)$\, is maintained. The specific \,$g$\, given by Equation~\eqref{eq:cmc:dual-gauge-g} does not depend on \,$z$\,, and therefore we have \,$g.\alpha = g^{-1}\,\alpha\,g$\,, and by computing the right-hand side expression we see that \,$-(g.\alpha_\lambda)^t = \check{\alpha}_\lambda$\, holds. This equation implies \,$\check{F}=(g.F)^{t,-1} = F^{t,-1}\,g^{-1}$\, because \,$g$\, is diagonal. By substituting \,$F \mapsto \check{F}$\, in Equation~\eqref{eq:cmc:sym-bobenko}, we see that the Sym-Bobenko formula for the parallel cmc surface is indeed given by Equation~\eqref{eq:cmc:sym-bobenko-complicated}. 
	\end{proof}	
	\item
		Comparing Equation~\eqref{eq:cmc:check-alpha} to Equation~\eqref{eq:cmc:alpha-usual}, we see that if we rotate the coordinate \,$z$\, such that \,$Q>0$\, holds, then \,$\check{\alpha}_\lambda = \alpha_{\mathrm{usual},\lambda}$\, holds. 
		In this sense the transition from the ``usual'' extended frame \,$F_{\mathrm{usual}}$\, to the rotated extended frame \,$F$\, corresponds to the transition from a cmc surface to its parallel one.
	\end{enumerate}
\end{remark}

{\bf Spectral data.}
We now describe spectral data for this integrable system in the case where \,$X=\mathbb{C}$\, and where the solution \,$\omega$\, of the sinh-Gordon equation~\eqref{eq:cmc:gauss} is (at least) simply-periodic,
i.e.~there exists a (minimal) period \,$T\in \C^\times$\, such that \,$\omega(z+T)=\omega(z)$\, holds for all \,$z$\,. In this situation, the corresponding connection form \,$\alpha$\, is likewise periodic, but in general the extended frame \,$F$\, does not need to be. Its departure from being periodic is measured by the \emph{monodromy (with base point \,$z_0$\,)} \,$M_{z_0}(\lambda)
= F_\lambda(z_0)^{-1} \cdot F_\lambda(z_0+T)$\,. The dependence of the monodromy on the base point \,$z_0$\, is described by the differential equation
\,$\diff M = [M, \alpha]$\,, and consequently we have
\begin{equation}
\label{eq:cmc:monodromy-basepoint}
M_{z_1}(\lambda) = F_\lambda(z_1)^{-1}\,M_{z_0}(\lambda)\,F_\lambda(z_1) \quad\text{for any other base point \,$z_1$\,.}
\end{equation}
It follows that the eigenvalues of \,$M_{z_0}(\lambda)$\, and the holomorphic function \,$\Delta(\lambda) = \tr M_{z_0}(\lambda)$\,
do not depend on the choice of the base point \,$z_0$\,.
Note that this trace function \,$\Delta(\lambda)$\, (unlike the trace function of a polynomial Killing field as described below) is a non-polynomial analytic function in \,$\lambda$\,, and that the corresponding discriminant function \,$\tfrac12(\Delta(\lambda)^2 -4)$\, always has infinitely many zeros, which accumulate near \,$\lambda=0$\, and \,$\lambda=\infty$\,, see \cite[Proposition~6.5(1)]{habil}.  We collect the eigenvalues \,$\mu$\, of \,$M_{z_0}(\lambda)$\, in the \emph{multiplier curve}  
$$ \underline{\Sigma} = \bigr\{ (\lambda,\mu) \in \C^{\times}  \times \C^{\times}  \,\bigr|\, \mu^2-\Delta(\lambda)\,\mu + 1 = 0 \bigr\} \; . $$
The multiplier curve is a possibly singular complex curve embedded in \,$\C^{\times}  \times \C^{\times} $\, with infinite arithmetic genus. 
It is hyperelliptic over
\,$\C^{\times} $\, in the sense that the holomorphic map \,$\underline{\Sigma} \to \C^{\times} , \; (\lambda,\mu) \mapsto \lambda$\, is a branched, two-fold covering map,
and the holomorphic involution \,$\underline{\sigma}: \underline{\Sigma} \to \underline{\Sigma},\; (\lambda,\mu) \mapsto (\lambda,\mu^{-1})$\, interchanges the two
sheets of this covering map. The reality condition \eqref{eq:cmc:alpha-F-reality} implies
\begin{equation}
  \label{eq:cmc:M-reality}
  M_{z_0}(\bar{\lambda}^{-1}) = \overline{M_{z_0}(\lambda)}^{t,-1} \quad\text{for \,$\lambda\in\C^{\times} $\,}
\end{equation}
and therefore
\,$\underline{\Sigma}$\, also has an anti-holomorphic involution \,$\underline{\rho}: \underline{\Sigma} \to \underline{\Sigma}, \; (\lambda,\mu) \mapsto (\bar{\lambda}^{-1},
\bar{\mu}^{-1})$\, which commutes with \,$\underline{\sigma}$\,.

For a fixed base point \,$z_0$\,, the eigenvectors of \,$M_{z_0}(\lambda)$\, define a holomorphic line bundle \,$\underline{\Lambda}_{z_0}$\,
on the complex curve \,$\underline{\Sigma}$\,. If we write \,$M_{z_0}(\lambda) = \left( \begin{smallmatrix} a(\lambda) & b(\lambda) \\ c(\lambda) & d(\lambda) \end{smallmatrix}
\right)$\, with the holomorphic functions \,$a,b,c,d: \C^{\times}  \to \C$\,, then eigenvectors \,$(v_1,v_2)^t \in \C^2$\, of \,$M_{z_0} (\lambda)$\, corresponding to the eigenvalue \,$\mu$\,
are characterised by either of the two equivalent equations
\begin{align*}
  (a(\lambda)-\mu)v_1 + b(\lambda)v_2 & = 0 \;, \\
  c(\lambda)v_1 + (d(\lambda)-\mu)v_2 & = 0 \;.
\end{align*}
It follows that \,$(1,\tfrac{\mu-a(\lambda)}{b(\lambda)})$\, and \,$(\tfrac{\mu-d(\lambda)}{c(\lambda)},1)$\, are meromorphic sections of \,$\underline{\Lambda}_{z_0}$\,. 

We shall follow the approach of Hitchin \cite{Hi} and define the open eigenline curve \, $\Sigma ^\circ $\, of the monodromy \,$M_{z_0}(\lambda)$\, as the 2-sheeted covering of  \,$\lambda\in \mathbb C ^\times $\, which is ramified to order \,$ n $\, at \,$\lambda $\, precisely when the eigenlines of the \,$(2\times 2) $-matrix \,$M_{z_0} (\lambda) $\, agree to order \,$n$\,. To measure this order of agreement of the eigenlines, as in \cite{Hi} we first consider the smooth open hyperelliptic curve which is the 2-sheeted covering of \,$\lambda\in\C ^\times$\, with simple branch points at the odd order roots of the trace\, $\Delta (\lambda) $\, of the monodromy. These roots are naturally considered as the branch points of the eigenvalue function \,$\mu $\, of the monodromy. For each point \,$ z_0\in\C $\,  this hyperelliptic curve supports a line bundle \,$ \Lambda_{z_0} $\, defined by the eigenlines of the monodromy with base point \,$ z_0$\,. Denoting the hyperelliptic involution by \,$\sigma $\,, then as detailed in \cite{Hi}, the fact that for \,$\lambda\in S ^ 1 $\, the monodromy is valued in \,$\mathrm{SU}(2) $\, gives rise to a symplectic form that defines a section  of \,$\Lambda_{z_0} ^*\otimes\sigma ^*(\Lambda_{z_0} ^*) $\,  over the hyperelliptic curve and we define the order of agreement of the eigenlines as the order of vanishing of this section. The open eigenline curve \,$\Sigma ^\circ $\,  has the hyperelliptic involution \,$\sigma: \Sigma^\circ \to \Sigma^\circ$\,
that interchanges the two sheets of the covering \,$\lambda$\,. The reality condition \eqref {eq:cmc:M-reality} ensures that \,$\Sigma ^\circ $\, also has an anti-holomorphic involution \,$\rho: \Sigma^\circ \to \Sigma^\circ$\, with \,$\sigma\circ \rho = \rho \circ \sigma$\, and \,$\lambda\circ \rho = \bar{\lambda}^{-1}$\,.

Hitchin considered the case where the solution \,$\omega$\, of the sinh-Gordon equation is doubly-periodic and showed \cite[Proposition 2.3]{Hi}  that for such solutions there are only finitely many \,$\lambda\in\C ^\times$\, at which the trace\, $\Delta (\lambda) $\, of the monodromy has an odd-order root. 

Not all simply-periodic \,$\omega$\, have trace functions \,$\Delta(\lambda)$\, with only finitely many odd-order roots, 
but we henceforth restrict our attention to those which do. 
An investigation of the asymptotic behaviour of the monodromy \,$M_{z_0}(\lambda)$\, near \,$\lambda=0$\, and \,$\lambda=\infty$\,,
see \cite[Section~3]{Hi} or \cite[Sections 4, 5]{habil},
then shows that the open eigenline curve has finite arithmetic genus  (\,$\dim H^1(\Sigma^\circ,\mathcal{O}) <  \infty$\,): it can be compactified at \,$\lambda=0$\, and \,$\lambda=\infty$\, to give a compact complex curve \,$\Sigma$\, called the  \emph {eigenline curve} or  \emph {spectral curve} of \,$f$\, or of \,$\omega$\,. Thus the \,$\omega$\, we consider are precisely those for which the eigenline curve has finite genus. 
Such  \,$\omega $\, and the corresponding cmc immersions \,$f $\, are said to be of \emph {finite type}. 
Furthermore, the asymptotics cited above show that the compactification \,$\Sigma$\, adds only a single smooth point at each of \,$\lambda=0$\, and \,$\lambda=\infty$\, to \,$\Sigma^\circ$\,. These
two added points are then regular branch points of \,$\Sigma$\,. It follows that \,$\Sigma$\, can be realised as a sub-variety in \,$\mathbb{P}^1 \times \mathbb{P}^1$\, as
\begin{equation}
  \label{eq:cmc:Sigma}
  \Sigma = \bigr\{ (\lambda,\nu)\in \mathbb{P}^1 \times \mathbb{P}^1 \,\bigr|\, \nu^2 = \lambda\,a(\lambda) \bigr\} \;,
\end{equation}
where \,$a(\lambda)$\, is a polynomial in \,$\lambda$\, of even degree \,$2g$\, and \,$g=g(\Sigma)$\, is the arithmetic genus of \,$\Sigma$\,.
We have \,$a(0)\neq 0$\, and we normalise \,$a(\lambda)$\, such that \,$a(0)=-\tfrac{1}{2}HQ$\, holds. We choose this normalisation so that the lowest coefficients of the polynomial Killing field are equal to certain entries of the connection form \eqref{eq:cmc:alpha}, see the calculation leading to Equation~\eqref{eq:cmc:xi-vm1w0} below. 
The reality condition on \,$\Sigma$\, implies
\begin{equation}
  \label{eq:cmc:a-reality}
  a(\lambda) = \lambda^{2g} \cdot \overline{a(\bar{\lambda}^{-1})} \; .
\end{equation}
The involutions \,$\sigma$\, and \,$\rho$\, are given by
$$ \sigma: \Sigma \to \Sigma, \; (\lambda,\nu) \mapsto (\lambda,-\nu) \AND \rho: \Sigma \to \Sigma, \;  (\lambda,\nu) \mapsto (\bar{\lambda}^{-1},  \bar{\lambda}^{-(g+1)}\bar{\nu}) \; . $$

We remark that the open eigenline curve \,$\Sigma ^\circ $\, coincides with the
 \emph{\,$\underline{\Lambda}_{z_0}$-halfway desingularisation} of \,$\underline{\Sigma}$\,, introduced in \cite[Section~4]{KLSS} as the maximal one-sheeted, branched covering of \,$\underline{\Sigma}$\, to which the generalised divisor
corresponding to the line bundle \,$\underline{\Lambda}_{z_0}$\, can be lifted. Hence an alternative approach is to introduce the open eigenline curve \,$\Sigma ^\circ $\, directly as the $\underline{\Lambda}_{z_0}$-halfway desingularisation of \,$\underline{\Sigma}$\,. 
Finite type \,$\omega $\, and \,$f $ are then those for which this curve has finite arithmetic genus. 

{\bf Polynomial Killing fields.}
It was shown in \cite[Section~4]{KLSS} that for every holomorphic function \,$\varphi$\, on \,$\Sigma^\circ$\, there exists one and only one holomorphic
\,$(2\times 2)$-matrix valued function \,$N(\lambda)$\, in \,$\lambda \in \mathbb{C}^\times$\, such that \,$\Lambda_{z_0}$\, is the eigenline bundle of \,$N$\, and  \,$\varphi$\, is the
corresponding eigenfunction, meaning that \,$N(\lambda)s = \varphi(\lambda) s$\, holds for every holomorphic section \,$s$\, of \,$\Lambda_{z_0}$\, over \,$\Sigma^\circ$\,.
If we apply this statement to
the original eigenfunction \,$\mu$\, on the multiplier curve \,$\underline{\Sigma}$\,, we recover the original monodromy \,$M_{z_0}$\, which gave rise to \,$\underline{\Sigma}$\,.
But we now apply the statement to the anti-symmetric holomorphic function \,$\varphi = \tfrac{\nu}{\lambda}$\, on \,$\Sigma$\,. The resulting \,$(2\times 2)$-matrix-valued
holomorphic function \,$\xi_{z_0} = \xi_{z_0}(\lambda)$\,, which satisfies
\begin{equation}
\label{eq:cmc:xi-eigenfunction}
\xi_{z_0}(\lambda)s = \tfrac{\nu}{\lambda}s \quad\text{for every holomorphic section \,$s$\, of
\,$\Lambda_{z_0}$\, over \,$\Sigma^\circ$\,,} 
\end{equation}
is called \emph{the polynomial Killing field for \,$\omega$\,} (at the base point \,$z_0$\,). Then \,$\lambda\,\xi_{z_0}(\lambda)s = \nu s$\, is holomorphic, and this equation shows that the polynomial Killing field \,$\xi_{z_0}$\, encodes the information of the eigenvalues \,$\nu$\, and therefore of the spectral curve \,$\Sigma$\, and also of the eigenline bundle given by the holomorphic section \,$s$\,. 

Because the eigenfunction \,$\tfrac{\nu}{\lambda}$\, is
anti-symmetric with respect to the hyperelliptic involution of \,$\Sigma$\,, \,$\xi_{z_0}(\lambda)$\, is trace-free, i.e.~\,$\xi_{z_0}$\, maps into \,$\mathfrak{sl}(2,\mathbb{C})$\,.
Because \,$\Sigma$\, is a compact complex curve (after the compactification described above), \,$\xi_{z_0}$\, is a polynomial in \,$\lambda$\, and \,$\lambda^{-1}$\,, 
and the equation \,$\det(\xi_{z_0}) = \tfrac{\nu}{\lambda}\cdot \tfrac{-\nu}{\lambda} = -\tfrac{1}{\lambda}a(\lambda)$\, when considered near \,$\lambda=0$\, shows that
the lowest power of \,$\lambda$\, that occurs in \,$\xi_{z_0}$\, is \,$\lambda^{-1}$\,. A more precise investigation of the asymptotic behaviour of \,$M_{z_0}(\lambda)$\,
near \,$\lambda=0$\, in fact shows \,$\Res_{\lambda=0}(\xi_{z_0} \diff \lambda) = \tfrac12\,H\,e^{\omega(z_0)/2} \left( \begin{smallmatrix} 0 & 1 \\ 0 & 0 \end{smallmatrix} \right)$\,. Additionally the reality conditions imply the following reality
condition for \,$\xi_{z_0}$\,:
\begin{equation}
  \label{eq:cmc:xi-reality}
  \xi_{z_0}(\lambda) = -\lambda^{g-1}\,\overline{\xi_{z_0}(\bar{\lambda}^{-1})}^t \; .
\end{equation}
Therefore the highest power of \,$\lambda$\, that occurs in \,$\xi_{z_0}$\, is \,$\lambda^g$\,. To summarise, \,$\xi_{z_0}$\, is of the form
$$ \xi_{z_0}(\lambda) = \sum_{k=-1}^g \xi_{z_0,k}\,\lambda^k \quad\text{with}\quad \xi_{z_0,k} = \begin{pmatrix} u_k & v_k \\ w_k & -u_k \end{pmatrix} \in \mathfrak{sl}(2,\mathbb{C}) \quad\text{for}\quad k\in\{-1,\dotsc,g\} \; , $$
where
$$ \xi_{z_0,-1} = \begin{pmatrix} 0 & \tfrac12\,H\,e^{\omega(z_0)/2} \\ 0 & 0 \end{pmatrix}\quad\text{and}\quad \xi_{z_0,g-(k+1)} = -\overline{\xi_{z_0,k}}^t \quad\text{for}\quad k\in\{-1,\dotsc,g\} \; . $$
Note that \,$\lambda\,\xi_{z_0}(\lambda)$\, is a polynomial in \,$\lambda$\, and \,$\lambda\,\xi_{z_0}\,s = \nu\,s$\, holds for any holomorphic section of \,$\Lambda_{z_0}$\,,
hence we have
\begin{equation}
\label{eq:cmc:det-xi}
-\lambda\,\det(\xi_{z_0}) = -\lambda^{-1}\,\det(\lambda\,\xi_{z_0}) = -\lambda^{-1}\,\nu\,(-\nu) = a(\lambda)
\end{equation}
and thus for \,$\lambda=0$\,: \,$v_{-1}\cdot w_0 = a(0) = -\tfrac12\,HQ$\,. We thus obtain
\begin{equation}
  \label{eq:cmc:xi-vm1w0}
  v_{-1} = \tfrac12\,H\,e^{\omega(z_0)/2} \quad\text{and}\quad w_0 = -Q\,e^{-\omega(z_0)/2} \; .
\end{equation}
Concerning the dependence of the polynomial Killing field \,$\xi_{z_0}$\, on the base point \,$z_0$\,, we note that for a different base point \,$z_1$\, we
have \,$\Lambda_{z_1} = F(z_1)^{-1}\,\Lambda_{z_0}$\, due to Equation~\eqref{eq:cmc:monodromy-basepoint}, and therefore
\begin{equation}
  \label{eq:cmc:xi-basepointchange}
  \xi_{z_1} = F(z_1)^{-1}\,\xi_{z_0}\,F(z_1) \; .
\end{equation}
Hence concerning differentiation with respect to the base point \,$z$\,, the family \,$\xi=\xi_z$\, of polynomial Killing fields fulfils the differential equation
\begin{equation}
  \label{eq:cmc:xi-dgl}
  \diff \xi + [\alpha_\lambda, \xi] = 0 \; .
\end{equation}
By decomposing this differential equation with respect to powers of \,$\lambda$\, and entries of the \,$(2\times 2)$-matrices, one sees that the polynomial
Killing field \,$\xi$\, can be reconstructed from the ``initial condition'' \,$v_{-1} = \tfrac12\,H\,e^{\omega/2}$\, by an iterative process. This process was
introduced by Pinkall/Sterling in \cite{PS:89} in the course of their proof that cmc tori are of finite type, and is now called the \emph{Pinkall-Sterling iteration}:

\begin{proposition}
  \label{P:cmc:ps}
  Let a solution \,$(\omega,H,Q)$\, of the sinh-Gordon equation~\eqref{eq:cmc:gauss} with a smooth real-valued function \,$\omega$\, and constants \,$H>0$\,, \,$Q\in \C^{\times} $\,
  be given. We suppose that this solution is of finite type \,$g$\, and write the corresponding polynomial Killing field \,$\xi=\xi_z$\, in the form
  $$ \xi = \sum_{k=-1}^{g} \xi_k\,\lambda^k \quad\text{with}\quad \xi_k = \begin{pmatrix} u_k & \tau_k\,e^{\omega/2} \\ \sigma_k\,e^{\omega/2} & -u_k \end{pmatrix} \;, $$
  where \,$u_k, \tau_k, \sigma_k$\, are smooth, complex-valued functions in \,$z$\,. Then we have \,$\tau_{-1} = \tfrac12 H$\,, \,$u_{-1} = \sigma_{-1} = 0$\,
  and for every \,$k\geq 0$\,:
  \begin{align}
    \label{eq:cmc:ps:ps-tau}
    \tau_{k,z} = -\tfrac{1}{2Q} (u_{k,zz} - \omega_{z} u_{k,z}) \qquad & \qquad \tau_{k,\bar{z}} = 2\bar{Q}\,e^{-\omega}\,u_k \\
    \label{eq:cmc:ps:ps-u}  
    u_{k+1} & = \tfrac{1}{H}(\tau_{k,z}+\omega_{z}\,\tau_k) \\
    \label{eq:cmc:ps:ps-sigma}  
    \sigma_{k+1} & = -\tfrac{1}{\bar{Q}} (u_{k+1,\bar{z}} + \tfrac12 H \,e^{\omega}\,\tau_k) \; . 
  \end{align}
  Moreover every \,$u_k$\, solves the linearisation of the sinh-Gordon equation~\eqref{eq:cmc:gauss}:
  \begin{equation}
    \label{eq:cmc:ps:uk-lin-gauss}
    2u_{k,z\bar{z}} + \left( H^2 e^{\omega} -4|Q|^2 e^{-\omega} \right) u_k = 0 \; .
  \end{equation}
\end{proposition}

\begin{proof}
  We write
  $$ \alpha_\lambda = (U_{-1}\,\lambda^{-1} + U_0)\mathrm{d}z + (V_0 + V_1\,\lambda)\mathrm{d}\bar{z} $$
  with
  \begin{align*}
    U_{-1} & = \frac12 \begin{pmatrix} 0 & He^{\omega/2} \\ 0 & 0 \end{pmatrix}\;, & 
    U_0 & = \frac14 \begin{pmatrix} \omega_{z} & 0 \\ -4Qe^{-\omega/2} & -\omega_z \end{pmatrix}\;, \\
    V_0 & = \frac14 \begin{pmatrix} -\omega_{\bar{z}} & 4\bar{Q}e^{-\omega/2} \\ 0 & \omega_{\bar{z}} \end{pmatrix} \quad\text{and} & 
    V_1 & = \frac12 \begin{pmatrix} 0 & 0 \\ -H e^{\omega/2} & 0 \end{pmatrix} \; .
  \end{align*}
  
We separate the differential equation \eqref{eq:cmc:xi-dgl} into its \,$\mathrm{d}z$-part and its \,$\mathrm{d}\bar{z}$-part,
and also into the individual powers of \,$\lambda$\, that occur.
In this way we obtain the equations
\begin{align*}
  \xi_{k,z} + [U_0,\xi_{k}] + [U_{-1},\xi_{k+1}] & = 0 \\
  \xi_{k,\bar{z}} + [V_1,\xi_{k-1}] + [V_0,\xi_k] & = 0
\end{align*}
for all \,$k\in\{-1,\dotsc,g\}$\,. By evaluating the brackets and separating the entries of the matrices, we obtain the following equations:
\begin{align}
  \label{eq:cmc:ps:ps1-ukz}
  u_{k,z} + Q \tau_k + \tfrac12 He^{\omega}\sigma_{k+1} & = 0 \\
  \label{eq:cmc:ps:ps1-ukbz}  
  u_{k,\bar{z}} + \tfrac12 He^{\omega}\tau_{k-1} + \bar{Q} \sigma_k & = 0 \\
  \label{eq:cmc:ps:ps1-taukz}
  \tau_{k,z} + \omega_z\,\tau_k - Hu_{k+1} & = 0 \\
  \label{eq:cmc:ps:ps1-taukbz}
  e^{\omega/2}\tau_{k,\bar{z}} - 2\bar{Q}e^{-\omega/2}u_k & = 0 \\
  \label{eq:cmc:ps:ps1-sigmakz}
  e^{\omega/2}\sigma_{k,z} - 2Qe^{-\omega/2}u_k & = 0 \\
  \label{eq:cmc:ps:ps1-sigmakbz}
  \sigma_{k,\bar{z}}+\omega_{\bar{z}} \sigma_k - Hu_{k-1} & = 0 \; .
\end{align}
The right-hand equation of \eqref{eq:cmc:ps:ps-tau} follows from Equation~\eqref{eq:cmc:ps:ps1-taukbz}, Equation~\eqref{eq:cmc:ps:ps-u} follows from
Equation~\eqref{eq:cmc:ps:ps1-taukz}, and Equation~\eqref{eq:cmc:ps:ps-sigma} follows from Equation~\eqref{eq:cmc:ps:ps1-ukbz}. 

From Equation~\eqref{eq:cmc:ps:ps1-ukz} for \,$k-1$\, and Equation~\eqref{eq:cmc:ps:ps1-ukbz} for \,$k$\, we obtain
\begin{equation}
  \label{eq:cmc:ps:ps1-sigmak-two}
  -\tfrac{2}{H}e^{-\omega} (u_{k-1,z} + Q \tau_{k-1}) = \sigma_k = -\tfrac{1}{\bar{Q}}(u_{k,\bar{z}} + \tfrac12He^{\omega}\tau_{k-1})
\end{equation}
and therefore
$$ (\tfrac{H}{2\bar{Q}} e^\omega - \tfrac{2Q}{H}e^{-\omega}) \tau_{k-1} - \tfrac{2}{H}e^{-\omega}u_{k-1,z} + \tfrac{1}{\bar{Q}}u_{k,\bar{z}} = 0 \; . $$
By substituting \,$u_{k-1}$\, using Equation~\eqref{eq:cmc:ps:ps1-taukbz} and \,$u_k$\, using Equation~\eqref{eq:cmc:ps:ps1-taukz} we get
$$ (\tfrac{H}{2\bar{Q}} e^\omega - \tfrac{2Q}{H}e^{-\omega}) \tau_{k-1} - \tfrac{1}{H\bar{Q}}e^{-\omega}\bigr( e^{\omega} \tau_{k-1,\bar{z}} \bigr)_z + \tfrac{1}{H\bar{Q}}\bigr( \tau_{k-1,z} + \omega_z\,\tau_{k-1} \bigr)_{\bar{z}} = 0 \; , $$
thus
$$ (\tfrac{1}{HQ}\omega_{z\bar{z}} + \tfrac{H}{2\bar{Q}} e^\omega - \tfrac{2Q}{H}e^{-\omega}) \tau_{k-1} = 0 $$
and hence again the Maurer-Cartan equation~\eqref{eq:cmc:gauss}. 

By differentiating the right hand half of \eqref{eq:cmc:ps:ps1-sigmak-two} by \,$z$\, and applying Equation~\eqref{eq:cmc:ps:ps1-taukz}  we moreover obtain
$$ \sigma_{k,z} = -\tfrac{1}{\bar{Q}}\bigr(u_{k,z\bar{z}} + \tfrac12 H e^{\omega}(\tau_{k-1,z}+\omega_{z}\,\tau_{k-1})\bigr) = -\tfrac{1}{\bar{Q}}u_{k,z\bar{z}}-\tfrac{H^2}{2\bar{Q}}e^{\omega}u_k \; . $$
On the other hand we have \,$\sigma_{k,z}=2Qe^{-\omega} u_k$\, by Equation~\eqref{eq:cmc:ps:ps1-sigmakz}. By equating these two presentations of \,$\sigma_{k,z}$\, we obtain
Equation~\eqref{eq:cmc:ps:uk-lin-gauss} for \,$u_k$\,. 

Separately, by differentiating Equation~\eqref{eq:cmc:ps:ps1-ukz} with respect to \,$z$\, we get
\begin{equation}
  \label{eq:cmc:ps:ps-taukz-pre}
  u_{k,zz} + Q \tau_{k,z} + \tfrac12 He^{\omega}(\sigma_{k+1,z}+\omega_{z}\,\sigma_{k+1}) = 0 \; .
\end{equation}
Equation~\eqref{eq:cmc:ps:ps1-ukz} implies
\begin{equation}
  \label{eq:cmc:ps:ps-taukz-pre1}
  \tfrac12He^{\omega}\sigma_{k+1} = -u_{k,z}- Q \tau_k
\end{equation}
and Equations~\eqref{eq:cmc:ps:ps1-sigmakz} and \eqref{eq:cmc:ps:ps1-taukz} imply
\begin{equation}
  \label{eq:cmc:ps:ps-taukz-pre2}
  \sigma_{k+1,z} = 2Q e^{-\omega} u_{k+1} = \tfrac{2Q}{H} e^{-\omega} (\tau_{k,z}+\omega_{z} \,\tau_k) \; .
\end{equation}
By inserting Equations~\eqref{eq:cmc:ps:ps-taukz-pre1} and \eqref{eq:cmc:ps:ps-taukz-pre2} into Equation~\eqref{eq:cmc:ps:ps-taukz-pre}, and solving for \,$\tau_{k,z}$\,,
we finally obtain the equation on the left-hand side of \eqref{eq:cmc:ps:ps-tau}.
\end{proof}

According to Equation~\eqref{eq:cmc:ps:uk-lin-gauss}, the Taylor coefficients of the diagonal entries of the polynomial Killing field are solutions of the
linearisation of the sinh-Gordon equation \eqref{eq:cmc:gauss}. Thus they can be interpreted as infinitesimal deformations in the space of conformal metrics
at the metric given by \,$e^{\omega}\,\diff z \diff \bar{z}$\,. This is the reason why \,$\xi$\, is called a polynomial ``Killing field''.

We now explain how one uses the Pinkall-Sterling iteration to reconstruct the entries of the polynomial Killing field corresponding to a given conformal factor \,$\omega$\,
of a cmc immersion of finite type \,$g$\,. We have the ``initial condition'' \,$\tau_{-1}=\tfrac12 H$\,. The sinh-Gordon equation~\eqref{eq:cmc:gauss}
and its linearisation \eqref{eq:cmc:ps:uk-lin-gauss} imply that in each step, the integrability condition \,$(\tau_{k,z})_{\bar z} = (\tau_{k,\bar{z}})_z$\, for \,$\tau_k$\,
is satisfied, thus the system of partial differential equations \eqref{eq:cmc:ps:ps-tau} has a solution \,$\tau_k$\,. This solution is only unique up to an
additive constant \,$C_k$\,. The condition that \,$\omega$\, has finite type \,$g$\, is equivalent to the property that the \,$C_k$\, can be chosen in such a way
that \,$u_{g}=0$\, holds; then also \,$\tau_g=0$\, holds.

As a corollary to the equations of the Pinkall-Sterling iteration, the lowest entries of the polynomial Killing field
$$ \xi = \frac12 \begin{pmatrix} 0 & H e^{\omega/2} \\ 0 & 0 \end{pmatrix} \lambda^{-1} + \sum_{k=0}^g \begin{pmatrix} u_k & v_k \\ w_k & -u_k \end{pmatrix} \lambda^k $$
can be written down explicitly: We have \,$\tau_{-1} = \tfrac12 H$\,. From Equation~\eqref{eq:cmc:ps:ps-u} we obtain
\,$u_0 = \tfrac{1}{H}(\tau_{-1,z}+\omega_z\tau_{-1}) = \tfrac12 \omega_z$\,, and then we obtain from Equation~\eqref{eq:cmc:ps:ps-sigma} and the
sinh-Gordon equation~\eqref{eq:cmc:gauss} \,$\sigma_0 = -\tfrac{1}{\bar{Q}}(\tfrac12 \omega_{z\bar{z}} + \tfrac14 H^2 e^\omega) = -Qe^{-\omega}$\,
(thereby recovering the right-hand half of \eqref{eq:cmc:xi-vm1w0}). Equations~\eqref{eq:cmc:ps:ps-tau} now give the following system of partial differential
equations for \,$\tau_0$\,
$$ \tau_{0,z} = -\frac{1}{4Q}(\omega_{zzz}-\omega_z\,\omega_{zz}) \qquad \tau_{0,\bar{z}} = \bar{Q}\,e^{-\omega}\,\omega_z \;, $$
which due to the Gauss equation~\eqref{eq:cmc:gauss} is solved by \,$\tau_0 = -\tfrac{1}{4Q}(\omega_{zz} - \tfrac12 \omega_z^2) + C_0$\, with a constant \,$C_0$\,.
Repeating one more step of the iteration gives
\begin{align*}
u_1 & = -\tfrac{1}{4HQ}\bigr(\omega_{zzz}-\tfrac12 \omega_z^3 - 4QC_0\omega_z \bigr) \\
\sigma_1 & = -\tfrac{1}{4H}\bigr(\omega_z^2+2\omega_{zz} + 8QC_0 \bigr)e^{-\omega} \\
\tau_1 & = \tfrac{1}{8HQ}\bigr(\omega_{zzzz}-\omega_{zzz}\omega_z + \tfrac12 \omega_{zz}^2 - \tfrac32 \omega_{zz}\omega_z^2 + \tfrac38 \omega_z^4 - 4C_0(\omega_{zz}-\tfrac12 \omega_z^2)\bigr) + C_1 \; . 
\end{align*}
From these calculations and the reality condition \eqref{eq:cmc:xi-reality} we see that 
that the polynomial Killing field \,$\xi$\, corresponding to the conformal factor \,$\omega$\, has the form
\begin{align}
  \label{eq:cmc:xi-terms}
  \xi & = \begin{pmatrix} 0 & \tfrac12 H e^{\omega/2} \\ 0 & 0 \end{pmatrix} \lambda^{-1}
  + \begin{pmatrix} \tfrac12 \omega_z & -e^{\omega/2}(\tfrac{1}{4Q}(\omega_{zz}-\tfrac12 \omega_z^2) - C_0) \\ -Qe^{-\omega/2} & -\tfrac12 \omega_z \end{pmatrix}
  + \dotsc \\
  \notag
  & \qquad \dotsc
  + \begin{pmatrix} -\tfrac12 \omega_{\bar{z}} & \bar{Q} e^{-\omega/2} \\ e^{\omega/2}(\tfrac{1}{4\bar{Q}}(\omega_{\bar{z}\bar{z}}-\tfrac12 \omega_{\bar{z}}^2) - \bar{C}_0) & \tfrac12 \omega_{\bar{z}} \end{pmatrix} \lambda^{g-1}
  + \begin{pmatrix} 0 & 0 \\ -\tfrac12 H e^{\omega/2} & 0 \end{pmatrix} \lambda^g \; .
\end{align}


In the Pinkall-Sterling iteration, specific choices of integration constants for $\tau_k $ yield the polynomial Killing field $\xi $ of the conformal factor $\omega $. We may more broadly consider all solutions to \eqref{eq:cmc:xi-dgl} of the form
\[
\xi = \sum_{k=-1}^{d} \xi_k\,\lambda^k \quad\text{with}\quad \xi_k = \begin{pmatrix} u_k & \tau_k\,e^{\omega/2} \\ \sigma_k\,e^{\omega/2} & -u_k \end{pmatrix} \;, 
\]
where the degree \,$d$\, is arbitrary and finite, \,$u_k, \tau_k, \sigma_k$\, are smooth, complex-valued functions of\,$z$ and \,$\tau_{-1} = \tfrac12 H$\,, \,$u_{-1} = \sigma_{-1} = 0$\,. Each such expression
is termed  \emph { a polynomial Killing field of $\omega$}. Not all choices of integration constants in the Pinkall-Sterling iteration process will yield solutions of finite degree but we restrict our attention to those which do.

\begin {lemma}\label{lemma:Killing}
(Cf.~\cite[Proposition~4.5]{HKS1})
 Assume that we are given as above a solution $ (\omega, H, Q) $ of the sinh-Gordon equation \eqref{eq:cmc:gauss} which is at least simply-periodic and of finite type. \begin {enumerate}
\item 
There is a unique polynomial Killing field $\xi ^ {\min} (\lambda,z) $ of minimal degree in $\lambda $. 
\item This minimal polynomial Killing field $ \xi ^ {\min} $ has no roots in $\lambda\in\C^\times $.
\item Each other polynomial Killing field can be expressed as $  p (\lambda)\xi ^ {\min} (\lambda, z) $ for some monic polynomial $ p \in \mathbb{C}[\lambda] $. 
\end {enumerate}
 \end {lemma}

\begin {proof} 
\begin {enumerate}
\item There exists a polynomial Killing field of finite degree, and therefore there exists one of minimal degree \,$d^{\min}$\,. It remains to show that there is only one polynomial Killing field of this degree \,$d^{\min}$\,. Suppose that $\xi ^ i $, $ i = 1, 2 $ are both polynomial Killing fields of $\omega $
$$ \xi ^ i = \sum_{k=-1}^{d^{\min}} \xi ^ i_k\,\lambda^k \quad\text{with}\quad \xi_k = \begin{pmatrix} u ^ i_k & \tau ^ i_k\,e^{\omega/2} \\ \sigma ^ i_k\,e^{\omega/2} & -u ^ i_k \end{pmatrix} \;. $$ 
By definition,
\[
\xi ^ 1_{-1} = \xi ^ 2_{-1} = \begin{pmatrix} 0 & \tfrac12 H e^{\omega/2} \\ 0 & 0 \end{pmatrix}.
\]
Consider the difference $ \xi ^D =\xi ^ 1 -\xi ^ 2 $. If it is nonzero then there is a smallest $l $,\, $ 0\leq l\leq d^{\min} $ for which $\xi ^D_l\neq 0 $. From Equations \eqref {eq:cmc:ps:ps-tau}, \eqref {eq:cmc:ps:ps-u} and  \eqref {eq:cmc:ps:ps-sigma}  we see that $ u_l =\sigma_l = 0 $ whilst $\tau_l $ is a nonzero constant. Then \,$\lambda^{-l-1}\tfrac{H}{2 \tau_l} \xi^D$\, solves the differential equation~\eqref{eq:cmc:xi-dgl} by linearity and has the correct lowest term to be a polynomial Killing field for $\omega $. It has degree strictly lower than \,$d^{\min}$\,. 
Hence the difference $\xi^D$ must vanish identically.
\item If $\xi ^ {\min}  $ has a root $\lambda_0 $ of some order $k \geq 1 $, then $\xi ^ 0: = \frac {(-\lambda_0) ^ k} {(\lambda -\lambda_0) ^ {k}}\xi ^ {\min} $ is a Laurent polynomial in $\lambda $ of degree \,$d^{\min}-k$\,, satisfies  \eqref{eq:cmc:xi-dgl} and has the same leading coefficient as $\xi ^ {\min} $. This contradicts (1), and hence \,$\xi^{\min}$\, has no roots.
\item Let $\tilde {\xi} =\sum_{k = -1} ^l \tilde {\xi}_k \lambda ^ k $ be a polynomial Killing field and $\xi ^ {\min} =\sum_{k = -1} ^{d^{\min}} \xi ^ {\min}_k\lambda ^ k $ the polynomial Killing field of minimal degree, where
\[
\tilde\xi_k  = \begin{pmatrix}  \tilde u_k  &  \tilde \tau_k \,e^{\omega/2} \\  \tilde \sigma_k\,e^{\omega/2} & - \tilde u_k\end{pmatrix}\;,\quad 
\xi ^{\min}_k = \begin{pmatrix} u ^{\min}_k & \tau ^{\min}_k \,e^{\omega/2} \\ \sigma ^{\min}_k \,e^{\omega/2} & -u ^{\min}_k\end{pmatrix}\;.
\]
 We show that it is possible to solve 
\[
\tilde {\xi} (\lambda, z) = p (\lambda)\xi ^ {\min} (\lambda, z) 
\]
with a monic polynomial \,$p \in \mathbb{C}[\lambda]$\,. All polynomial Killing fields obey \eqref {eq:cmc:xi-dgl} and hence fulfil Proposition~\ref{P:cmc:ps}. Each coefficient $ p_{k +1} $ of $p(\lambda)$ is iteratively determined by
\[
\tilde\tau_k = p_0\tau ^ {\min}_k + p_1\tau ^ {\min}_{k -1} +\cdots + p_{k +1}\tau ^ {\min}_{-1},
\]
giving in particular $ p_0 = 1 $. Since Equation~\eqref{eq:cmc:ps:ps-u} shows that \,$u_{k+1}$\, is uniquely determined by \,$\tau_k$\, and its derivative, we have 
\[
 \tilde u_{k +1} =  u ^ {\min}_{k +1} + p_1 u ^ {\min}_{k } + \cdots + p_{k +1} u ^ {\min}_0 \; . 
\]
Similarly, Equation~\eqref{eq:cmc:ps:ps-sigma} shows that \,$\sigma_{k+1}$\, is uniquely determined by \,$\tau_k$\, and \,$u_{k+1}$\,, which yields the analogous equation for $\sigma_{k+1}$. 
\end{enumerate}
\end{proof}

\begin {proposition}\label {proposition:spectral}
 Assume that we are given as above a solution $ (\omega, H, Q) $ of the sinh-Gordon equation \eqref{eq:cmc:gauss} which is at least simply-periodic and of finite type. 
\begin {enumerate}
\item  The eigenline curve $\Sigma $ of the monodromy agrees with the eigenline curve of any polynomial Killing field \,$\xi$\, of \,$\omega$\,. It agrees with the characteristic polynomial curve $\nu ^ 2 = -\det\left(\lambda\xi \right) $ of \,$\lambda\xi$\, if and only if \,$\xi=\xi^{\min}$\,.  
\item The polynomial Killing field defined by \eqref {eq:cmc:xi-eigenfunction} is the unique polynomial Killing field $\xi ^ {\min}$ of minimal degree. 
\end {enumerate}
\end {proposition}

\begin {proof}
\begin {enumerate}
\item The curve $\Sigma $ is defined as the eigenline curve of the monodromy $ M_{z_0} (\lambda) $ of the frame $ F_\lambda (z_0) $,  and  \eqref {eq:cmc:monodromy-basepoint} shows that $\Sigma $ is independent of the base point $z_0 $. We may alternatively view $ M_{z_0} $ as the monodromy of the family of flat connections $ d +\alpha_\lambda $ with respect to the base point $ z_0 $, where $\alpha_\lambda = F_\lambda^ {- 1}d F_\lambda $. A polynomial Killing field $\xi $ is a section of the endomorphism bundle of the trivial \,$\mathbb{C}^2$-bundle over \,$\Sigma$\,. By \eqref {eq:cmc:xi-dgl}, it is parallel with respect to $ d +\alpha_\lambda $. The monodromy action on the endomorphism bundle is by conjugation:
\[
 \xi (\lambda, z + T) = M_z ^ {- 1}(\lambda)\xi (\lambda, z) M_z (\lambda)\;,
\]
where \,$T\in \mathbb{C}^\times$\, is the minimal period of the (at least) simply-periodic solution \,$\omega$\,. 
The Pinkall-Sterling iteration preserves the period \,$T$\, of the initial data \,$\omega$\,, and therefore the polynomial Killing field $\xi$ is also periodic with period $T$. Hence any polynomial Killing field $\xi $ commutes with the monodromy and their eigenline curves agree.

Since the trace of $\lambda\xi$ vanishes, its characteristic polynomial takes the form $\nu ^ 2 = -\det\left(\lambda\xi (\lambda)\right) $. We now consider \,$\xi=\xi^{\min}$\, and 
take $\lambda_0\in\mathbb C ^\times $. If $\xi ^ {\min} (\lambda_0) $ has two distinct eigenvalues,  then it has distinct eigenlines. Otherwise $\xi ^ {\min} (\lambda_0) $ has up to multiplicity only one eigenvalue, which must be $ 0 $ because the trace vanishes. By Lemma~\ref{lemma:Killing}, $\xi ^ {\min} $ has no roots in $\lambda $, so the $ 0 $-eigenspace is one-dimensional. Writing
 \[
\xi ^{\min}  = \begin{pmatrix} u ^{\min}  & \tau ^{\min}\,e^{\omega/2} \\ \sigma ^{\min}\,e^{\omega/2} & -u ^{\min}\end{pmatrix} \;, 
\]
one of the off-diagonal entries is non-zero at $\lambda_0 $. If $\tau ^{\min} (\lambda_0)\neq 0 $, then in a neighbourhood of $\lambda_0 $  
the eigenlines of $\xi ^{\min} (\lambda) $
are spanned by
\[
\begin {pmatrix}\tau^ {\min}  e ^ {\omega/2}\\
\pm \left (-\det\left (\xi ^{\min} \right)\right) ^ \frac 12 - u ^{\min} 
\end {pmatrix}.
\]
Recall that the eigenline curve is given by $\nu ^ 2 =\lambda a (\lambda) $, where the polynomial \,$a(\lambda)$\, vanishes at \,$\lambda_0$\, to the order to which the eigenlines at \,$\lambda_0$\, agree. By the above formula, this is equal to the order of vanishing of $-\det\left (\xi ^{\min} \right) $, or equivalently of $-\det\left (\lambda \xi ^{\min} \right) $ at \,$\lambda_0 \in \bbC^{\times}$\,. The case $\sigma ^ {\min} (\lambda_0)\neq 0 $ is analogous.

To complete the proof of \,$-\det\left(\lambda\xi ^{\min} (\lambda)\right) = \lambda a (\lambda)$\,, we consider \,$\lambda=0$\,. 
 From  \eqref{eq:cmc:xi-terms},
 \[
-\det\left(\lambda\xi ^{\min} (\lambda)\right) = -\frac 12 HQ\lambda + O (\lambda^2). 
\]
Therefore the polynomials $-\det\left(\lambda\xi ^{\min} (\lambda)\right) $ and $\lambda a (\lambda) $ both have a simple root at \,$\lambda=0$\, and the same lowest-order term.
It follows that these polynomials coincide.

On the other hand, if \,$\xi$\, is a polynomial Killing field that is not the minimal one, then \,$-\det(\lambda \xi(\lambda))$\, has additional roots by Lemma~\ref{lemma:Killing}(3) and therefore cannot be equal to \,$\lambda a(\lambda)$\,. 

\item 
The polynomial Killing field $\xi (\lambda) $ of  \eqref {eq:cmc:xi-eigenfunction} was defined so that the eigenvalues of $\lambda\xi (\lambda) $ are given by the function $\nu $ of \eqref {eq:cmc:Sigma} and so that the characteristic polynomial curve of $\lambda\xi (\lambda) $ coincides with the eigenline curve $\Sigma $.  By (1), we have \,$\xi = \xi^{\min}$\,.
\end {enumerate}
\end {proof}

To sum up this discussion, we consider the \emph{space of cmc potentials} of spectral genus \,$g$\,
\begin{equation}
  \label{eq:cmc:potentials}
  \mathcal{P}_g = \left\{ \left. \zeta_\lambda = \sum_{k=-1}^g \zeta_k\lambda^k \right| \zeta_k = -\overline{\zeta_{g-(k+1)}}^t \in \mathfrak{sl}(2,\C), \; \zeta_{-1} \in \left( \begin{smallmatrix} 0 & \R^+ \\ 0 & 0 \end{smallmatrix} \right), \mathrm{tr}(\zeta_{-1}\,\zeta_0) \neq 0 \right\} \; .
\end{equation}
Then the polynomial Killing field \,$\xi = \xi_z$\, corresponding to the conformal factor \,$\omega$\, of a cmc immersion,
seen as a function depending on the base point \,$z$\,, has values in \,$\mathcal{P}_g$\, 
and is a solution of the differential equation \eqref{eq:cmc:xi-dgl}.
On \,$\mathcal{P}_g$\, we moreover consider the following 1-form-valued linear map (cf.~\cite[Equation~(3.3)]{HKS1}), which associates to each potential the corresponding cmc connection form:
\begin{equation}
  \label{eq:cmc:alpha-pkf}
  \mathcal{P}_g \to \Omega^1(\C) \otimes \mathfrak{sl}(2,\C), \; \zeta \mapsto \alpha_\lambda(\zeta) := \begin{pmatrix} \tfrac12 u_0 & v_{-1}\,\lambda^{-1} \\ w_0 & -\tfrac12 u_0 \end{pmatrix} \diff z
  - \begin{pmatrix} \tfrac12 \bar{u}_0 & \bar{w}_0 \\ \bar{v}_{-1}\,\lambda & -\tfrac12 \bar{u}_0 \end{pmatrix} \diff \bar{z} \; .
\end{equation}
By comparison of Equation~\eqref{eq:cmc:xi-terms} with Equation~\eqref{eq:cmc:alpha} we see that we have \,$\alpha_\lambda(\zeta) = \alpha_\lambda$\,,
where \,$\alpha_\lambda(\zeta)$\, is defined by Equation~\eqref{eq:cmc:alpha-pkf} and \,$\alpha_\lambda$\, is defined by Equation~\eqref{eq:cmc:alpha} with respect
to the conformal factor \,$\omega$\, to which the polynomial Killing field \,$\xi$\, conforms. In particular, all cmc polynomial Killing fields \,$\xi$\, of spectral genus \,$g$\,
are solutions of the differential equation
\begin{equation}
  \label{eq:cmc:xi-dgl-2}
  \diff \xi + [\alpha_\lambda(\xi),\xi] = 0 \quad\text{on \,$\mathcal{P}_g$\,,}
\end{equation}
and from such a solution, say \,$\xi = \left( \begin{smallmatrix} 0 & v_{-1} \\ 0 & 0 \end{smallmatrix} \right)\lambda^{-1} + \dotsc$\,,
the corresponding conformal factor \,$\omega$\, can be recovered by means of the equation \,$\tfrac12 H e^{\omega/2} = v_{-1}$\,. 
In this sense the definition of the space of potentials \eqref{eq:cmc:potentials} together with the linear map \eqref{eq:cmc:alpha-pkf} characterises the integrable system
of the sinh-Gordon equation on the level of polynomial Killing fields.

{\bf Symes' method.}
The polynomial Killing field encodes all information of the spectral curve \,$\Sigma$\, and the eigenline bundle \,$\Lambda_{z_0}$\, via the the characterisation that
\,$\Sigma$\, is the complex curve defined by the characteristic equation of \,$\lambda\,\xi_{z_0}$\, and \,$\Lambda_{z_0}$\, is the eigenline bundle of \,$\lambda\,\xi_{z_0}$\,
seen as a holomorphic line bundle on \,$\Sigma$\,. 
Therefore it is expected that the extended frame \,$F_\lambda$\, and hence (via the Sym-Bobenko formula~\eqref{eq:cmc:sym-bobenko}) the cmc immersion \,$f_\lambda$\,
can be recovered from the polynomial Killing field \,$\xi_{z_0}$\, (at a fixed, arbitrarily chosen base point \,$z_0$\,). The process by which this is accomplished
is known as \emph{Symes' method} \cite{symes}. We give a description of this method below, which closely follows \cite[Section~3]{HKS1}.
For this purpose we need the Iwasawa decomposition of the loop group \,$\Lambda \mathrm{SL}(2,\C)$\, in the form introduced by Pressley/Segal \cite{PreS} and then generalised by McIntosh \cite{McI}: 

\begin{proposition}
  \label{P:cmc:iwasawa}
  For \,$0<r<\infty$\, we denote the circle \,$S^1_r = \partial B(0,r)= \{\lambda \in \C\mid |z|=r\}$\, and the annulus \,$A_r = \{\lambda \in \C\mid \min(r,r^{-1})<|\lambda|<\max(r,r^{-1})\}$\,. We consider the loop groups
\begin{align*}
  \Lambda_r \mathrm{SL}(2,\C) & = \{ \Phi: S^1_r \to \mathrm{SL}(2,\C) \text{ analytic }  \} \\
  \Lambda_r \mathrm{SU}(2) & = \{ F \in \Lambda_r \mathrm{SL}(2,\C) \mid \text{\,$F$\, extends analytically to \,$A_r$\, and \,$F[S^1_r] \subset \mathrm{SU}(2)$\,} \} \\
  \Lambda_r^+ \mathrm{SL}(2,\C) & = \{ \Phi \in \Lambda_r \mathrm{SL}(2,\C) \mid \text{\,$\Phi$\, extends analytically to \,$\overline{B(0,r)}$\,}\} \\
  \Lambda_r^{+,\R} \mathrm{SL}(2,\C) & = \{ \Phi \in \Lambda_r^+ \mathrm{SL}(2,\C) \mid \text{\,$\Phi|_{\lambda=0} = \left( \begin{smallmatrix} \rho & c \\ 0 & \rho^{-1} \end{smallmatrix} \right) $\, with some \,$\rho>0$\, and \,$c\in\C$\,}\}  
\end{align*}
In the case \,$r=1$\, we omit the subscript \,${}_r$\, from the loop groups.

  The map
  $$ \textsf{Iwasawa}_r: \Lambda_r \mathrm{SU}(2) \times \Lambda_r^{+,\R} \mathrm{SL}(2,\C) \to \Lambda_r \mathrm{SL}(2,\C),\; (F,B) \mapsto F\cdot B $$
  is a real analytic diffeomorphism onto \,$\Lambda_r \mathrm{SL}(2,\C)$\,. For given \,$\Phi \in \Lambda_r \mathrm{SL}(2,\C)$\,, \,$(F,B) = \textsf{Iwasawa}_r^{-1}(\Phi)$\,
  is called the \emph{\,$r$-Iwasawa decomposition} of \,$\Phi$\,. 
\end{proposition}

\begin{proposition}
  \label{P:cmc:symes}
  Let a sinh-Gordon potential \,$\zeta \in \mathcal{P}_g$\, be given. We fix a base point \,$z_0$\, and let \,$\Phi_\lambda(z) = F_\lambda(z)\cdot B_\lambda(z)$\,
  be the Iwasawa decomposition (Proposition~\ref{P:cmc:iwasawa}) of \,$\Phi_\lambda(z) = \exp\bigr( (z-z_0)\zeta \bigr)$\,.
  In this situation \,$F_\lambda(z)$\, extends holomorphically to \,$\C^{\times}  \ni \lambda$\,, we have \,$\diff F_\lambda = F_\lambda\,\alpha_\lambda$\,,
  and for \,$\lambda \in S^1$\,, the Sym-Bobenko formula
  \eqref{eq:cmc:sym-bobenko} with this \,$F_\lambda$\, gives a conformal cmc immersion \,$f$\, into \,$\R^3$\, with conformal metric \,$e^\omega\,\diff z \,\diff \bar{z}$\,.
  Here \,$\omega$\, is the real-valued function characterised by \,$\xi = \left( \begin{smallmatrix} 0 & \tfrac12 H e^{\omega/2} \\ 0 & 0 \end{smallmatrix} \right)\lambda^{-1} + \dotsc$\,
  for the unique solution \,$\xi: \C \to \mathcal{P}_g$\, of Equation~\eqref{eq:cmc:xi-dgl-2} with \,$\xi_{z_0}=\zeta$\,, and \,$\alpha_\lambda$\, is defined
  by Equation~\eqref{eq:cmc:alpha} with respect to this \,$\omega$\,. 
\end{proposition}

\begin{proof}
  (Compare \cite[Proposition~3.2]{HKS1}.) It suffices to show \,$F_\lambda^{-1}\,\diff F_\lambda = \alpha_\lambda(\xi)$\,.
  We have \,$F_\lambda = \Phi_\lambda \cdot B_\lambda^{-1}$\,, so \,$F_\lambda$\, is obtained from \,$\Phi_\lambda$\,
  by a gauge transformation with the gauge \,$B_\lambda^{-1}$\,. It follows that the corresponding connection form \,$F_\lambda^{-1}\,\diff F_\lambda$\, is obtained from
  \,$\Phi_\lambda^{-1}\,\diff \Phi_\lambda = \zeta \,\diff z$\, by the formula
  \begin{equation}
    \label{eq:cmc:symes:F-regauge}
    F_\lambda^{-1}\,\diff F_\lambda = B_\lambda\,\zeta\,B_\lambda^{-1}\,\diff z - \diff B_\lambda \cdot B_\lambda^{-1} = \xi_\lambda\,\diff z - \diff B_\lambda \cdot B_\lambda^{-1} \; .
  \end{equation}
  Here the second equality follows from Equation~\eqref{eq:cmc:xi-basepointchange} and the fact that \,$\Phi_\lambda = F_\lambda\cdot B_\lambda$\, commutes with \,$\zeta$\,. 
  By the properties of the Iwasawa decomposition, \,$B_\lambda$\, extends holomorphically to \,$\lambda=0$\,, thus the series expansion of the
  right hand side of Equation~\eqref{eq:cmc:symes:F-regauge} in \,$\lambda$\, 
  can only contain powers \,$\lambda^k$\, with \,$k\geq -1$\,. On the other hand, the reality condition \eqref{eq:cmc:xi-reality} for \,$\zeta$\, implies
  that
  \begin{equation}
    \label{eq:cmc:symes:F-reality}
    \overline{F_{\bar{\lambda}^{-1}}} = F_\lambda^{t,-1}
  \end{equation}
  holds also in the present situation. Therefore also \,$\overline{F_{\bar{\lambda}^{-1}}}$\, only contains powers \,$\lambda^k$\,
  of \,$\lambda$\, with \,$k\geq -1$\,. Combining these two statements, we see that \,$F_\lambda^{-1}\,\diff F_\lambda$\, contains only the powers \,$\lambda^k$\, of \,$\lambda$\,
  with \,$k \in \{-1,0,1\}$\,. We thus write
  $$ F_\lambda^{-1}\,\diff F_\lambda = (A_{-1}'\lambda^{-1}+A_0'+A_1'\lambda)\diff z + (A_{-1}''\lambda^{-1}+A_0''+A_1''\lambda)\diff \bar{z} \quad\text{with}\quad
  A_k', A_k'' \in \mathfrak{sl}(2,\C) \; . $$
  Here the reality condition \eqref{eq:cmc:symes:F-reality} implies
  \begin{equation}
    \label{eq:cmc:symes:F-reality2}    
    A_{-1}'' = -\overline{A_1'}^t\;,\quad A_0'' = -\overline{A_0'}^t \quad\text{and}\quad A_{1}'' = -\overline{A_{-1}'}^t \; .
  \end{equation}
  We now write
  $$ B_{\lambda=0} = \left( \begin{smallmatrix} \rho & c \\ 0 & \rho^{-1} \end{smallmatrix} \right) \quad\text{and}\quad
  \xi_\lambda = \sum_{k=-1}^g \xi_{k}\lambda^k \quad\text{with}\quad \xi_{-1} = \left( \begin{smallmatrix} 0 & {v}_{-1} \\ 0 & 0 \end{smallmatrix} \right)
  \;, \xi_{k} = \left( \begin{smallmatrix} {u}_k & {v}_{k} \\ {w}_k & -{u}_k \end{smallmatrix} \right) $$
  with smooth functions \,$\rho: \C \to \R^+$\, and \,$c,{u}_k,{v}_k,{w}_k: \C \to \C$\,.
  It then follows from Equation~\eqref{eq:cmc:symes:F-regauge} that
  \begin{align}
    \label{eq:cmc:symes:A-1'}
    A_{-1}' & = \xi_{-1} \;, \\
    \label{eq:cmc:symes:A-1''}    
    A_{-1}'' & = 0 \;, \\
    \label{eq:cmc:symes:A0'}    
    A_0' & = \xi_0 - (\tfrac{\diff\ }{\diff z} B_{\lambda=0})\,B_{\lambda=0}^{-1} = \left( \begin{smallmatrix} {u}_0-\rho^{-1}\,\rho_{z} & {v}_0-\rho\,c_{z}+c\,\rho_{z} \\ {w}_0 & -{u}_0+\rho^{-1}\,\rho_{z} \end{smallmatrix} \right) \; ,  \\
    \label{eq:cmc:symes:A0''}    
    A_0'' & = -(\tfrac{\diff\ }{\diff\bar{z}} B_{\lambda=0})\,B_{\lambda=0}^{-1} = \left( \begin{smallmatrix} -\rho^{-1}\,\rho_{\bar z} & -\rho\,c_{\bar z}+c\,\rho_{\bar z} \\ 0 & \rho^{-1}\,\rho_{\bar z} \end{smallmatrix} \right) \; . 
  \end{align}
  By comparing the two representations of \,$A_0'' = -\overline{A_0'}^t$\, that are obtained from Equations~\eqref{eq:cmc:symes:A0'} and \eqref{eq:cmc:symes:A0''}
  and using the fact that \,$\rho$\, is real-valued, we see that \,$-(\overline{{u}_0-\rho^{-1}\,\rho_{z}}) = -\rho^{-1}\,\rho_{\bar z}$\, and hence
  \,${u}_0 = 2\rho^{-1}\,\rho_z$\, holds, that \,$-\overline{{w}_0} = -\rho\,c_{\bar z} + c\,\rho_{\bar z}$\, and hence
  \,${w}_0 = \rho\,\bar{c}_z - \bar{c}\,\rho_z$\, holds, and that \,$-(\overline{{v}_0-\rho\,c_{z}+c\,\rho_{z}}) = 0$\, and hence
  \,${v}_0 = \rho\,c_z - c\,\rho_z$\, holds. By inserting these equations into
  \eqref{eq:cmc:symes:A0'}, we see that
  \begin{equation}
    \label{eq:cmc:symes:A0'-2}
    A_0' = \left( \begin{smallmatrix} \tfrac12 {u}_0 & 0 \\ {w}_0 & -\tfrac12 {u}_0 \end{smallmatrix} \right)
  \end{equation}
  holds. It now follows from Equations~\eqref{eq:cmc:symes:A-1'}, \eqref{eq:cmc:symes:A-1''}, \eqref{eq:cmc:symes:A0'-2} and \eqref{eq:cmc:symes:F-reality2}
  that we have \,$F_\lambda^{-1}\,\diff F_\lambda = \alpha_\lambda(\xi)$\,. 
\end{proof}

We summarise the preceding results in the following statement:

\begin{theorem}
\label{T:cmc:bijections}
We have constructed explicit bijections between each of the following sets of data:
\begin{enumerate}
	\item 
	Cmc immersions \,$f: \mathbb{R}^2 \to \R^3$\, with fixed values of \,$H$\, and \,$Q$\, and of finite type \,$g$\,, modulo isometries of the ambient space \,$\R^3$\,. 
	\item 
	Real-valued solutions \,$\omega$\, of the sinh-Gordon equation of finite type \,$g$\,.
	\item
	Cmc polynomial Killing fields, i.e.~smooth solutions \,$\xi: \mathbb{R}^2 \to \mathcal{P}_g$\, of the Lax equation~\eqref{eq:cmc:xi-dgl-2} with \,$\xi \mapsto \alpha(\xi)$\, given by Equation~\eqref{eq:cmc:alpha-pkf}, without zeros in \,$(z,\lambda)$\,.
	\item
	Cmc potentials without zeros (i.e.~elements of \,$\mathcal{P}_g$\, without	zeros).
\end{enumerate}
\end{theorem}

\begin{proof}
Given the data in (1), \,$\omega$\, in (2) is the conformal factor of \,$f$\,. Given the data in (2), the polynomial Killing field in (3) is obtained by Lemma~\ref{lemma:Killing}. Given the data in (3), the cmc potential of (4) is \,$\xi(0)$\,. And given the data in (4), the cmc immersion \,$f$\, in (1) is obtained by Proposition~\ref{P:cmc:symes} with the choice of base point \,$z_0=0$\,. 	
\end{proof}

Note that the bijections in Theorem~\ref{T:cmc:bijections} depend on the following choices: in (1) and (2) we chose the specific form of the sinh-Gordon equation \eqref{eq:cmc:gauss}, in (3) we chose the specific form of \,$\alpha(\xi)$\, in \eqref{eq:cmc:alpha}, and in (4) we chose the base point \,$z_0=0$\,. 

These bijections induce bijections between cmc torus immersions in (1) and corresponding subsets in (2), (3) and (4). These subsets can be characterised by using the explicit descriptions of the bijections.

{\bf Product and trace formulas.}
The spectral curve \,$\Sigma$\, is essentially encoded in the location of the zeros of the polynomial \,$a(\lambda)$\, of degree \,$2g$\,
that characterises the hyperelliptic complex curve \,$\Sigma$\, via Equation~\eqref{eq:cmc:Sigma}. For the purposes of constructing blowups of cmc immersions
it is useful to express the relationship explicitly, as is done in the following proposition.

\begin{proposition}
  \label{P:cmc:lambdak}
  The polynomial \,$a(\lambda)$\, has exactly \,$g$\, zeros \,$\lambda_1,\dotsc,\lambda_g \in B(0,1) \setminus \{0\}$\, (counted with multiplicity)
  inside the unit disk, which we number
  such that \,$|\lambda_1|\leq \dotsc \leq |\lambda_g|$\,. Then \,$\lambda_1,\dotsc,\lambda_g,\bar{\lambda}_1^{-1},\dotsc,\bar{\lambda}_g^{-1}$\, are all zeros of
  \,$a(\lambda)$\,, and thus we have
  \begin{equation}
    \label{eq:cmc:lambdak:a}
    a(\lambda) = -\frac12 HQ \prod_{k=1}^g (1-\lambda_k^{-1}\lambda)\cdot (1-\bar{\lambda}_k\,\lambda) \; .
  \end{equation}
\end{proposition}

\begin{proof}
  We first show that \,$a(\lambda)$\, has no zeros on \,$S^1$\,. Indeed, for \,$\lambda \in S^1$\,, the matrix \,$\xi(\lambda)$\, is skew-Hermitian. Because on the space of trace-free skew-Hermitian matrices, the determinant is the square of a norm, \,$a(\lambda)=-\lambda\det(\xi(\lambda))=0$\, would imply \,$\xi(\lambda)=0$\, for such \,$\lambda$\, in contradiction to the minimality of \,$\xi$\,. 

  We have \,$a(0)=-\tfrac12 HQ \neq 0$\,, so \,$\lambda=0$\, is not a zero of \,$a(\lambda)$\,.  If some \,$\lambda\in \C^{\times} $\, is a zero of \,$a(\lambda)$\,,
  then \,$\bar{\lambda}^{-1}$\, also is a zero of \,$a(\lambda)$\, due to the reality condition~\eqref{eq:cmc:a-reality}.
  It follows that out of the \,$2g$\, many zeros of \,$a(\lambda)$\,, exactly \,$g$\, are inside the punctured unit disk, which we denote by \,$\lambda_1,\dotsc,\lambda_g$\,
  as in the proposition, and then the other \,$g$\, zeros are \,$\bar{\lambda}_1^{-1}\,\dotsc,\bar{\lambda}_g^{-1} \in \C \setminus \overline{B(0,1)}$\,. Now we note that
  the expression on the right hand side of Equation~\eqref{eq:cmc:lambdak:a} is a polynomial of degree \,$2g$\, in \,$\lambda$\, which has the same zeros as \,$a(\lambda)$\,
  and whose value at \,$\lambda=0$\, equals \,$-\tfrac12 HQ = a(0)$\,. This proves Equation~\eqref{eq:cmc:lambdak:a}.
\end{proof}

Similarly, the polynomial Killing field \,$\xi_{z_0}$\, and the values of \,$\omega$\, and \,$\omega_z$\, at \,$z=z_0$\, can be reconstructed from the divisor of any meromorphic section \,$s$\,  of the spectral line bundle \,$\Lambda_{z_0}$\,. More specifically, if we write
\begin{equation}
  \label{eq:cmc:uvw}
  \xi_{z_0} = \begin{pmatrix} u(\lambda) & v(\lambda) \\ w(\lambda) & -u(\lambda) \end{pmatrix} \;,
\end{equation}
where \,$u(\lambda)$\,, \,$\lambda v(\lambda)$\, and \,$w(\lambda)$\, are polynomials of degree \,$g$\,, then \,$\psi = \tfrac{\nu-\lambda u(\lambda)}{\lambda v(\lambda)}$\,
is a meromorphic function on \,$\Sigma$\, such that
\,$\xi_{z_0} \left( \begin{smallmatrix} 1 \\ \psi \end{smallmatrix} \right) = \tfrac{\nu}{\lambda} \left( \begin{smallmatrix} 1 \\ \psi \end{smallmatrix} \right)$\, holds,
thus \,$(1,\psi)$\, is a meromorphic section of \,$\Lambda_{z_0}$\, without zeros. 
Hence a representative divisor for the line bundle \,$\Lambda_{z_0}$\, is given by the polar divisor
of \,$\psi$\,. We call this specific representative divisor the \emph{spectral divisor} of \,$\omega$\,. 
For \,$\lambda \in \bbC^\times$\,, a pole of \,$\psi$\, can only occur at points where \,$\lambda v(\lambda)=0$\,. At such points, we have \,$\lambda u(\lambda)=\pm \nu$\, due to
Equation~\eqref{eq:cmc:det-xi}, and hence a pole of \,$\psi$\, occurs at some \,$(\lambda,\nu)\in \Sigma$\, if and only if \,$\lambda u(\lambda)=-\nu$\, and \,$\lambda v(\lambda)=0$\,
holds. It follows that the support of the spectral divisor restricted to \,$\lambda\in \bbC^\times$\, consists of those points \,$(\beta_k,\nu_k) \in \Sigma$\, where \,$(\lambda v)(\beta_k)=0$\,
and \,$\nu_k = -\beta_k\,u(\beta_k)$\, holds. Because \,$\lambda v$\, is a polynomial of degree \,$g$\,, there are exactly \,$g$\, such points.
Again it is useful to make the reconstruction process explicit:
\begin{proposition}
  \label{P:cmc:xi-trace}
  Let \,$(\beta_k,\nu_k) \in \Sigma$\, be the points in the support of the spectral divisor, where \,$k=1,\dotsc,g$\,. 
  Then we have:
  \begin{align}
    \label{eq:cmc:xi-trace:omega}
    e^{\omega(z_0)} & = 2(-1)^g \frac{\bar{Q}}{H} \prod_{k=1}^g \beta_k =  \frac{2|Q|}{H} \prod_{k=1}^g |\beta_k| \quad\text{where}\quad (-1)^g\,\bar{Q} \prod_{k=1}^g \beta_k \in \mathbb{R}_+ \\
    \label{eq:cmc:xi-trace:omegaz}
    \omega_z(z_0)^2 & = 8 HQ \sum_{k=1}^g \left( \lambda_k^{-1}-\beta_k^{-1} + \bar{\lambda}_k - \bar{\beta}_k \right) \\
    \label{eq:cmc:xi-trace:u}
    u(\lambda) & = -\sum_{k=1}^g \beta_k^{-1}\,\nu_k\,\chi_k(\lambda) \\
    \label{eq:cmc:xi-trace:chi}
    & \quad\text{with}\quad \chi_k(\lambda) = \prod_{k'\neq k} \frac{\lambda-\beta_{k'}}{\beta_k-\beta_{k'}} = -\frac{2}{H}e^{-\omega(z_0)/2}\beta_k\left( \prod_{k'\neq k} (1-\beta_{k'}^{-1}\beta_k)^{-1}\right) \frac{\lambda v(\lambda)}{\lambda-\beta_k} \\
    \label{eq:cmc:xi-trace:v}
    v(\lambda) & = \frac12 H e^{\omega(z_0)/2}\,\lambda^{-1}\prod_{k=1}^g (1-\beta_k^{-1}\lambda) \\
    \label{eq:cmc:xi-trace:w}
    w(\lambda) & = -Q e^{-\omega(z_0)/2}\,\prod_{k=1}^g (1-\bar{\beta}_k \lambda) \; . 
  \end{align}
In~\eqref{eq:cmc:xi-trace:chi} we have assumed that the $\beta_k$ are pairwise different. If not, then this equation should be replaced by the appropriate limit.
\end{proposition}

\begin{proof}
  It was discussed above the proposition that the \,$\beta_1,\dotsc,\beta_g$\, are the zeros of the polynomial \,$\lambda v(\lambda)$\,.
  Because of \,$\lambda v(\lambda)|_{\lambda=0} = v_{-1} \neq 0$\, they are all in \,$\C^{\times} $\,. 
  Moreover, the zeros of the polynomial \,$w(\lambda)$\, with \,$w(0)=w_0$\, are \,$\bar{\beta}_1^{-1},\dotsc,\bar{\beta}_g^{-1}$\, 
  due to the reality condition~\eqref{eq:cmc:xi-reality}. Thus we have
  \begin{equation}
    \label{eq:cmc:xi-trace:vw-prod}
    \lambda v(\lambda) = v_{-1}\,\prod_{k=1}^g (1-\beta_k^{-1}\lambda) \AND w(\lambda) = w_0\,\prod_{k=1}^g (1-\bar{\beta}_k \lambda) \; .
  \end{equation}
  By inserting \eqref{eq:cmc:xi-vm1w0} into these equations, Equations~\eqref{eq:cmc:xi-trace:v} and \eqref{eq:cmc:xi-trace:w} follow.
  On the other hand, it also follows from the reality condition \eqref{eq:cmc:xi-reality} together with the product formula for \,$v(\lambda)$\, in \eqref{eq:cmc:xi-trace:vw-prod}
  that
  \begin{align*}
    w(\lambda) & = - \lambda^g \,\overline{(\lambda v)(\bar{\lambda}^{-1})} = -\lambda^g \,\bar{v}_{-1}\,\prod_{k=1}^g (1-\bar{\beta}_k^{-1}\lambda^{-1}) \\
    & = (-1)^{g+1}\,\bar{v}_{-1} \prod_{k=1}^g \bar{\beta}_k^{-1}\cdot \prod_{k=1}^g (1-\bar{\beta}_k\lambda) \; .
  \end{align*}
  By comparison with the product formula for \,$w(\lambda)$\, in \eqref{eq:cmc:xi-trace:vw-prod} we see that
  $$ w_0 = (-1)^{g+1}\,\bar{v}_{-1} \prod_{k=1}^g \bar{\beta}_k^{-1} \quad\text{and thus}\quad \frac{\bar{v}_{-1}}{w_0} = (-1)^{g+1} \prod_{k=1}^g \bar{\beta}_k $$
  holds. Due to Equation~\eqref{eq:cmc:xi-vm1w0} we have \,$\tfrac{\bar{v}_{-1}}{w_0} = -\tfrac{H}{2Q} e^{\omega(z_0)}$\, and hence
  $$ e^{\omega(z_0)} = (-1)^g \frac{2Q}{H} \prod_{k=1}^g \bar{\beta}_k \; . $$
  Because of \,$e^{\omega(z_0)} > 0$\, we deduce \,$(-1)^g Q \prod_{k=1}^g \bar{\beta}_k \in \R_+$\,, in particular \,$Q \prod_{k=1}^g \bar{\beta}_k = \bar{Q} \prod_{k=1}^g \beta_k$\,.
  This proves Equation~\eqref{eq:cmc:xi-trace:omega}.

  The polynomial \,$u(\lambda)$\, has degree \,$g-1$\, (indeed \,$u_g = -\bar{u}_{-1}=0$\, by the reality condition~\eqref{eq:cmc:xi-reality}), and is therefore
  uniquely determined by the \,$g$-many equations \,$u(\beta_k)=-\beta_k^{-1}\,\nu_k$\,. It is therefore given by Equation~\eqref{eq:cmc:xi-trace:u}, where we define
  \,$\chi_k(\lambda)$\, as the unique polynomial of degree \,$g-1$\, with \,$\chi_k(\beta_{k'})=0$\, for \,$k'\neq k$\, and \,$\chi_k(\beta_k)=1$\,. Then
  the first equals sign in \eqref{eq:cmc:xi-trace:chi} is obvious. Due to the fact that \,$\lambda v(\lambda)$\, is a polynomial of degree \,$g$\, whose zeros are
  exactly \,$\beta_1,\dotsc,\beta_g$\,, we also have
  $$ \chi_k(\lambda) = \frac{\lambda v(\lambda)}{(\lambda v)'(\beta_k) \cdot (\lambda-\beta_k)} \;. $$
  By Equation~\eqref{eq:cmc:xi-trace:v} we have
  $$ (\lambda v)'(\beta_k) = -\frac12 H e^{\omega(z_0)/2} \beta_k^{-1} \prod_{k'\neq k} (1-\beta_{k'}^{-1}\,\beta_k) \; , $$
  and this implies the second equals sign in \eqref{eq:cmc:xi-trace:chi}. 

  Finally, by Equation~\eqref{eq:cmc:det-xi} we have
  $$ a(\lambda) = -\lambda\,\det(\xi_{z_0}) = \lambda\,u(\lambda)^2 + (\lambda v(\lambda))\cdot w(\lambda) \;, $$
  therefore
  $$ a'(\lambda) = u(\lambda)^2 + 2\lambda\,u(\lambda)\,u'(\lambda) + (\lambda v)'(\lambda)\cdot w(\lambda) + (\lambda v(\lambda)) \cdot w'(\lambda) $$
  and hence
  \begin{equation}
    \label{eq:cmc:xi-trace:a'0}
    a'(0) = u_0^2 +v_0\,w_0 + v_{-1}\,w_1 \; .
  \end{equation}
  On the other hand, from Equations~\eqref{eq:cmc:lambdak:a}, \eqref{eq:cmc:xi-trace:v} and \eqref{eq:cmc:xi-trace:w} we obtain
  \begin{align*}
    a'(0) & = \frac12 HQ \sum_{k=1}^g (\lambda_k^{-1}+\bar{\lambda}_k) \\
    v_{-1} & = \frac12 H e^{\omega(z_0)/2} \\
    v_0 & = -\frac12 H e^{\omega(z_0)/2} \sum_{k=1}^g \beta_k^{-1} \\
    w_0 & = -Q e^{-\omega(z_0)/2} \\
    w_1 & = Q e^{-\omega(z_0)/2} \sum_{k=1}^g \bar{\beta}_k \; .
  \end{align*}
  Thus we obtain from Equation~\eqref{eq:cmc:xi-trace:a'0}
  $$ u_0^2 = a'(0) - v_0\,w_0 - v_{-1}\,w_1 = \frac12 HQ \sum_{k=1}^g \left( \lambda_k^{-1}-\beta_k^{-1} + \bar{\lambda}_k - \bar{\beta}_k \right) \; . $$
  Due to \,$u_0 = \tfrac14 \omega_z(z_0)$\,, Equation~\eqref{eq:cmc:xi-trace:omegaz} follows.
\end{proof}

{\bf The isospectral set.}
Let \,$a(\lambda)$\, be a polynomial of degree \,$2g$\, without higher order roots that satisfies \,$a(0)=-\tfrac12 HQ$\, and the reality condition \eqref{eq:cmc:a-reality}.
Then \,$a(\lambda)$\, defines a spectral curve \,$\Sigma$\, by means of Equation~\eqref{eq:cmc:Sigma}.
We can ask about the set of polynomial Killing fields \,$\zeta \in \mathcal{P}_g$\, that correspond to this spectral curve \,$\Sigma$\,; because of Equation~\eqref{eq:cmc:det-xi}
this condition is equivalent to the equation \,$-\lambda\,\det(\zeta) = a(\lambda)$\,. The set of \,$\zeta \in \mathcal{P}_g$\, that solve this equation
is called the \emph{isospectral set}
$$ I(a) = \{ \zeta\in \mathcal{P}_g \mid -\lambda\,\det(\zeta) = a(\lambda) \} \; . $$
It is known that \,$I(a)$\, is a compact and connected real-$g$-dimensional variety.
There exists an action of \,$(t_0,\dotsc,t_{g-1}) \in \C^g$\, on \,$I(a)$\, which is called the \emph{isospectral flow}. It is infinitesimally given by the vector fields on \,$I(a)$\, described by
\begin{equation}
  \label{eq:cmc:infinitesimal-isospectral-action}
  \frac{\partial \zeta}{\partial t_n} = [(t_n\,\zeta \lambda^{-n})_+,\zeta] = -[(t_n\,\zeta\lambda^{-n})_-,\zeta] \; ,
\end{equation}
see \cite[Section~4]{HKS1}. Here \,$t_n\,\zeta\lambda^{-n} = (t_n\,\zeta\lambda^{-n})_+ + (t_n\,\zeta\lambda^{-n})_-$\, is the Lie algebra decomposition corresponding to the Iwasawa decomposition described in Proposition~\ref{P:cmc:iwasawa}. In more explicit terms, if we write 
\begin{equation}
\label{eq:cmc:zeta-once-again}
\zeta= \sum_{k=-1}^g \zeta_k\,\lambda^k = \sum_{k=-1}^g \left( \begin{matrix} u_k & v_k \\ w_k & -u_k \end{matrix} \right) \,\lambda^k
\end{equation}
as usual, we have
\begin{equation}
\label{eq:cmc:infinitesimal-zeta+}
(t_n\,\zeta\lambda^{-n})_+ = \begin{pmatrix} \RE(t_n\,u_n) & t_n\,v_n \\ 0 & -\RE(t_n\,u_n) \end{pmatrix} + \sum_{k=n+1}^g t_n\,\zeta_k\,\lambda^{k-n}
\end{equation}
and 
$$ (t_n\,\zeta\lambda^{-n})_- = \sum_{k=-1}^{n-1} t_n\,\zeta_k\,\lambda^{k-n} + \begin{pmatrix} \IM(t_n\,u_n)\,\mi & 0 \\ t_n\,w_n & -\IM(t_n\,u_n)\mi \end{pmatrix} \; . $$
Note that the action of \,$\C^g$\, on \,$I(a)$\, in particular preserves the reality condition of the potentials \,$\zeta \in \mathcal{P}_g$\,. The kernel of the action of \,$\C^g$\, on \,$I(a)$\, is real-$g$-dimensional, see the explicit description in \cite[Proposition~4.11]{HKS1}. The transformation group corresponding to the complementary real-$g$-dimensional action on \,$I(a)$\, is called the \emph{isospectral group}. 

Replacing \,$\xi_{z_0}$\, by \,$\zeta \in \mathcal{P}_g$\,, we define \,$u(\lambda)$\,, \,$v(\lambda)$\,, \,$w(\lambda)$\, by Equation~\eqref{eq:cmc:uvw}, 
the number \,$\omega(\zeta) \in \R$\, by \,$v_{-1} = \tfrac12 H e^{\omega(\zeta)}$\,, and the \,$\lambda_k$\,, \,$\beta_k$\,, \,$\nu_k$\, characterising
the associated spectral curve and spectral divisor as before. Then Propositions~\ref{P:cmc:lambdak} and \ref{P:cmc:xi-trace} (except
for Equation~\eqref{eq:cmc:xi-trace:omegaz}) hold in the present situation.

The following proposition describes estimates for \,$e^{\pm \omega(\zeta)}$\, and for the \,$\beta_k$\,. Via Proposition~\ref{P:cmc:xi-trace} these estimates can be used
to estimate the coefficients of polynomial Killing fields \,$\zeta \in I(a)$\,. This will be very important for the construction of our blow-ups. We assume that the zeros \,$\lambda_k$\, of \,$a(\lambda)$\, inside the unit circle are ordered such that
\begin{equation}
  \label{eq:cmc:lambdak-ordering}
  0 < |\lambda_1|\leq \dotsc \leq |\lambda_g| < 1
\end{equation}
holds. 

\begin{proposition}
\label{P:cmc:magic-estimates-new}
In the situation described above, we have for any \,$\zeta \in I(a)$\,:
\begin{enumerate}
	\item 
    \,$\frac{2|Q|}{H} \prod_{k=1}^g |\lambda_k| \leq e^{\omega(\zeta)} \leq \frac{2|Q|}{H} \prod_{k=1}^g |\lambda_k|^{-1}$\,. \\
	In this estimate, the lower resp.~the upper bound for \,$e^{\omega(\zeta)}$\, is attained when \,$u(\zeta)=0$\, and \,$\beta_k=\lambda_k$\, resp.~\,$\beta_k=\bar{\lambda}_k^{-1}$\, holds for all \,$k \in \{1,\dotsc,g\}$\,. 
	\item 
    \,$|\lambda_1| \leq |\beta_k| \leq |\lambda_1|^{-1}$\, for all \,$k \in \{1,\dotsc,g\}$\,.	
\end{enumerate}
\end{proposition}

\begin{proof}
\emph{For (1).}
Note that \,$f: I(a) \to \bbR^+,\; \zeta \mapsto e^{\omega(\zeta)}$\, is a smooth, globally-defined function on \,$I(a)$\,. Because \,$I(a)$\, is compact, \,$f$\, attains its maximum and its minimum on \,$I(a)$\,. These extremal points are critical points of \,$f$\,. We will now show that for any critical point \,$\zeta \in I(a)$\, of \,$f$\,, the matrix \,$\zeta(\lambda)$\, is off-diagonal, i.e.~that with the notation of Equation~\eqref{eq:cmc:zeta-once-again} we have \,$u_n=0$\, for all \,$n\in \{-1,0,\dotsc,g\}$\,. Because \,$u_{-1} = u_g = 0$\, always holds, it suffices to consider \,$n\in \{0,\dotsc,g-1\}$\,. We consider the infinitesimal isospectral action \eqref{eq:cmc:infinitesimal-isospectral-action} of \,$t_n \in \C$\, on \,$I(a)$\, at \,$\zeta$\,. Because the \,$\lambda^{-1}$-component of any potential in \,$\mathcal{P}_g$\, is proportional to \,$f^{1/2}$\,, the \,$\lambda^{-1}$-component of \,$\tfrac{\partial \zeta}{\partial t_n}$\, vanishes at the critical point \,$\zeta$\, of \,$f$\,. Writing \,$t_n = x_n + \mi y_n$\, with \,$x_n, y_n \in \R$\,, we calculate using Equations~\eqref{eq:cmc:zeta-once-again} and \eqref{eq:cmc:infinitesimal-zeta+} that the \,$\lambda^{-1}$-component of \,$\tfrac{\partial \zeta}{\partial x_n} = [(x_n\,\zeta\lambda^{-n})_+,\zeta]$\, is given by
$$ \left[ \begin{pmatrix} \RE(u_n) & v_n \\ 0 & -\RE(u_n) \end{pmatrix} , \begin{pmatrix} 0 & v_{-1} \\ 0 & 0 \end{pmatrix} \right]x_n = \begin{pmatrix} 0 & 2\RE(u_n) v_{-1} \\ 0 & 0 \end{pmatrix}x_n \;, $$
whence \,$\RE(u_n)=0$\, follows. The evaluation of \,$\tfrac{\partial \zeta}{\partial y_n}$\, similarly yields \,$\IM(u_n)=0$\,. 

For any critical point \,$\zeta \in I(a)$\, of \,$f$\, we thus have \,$u(\lambda)=0$\, and therefore
$$ a(\lambda) = -\lambda\,\det(\zeta) = \lambda v(\lambda)\cdot w(\lambda) \; . $$
This shows that for such \,$\zeta \in I(a)$\,, 
the set \,$\{\lambda_k, \bar{\lambda}_k^{-1} \mid 1\leq k\leq g\}$\, of zeros of \,$a(\lambda)$\, is equal to the set
\,$\{\beta_k, \bar{\beta}_k^{-1} \mid 1\leq k\leq g\}$\, of zeros of \,$\lambda v(\lambda) \cdot w(\lambda)$\,. It follows that we have 
$$ \prod_{k=1}^g |\lambda_k| \leq \prod_{k=1}^g |\beta_k| \leq \prod_{k=1}^g |\lambda_k|^{-1} $$
and hence by Equation~\eqref{eq:cmc:xi-trace:omega}
$$ \frac{2|Q|}{H} \prod_{k=1}^g |\lambda_k| \leq e^{\omega(\zeta)} \leq \frac{2|Q|}{H} \prod_{k=1}^g |\lambda_k|^{-1} \; . $$
Because the function \,$f$\, attains its maximum and its minimum at critical points, the latter estimate in fact holds for all \,$\zeta \in I(a)$\,. It also follows that equality holds in the lower resp.~the upper estimate if and only if \,$\beta_k=\lambda_k$\, resp.~\,$\beta_k=\bar{\lambda}_k^{-1}$\, holds for all \,$k$\,.

\emph{For (2).} To prove the lower estimate, we consider the globally defined, continuous function 
$$ f: I(a) \to \R,\; \zeta \mapsto \min \{ \, |\beta_k(\zeta)| \,\bigr| \, k=1,\dotsc,n \} \; . $$
This function again attains its minimum on the compact set \,$I(a)$\,, say at some \,$\zeta_0 \in I(a)$\,. Below we will show that for any continuous function \,$\beta$\, defined locally around \,$\zeta_0$\, which for any \,$\zeta$\, is equal to one of the \,$\beta_k(\zeta)$\, and so that \,$|\beta(\zeta_0)| = f(\zeta_0)$\,, we have \,$u(\beta)=0$\, at \,$\zeta_0$\,. It follows by Equation~\eqref{eq:cmc:det-xi} that \,$a(\beta)=0$\, holds at \,$\zeta_0$\,, and hence \,$\beta(\zeta_0)$\, is one of the values \,$\lambda_k$\,, \,$\bar{\lambda}_k^{-1}$\,, which implies \,$|\beta(\zeta_0)| \geq |\lambda_1|$\, by \eqref{eq:cmc:lambdak-ordering}. Because \,$f$\, attains its minimum at \,$\zeta_0$\,, the latter inequality holds for all zeros \,$\beta$\, of \,$\lambda\,v(\lambda)$\, and at all points \,$\zeta \in I(a)$\,. 

To prove that \,$u(\beta)|_{\zeta_0}=0$\, holds in the situation described above, we first consider the case where the root \,$\beta(\zeta_0)$\, of \,$\lambda\,v(\lambda)$\, is simple. Then the function \,$\beta$\, is smooth near \,$\zeta_0$\,, and \,$|\beta|$\, has a critical point at \,$\zeta_0$\,. Hence, if we denote by \,$\tfrac{\partial}{\partial t_n}$\, again the isospectral flow \eqref{eq:cmc:infinitesimal-isospectral-action}, then \,$\tfrac{\partial\beta}{\partial t}|_{\zeta=\zeta_0}$\, is a real multiple of \,$\mi\beta(\zeta_0)$\, for any \,$t = (t_0,\dotsc, t_{g-1}) \in \bbC^g$\,. 

Because \,$\lambda\,v(\lambda)$\, has only a first order root at \,$\lambda=\beta$\,, we have \,$\tfrac{\partial v}{\partial\lambda}|_{\lambda=\beta} \neq 0$\,, and therefore the implicit function theorem gives
\begin{equation}
\label{eq:cmc:magic-estimates-new:implicit}
\left.\frac{\partial \beta}{\partial t_n}\right|_{\zeta=\zeta_0} = -\left( \left.\frac{\partial v}{\partial \lambda}\right|_{\lambda=\beta} \right)^{-1} \cdot \left.\frac{\partial v}{\partial t_n}\right|_{\zeta=\zeta_0} \; .
\end{equation}
By Equations~\eqref{eq:cmc:infinitesimal-isospectral-action} and \eqref{eq:cmc:infinitesimal-zeta+} we have for any \,$n\in \{0,\dotsc,g-1\}$\,
\begin{equation*}
	\left. \frac{\partial v}{\partial t_n}\right|_{\lambda=\beta} = -2 \left(\sum_{k=n}^g v_k\,\beta^{k-n} \right) u(\beta)\,t_n 
\end{equation*}
and in particular for \,$n=0$\, due to \,$v(\beta)=0$\,
$$
	\left. \frac{\partial v}{\partial t_0}\right|_{\lambda=\beta} = 2 v_{-1}\,\beta^{-1}\, u(\beta)\,t_0 \; . 
$$
Equation~\eqref{eq:cmc:magic-estimates-new:implicit} thus implies
$$
\left.\frac{\partial \beta}{\partial t_0}\right|_{\zeta=\zeta_0} = -2\left( \left.\frac{\partial v}{\partial \lambda}\right|_{\lambda=\beta} \right)^{-1} \cdot v_{-1} \beta^{-1}\,u(\beta) t_0 \; . 
$$
In the product on the right-hand side, all factors except for \,$u(\beta)$\, are known to be non-zero. Under the assumption \,$u(\beta)|_{\zeta_0} \neq 0$\,, 
this equation would show that by choosing \,$t_0\in \C^\times$\, with arbitrary phase, the value of \,$\beta$\, can change under the isospectral flow at \,$\zeta_0$\, in arbitrary directions in the complex plane, which contradicts the previous observation that \,$\tfrac{\partial\beta}{\partial t_0}$\, is a real multiple of \,$\mi\beta(\zeta_0)$\,. Therefore we have \,$u(\beta)|_{\zeta_0} = 0$\,. 

However, it is possible that \,$\lambda v(\lambda)$\, has at \,$\lambda=\beta$\, a root of order \,$d \geq 2$\,, and in this case we need a more refined argument. We now denote the global isospectral flow defined by \eqref{eq:cmc:infinitesimal-isospectral-action} by \,$\zeta(t_0,\dotsc,t_{g-1})$\, with \,$\zeta(0,\dotsc,0)=\zeta_0$\,. 
Under the assumption that \,$u(\beta)|_{\zeta_0} \neq 0$\, holds, we can again choose a tangent direction \,$\hat{t}=(\hat{t}_0,\dotsc,\hat{t}_{g-1}) \in \C^g$\, of the isospectral flow so that 
\begin{equation}
	\label{eq:cmc:magic-estimates-new:non-trivial}
	\left.\frac{\partial v}{\partial \hat{t}} \right|_{\zeta=\zeta_0, \lambda=\beta} \neq 0 \; . 
\end{equation}
We now define the integral curve
$$ \R \to I(a),\; s \mapsto \zeta(s) := \zeta(s^d\,\hat{t}_0,\dotsc,s^d\,\hat{t}_{g-1}) \, . $$
Note that \,$\zeta(0)=\zeta_0$\,. In the sequel, we will consider \,$\zeta$\, and its component functions as functions in \,$s$\,. We shall show that \,$\beta_0 := \beta(\zeta_0)$\, can be extended to a smooth function of roots of \,$\lambda\,v(\lambda)$\, along the integral curve \,$\zeta(s)$\,. 
The translated polynomial \,$(\lambda-\beta_0)\,v(\lambda-\beta_0)$\, can be decomposed in the form 
$$ (\lambda-\beta_0)\,v(\lambda-\beta_0) = P_s(\lambda) \cdot R_s(\lambda) \;, $$
where \,$P_s(\lambda)$\, and \,$R_s(\lambda)$\, are polynomials of order \,$d$\, and \,$(g+1)-d$\,, respectively, \,$P_{s=0}(\lambda)$\, has a zero of degree \,$d$\, at \,$\lambda=0$\,, and \,$R_{s=0}(0) \neq 0$\, holds. The decomposition is made unique by the additional stipulation that the leading coefficient of \,$P_s(\lambda)$\, is constant \,$1$\,. We then have \,$P_{s=0}(\lambda) = \lambda^d$\, and more generally
\begin{equation}
\label{eq:cmc:magic-estimates-new:PQ}
P_s(\lambda) = \lambda^d + s^d \cdot Q_s(\lambda)
\end{equation}
with a polynomial \,$Q_s(\lambda)$\, of degree \,$d-1$\,. Because of \eqref{eq:cmc:magic-estimates-new:non-trivial} we have \,$Q_s(0) \neq 0$\, for small \,$|s|$\,. 

We now blow up these polynomials by introducing the new variable \,$\tilde{\lambda} = s^{-1}\,\lambda$\, and defining new polynomials \,$\tilde{P}_s(\tilde{\lambda})$\, and \,$\tilde{Q}_s(\tilde{\lambda})$\, of degree \,$d$\, and \,$d-1$\, respectively by 
$$ s^d \, \tilde{P}_s(\tilde{\lambda}) = P_s(s\tilde{\lambda})
\quad\text{and}\quad 
\tilde{Q}_s(\tilde{\lambda}) = Q_s(s\tilde{\lambda}) \; . $$
From Equation~\eqref{eq:cmc:magic-estimates-new:PQ} we then obtain
$$ \tilde{P}_s(\tilde{\lambda}) = \tilde{\lambda}^d + \tilde{Q}_s(\tilde{\lambda}) \; .  $$
Because of \,$\tilde{Q}_s(0) = Q_s(0)\neq 0$\,, we see that for small values of \,$|s|$\,, the polynomial \,$\tilde{P}_s(\tilde{\lambda})$\, has \,$d$\, distinct zeros. In particular the location of these zeros depends smoothly on \,$s$\,. It follows that along the integral curve \,$\zeta(s)$\,, the \,$d$\, roots of \,$\lambda\,v(\lambda)$\, at \,$\beta_0$\, move into \,$d$\, distinct roots, and the initial directions of their motion is colinear to the \,$d$-th roots of \,$-Q_s(0)$\,. For \,$d \geq 3$\, there are at least three different directions, so this is an immediate contradiction to the observation that \,$\tfrac{\partial \beta}{\partial \hat{t}}|_{\zeta=\zeta_0}$\, is a real multiple of \,$\mi\beta_0$\,. For \,$d=2$\,, we can vary the phase of \,$Q_s(0)$\, and hence of the direction of motion of \,$\beta$\, by varying the phase of \,$\hat{t}$\,, which again produces a contradiction. 

To show the upper estimate in (2), we apply the analogous arguments to the continuous function \,$f: I(a)\to \R,\; \zeta \mapsto \max\{|\beta_k|\}$\, and the \,$\zeta_0 \in I(a)$\, where this \,$f$\, attains its maximum. 
\end{proof}

\section{Polynomial Killing fields and Symes' method for minimal surface immersions}
\label{Se:KdV}

We now develop the analogue of polynomial Killing fields and Symes' method for a different integrable system, which is associated to minimal surface immersions into \,$\R^3$\, locally around non-umbilical points. This integrable system belongs to Liouville's equation, and it will turn out that it is closely related to the integrable system of the complex KdV equation. 
As we will see in Section~\ref{Se:blowup1}, it will occur as the blow-up of a sequence of solutions of the sinh-Gordon equation under certain circumstances.

{\bf Minimal immersions into 3-space.}
Again let \,$f: X \to \R^3$\, with \,$X\subset \C$\, be a conformal surface immersion into \,$\R^3$\,. In relation to \,$f$\, we again use the notations introduced
at the beginning of Section~\ref{Se:cmc}, but now assume that \,$f$\, is minimal, meaning that its mean curvature \,$H$\, vanishes. We again consider \,$f$\, near a
non-umbilical point and assume that the coordinate \,$z$\, on \,$X$\, is chosen such that the function \,$Q$\, describing the Hopf differential \,$Q\,\diff z^2$\,
is constant and non-zero. In this setting the Codazzi equation~\eqref{eq:cmc:codazzi} again reduces to \,$0=0$\,, whereas the Gauss equation~\eqref{eq:cmc:gauss}
means that the negative \,$-\omega$\, of the conformal factor of \,$f$\, is a solution of Liouville's equation
\begin{equation}
  \label{eq:KdV:liouville}
  -\omega_{z\bar{z}} = -2\,|Q|^2\,e^{-\omega} \; .
\end{equation}
To define an extended frame for \,$f$\,, we again consider a family \,$\alpha_\lambda$\, of connection forms with respect to a \emph{spectral parameter} \,$\lambda\in \C$\,
(unlike in the cmc case, \,$\lambda=0$\, is now permitted):
\begin{equation}
  \label{eq:KdV:alpha}
  \alpha = \alpha_\lambda = \frac14 \begin{pmatrix} \omega_z & e^{\omega/2}\,\lambda \\ -4Q\,e^{-\omega/2} & -\omega_z \end{pmatrix} \diff z
  + \frac14 \begin{pmatrix} -\omega_{\bar z} & 4\bar{Q}\,e^{-\omega/2} \\ 0 & \omega_{\bar z} \end{pmatrix} \diff \bar{z}
\end{equation}
Note that in comparison to the cmc connection form \eqref{eq:cmc:alpha}, \,$\lambda^{-1}$\, has been replaced by \,$\lambda$\,. This has been done so that
\,$\alpha_\lambda$\, can be defined for \,$\lambda \in \C$\, (rather than \,$\lambda \in \C^{\times}  \cup \{\infty\}$\,) and also so that the Sym point (see below)
is at \,$\lambda_s = 0$\, (rather than \,$\lambda_s=\infty$\,, which would complicate the Sym-Bobenko formula). Again one checks that Liouville's equation
\eqref{eq:KdV:liouville} is equivalent to the Maurer-Cartan equation for \,$\alpha_\lambda$\,. Thus the initial value problem
$$ \diff F_\lambda = F_\lambda\,\alpha_\lambda \quad\text{with}\quad F_\lambda(z_0)=\one $$
has for every \,$\lambda \in \C$\, a unique solution \,$F_\lambda: X \to \mathrm{SL}(2,\C)$\,, called the \emph{extended frame} of \,$f$\,. Again, \,$F_\lambda$\,
depends holomorphically on \,$\lambda$\,, and due to \,$\alpha_{\lambda=0} \in \mathfrak{su}(2)$\,, \,$F_{\lambda=0}$\, maps into \,$\mathrm{SU}(2)$\,.
Because this reality condition holds only for a discrete subset of \,$\lambda \in \C$\,, there is no formula that is analogous to the right-hand side
equation of \eqref{eq:cmc:alpha-F-reality} in the present situation.

However there is a variant of the Sym-Bobenko formula in this setting, given in Proposition~\ref{P:KdV:sym-bobenko} below. As we will see in Section~\ref{Se:blowup1}, it also arises as the limit of a blow-up of the Sym-Bobenko formula \eqref{eq:cmc:sym-bobenko} for cmc surfaces. 
Both Sym-Bobenko formulae recover an immersion \,$f$\, as a scalar multiple of \,$G\,F^{-1}$\,, evaluated at a Sym point, where \,$F$\, is the extended frame and \,$G$\, is a solution of a partial differential equation of the form \,$\mathrm{d}G=G\alpha + F\beta$\,. In the cmc case we have \,$\beta = \tfrac{\partial \alpha}{\partial \lambda}$\,, and the choice of the phase \,$\varphi$\, of the Sym point \,$\lambda_s$\, corresponds to the choice of a member of the associated family of cmc surfaces. 
The aforementioned blow-up is effected by blowing up the potential \,$\zeta$\, and the spectral parameter \,$\lambda$\, about \,$\lambda=0$\, by rescaling with a positive real blow-up factor
in such a manner that the phase \,$\varphi$\, of the Sym point is preserved. 
After the replacement of the blown-up \,$\lambda$\, by its inverse as described above, the cmc Sym point \,$\lambda_s \in S^1$\, goes to \,$0$\, in the blow-up. In this process, the information of the phase \,$\varphi$\, of \,$\lambda_s$\, is obliterated. 

It turns out that the blow-up limit of \,$\tfrac{\partial \alpha}{\partial \lambda}$\, at the Sym point depends on the phase \,$\varphi$\,, in particular the blow-up limit does not commute with the action of the differential operator \,$\tfrac{\partial\ }{\partial \lambda}$\, on \,$\alpha$\,. This is the reason why in the Sym-Bobenko formula \eqref{eq:KdV:sym-bobenko} for minimal surfaces the Sym point is now fixed at \,$\lambda=0$\,, but there is an additional choice of a phase \,$\varphi$\, to be made, which now corresponds to the choice of a member of the associated family. This phenomenon also explains why in the minimal case, 
\,$\beta$\, \eqref{eq:KdV:sym-bobenko:beta} is not equal
to \,$\tfrac{\partial \alpha}{\partial \lambda}$\,. 
  
\begin{proposition}
  \label{P:KdV:sym-bobenko}
  Let a real-valued, smooth function \,$\omega$\, and a constant \,$Q \in \C^{\times} $\, be given so that these data satisfy Liouville's equation~\eqref{eq:KdV:liouville}.
  Moreover let \,$\varphi \in \R$\, and consider the \,$\mathfrak{sl}(2,\C)$-valued connection form
  \begin{equation}
  \label{eq:KdV:sym-bobenko:beta}
  \beta_\lambda = e^{\omega/2}\, \frac{\mi}{2}\begin{pmatrix} 0 & e^{-\mi\varphi}\,\diff z  \\ e^{\mi\varphi}\,\diff \bar{z} & 0 \end{pmatrix} 
  \end{equation}
  and a solution \,$G=G(z,\lambda)$\, of the inhomogeneous linear ordinary 
  differential equation \,$\diff G = G\alpha + F\beta$\, with \,$G(z_0,\lambda=0)\in \mathfrak{su}(2)$\,. 
  Then 
  \begin{equation}
  \label{eq:KdV:sym-bobenko}
  f = G\,F^{-1} \biggr|_{\lambda=0}
  \end{equation}
  is a minimal immersion \,$X \to \mathfrak{su}(2) \cong \R^3$\, with induced metric \,$e^\omega\,\diff z \,\diff \bar{z}$\, and
  Hopf differential \,$e^{-\mi\varphi} Q\,\diff z^2$\,. 
  The tangential directions \,$f_x, f_y$\, of \,$f$\, and the unit normal field \,$N$\, of \,$f$\, are again given by Equation~\eqref{eq:cmc:sym-bobenko:tangential}
  with \,$\lambda_s=e^{\mi\varphi}$\, and Equation~\eqref{eq:cmc:sym-bobenko:normal} respectively.
\end{proposition}

\begin{proof}
  The proof is similar to the proof of Proposition~\ref{P:cmc:sym-bobenko}. We again have
  $$ \diff f = \diff G \cdot F^{-1} - G \, F^{-1}\,\diff F \, F^{-1} = G\,\alpha\, F^{-1} + F\,\beta\, F^{-1} - G\, F^{-1}\,F\,\alpha\,F^{-1} = F\,\beta\,F^{-1} \; . $$
  For \,$\lambda=0$\, we have \,$F(z,\lambda=0)\in \mathrm{SU}(2)$\, and \,$\beta_{\lambda=0}\in\mathfrak{su}(2)$\,, hence \,$\diff f$\, maps into \,$\mathfrak{su}(2)$\,.
  Because we also have \,$f(z_0) = G(z_0,\lambda=0) \in \mathfrak{su}(2)$\,, it follows that \,$f$\, indeed maps into \,$\mathfrak{su}(2) \cong \R^3$\,.
  The statements on the tangential and normal directions of \,$f$\, also follow from the formula \,$\diff f = F\,\beta\,F^{-1}$\,. Moreover we have
  \begin{equation}
    \label{eq:KdV:sym-bobenko:fzfbarz}
    f_z = e^{\omega/2}\, e^{-\mi\varphi}\, F \tfrac{\mi}{2} \left( \begin{smallmatrix} 0 & 1 \\ 0 & 0 \end{smallmatrix} \right) F^{-1}
    \quad\text{and}\quad
    f_{\bar{z}} = e^{\omega/2}\, e^{\mi\varphi} F \tfrac{\mi}{2} \left( \begin{smallmatrix} 0 & 0 \\ 1 & 0 \end{smallmatrix} \right) F^{-1} \;,
  \end{equation}
  and therefore
  $$ 2\langle f_z,f_{\bar z} \rangle = 2e^{\omega} \cdot \left\langle \tfrac{\mi}{2} \left( \begin{smallmatrix} 0 & 1 \\ 0 & 0 \end{smallmatrix} \right),
  \tfrac{\mi}{2} \left( \begin{smallmatrix} 0 & 0 \\ 1 & 0 \end{smallmatrix} \right) \right\rangle = e^\omega \;, $$
  whence it follows that \,$f$\, is a conformal immersion with the induced metric \,$e^\omega\,\diff z \,\diff \bar{z}$\,. 
  We moreover obtain from Equations~\eqref{eq:KdV:sym-bobenko:fzfbarz} and \eqref{eq:KdV:alpha}:
  \begin{align*}
    f_{zz} & = \tfrac12 \omega_z\,f_z + [FUF^{-1}, f_z] = \omega_z\,f_z + e^{-\mi\varphi}\,Q\,N \\
    f_{z\bar{z}} & = \tfrac12 \omega_{\bar{z}}\,f_z + [FVF^{-1}, f_z] = 0 \\
    f_{\bar{z}\bar{z}} & = \tfrac12 \omega_{\bar{z}}\,f_{\bar{z}} + [FVF^{-1}, f_{\bar{z}}] = \omega_{\bar{z}}\,f_{\bar{z}} + e^{\mi\varphi}\,\bar{Q}\,N \; .
  \end{align*}
  Thus the Hopf differential of \,$f$\, is given by \,$\langle f_{zz},N \rangle \diff z^2 = e^{-\mi\varphi}\,Q\,\diff z^2$\, and the mean curvature of \,$f$\, vanishes.
\end{proof}

{\bf Polynomial Killing fields for minimal immersions.}
Analogously as for the cmc case, we consider polynomial Killing fields for the integrable system of minimal surface immersions of finite type \,$g\in \N$\,.
For this integrable system, the space of potentials is
$$ \mathcal{P}_g^{\Liouville} = \left\{ \left. \zeta_\lambda = \sum_{k=1}^{-g} \zeta_k\lambda^k \right| \zeta_k \in \mathfrak{sl}(2,\C), \; \zeta_{1} \in \left( \begin{smallmatrix} 0 & \R^+ \\ 0 & 0 \end{smallmatrix} \right), \mathrm{tr}(\zeta_{1}\,\zeta_0) \neq 0 \right\} \; . $$
By analogy with the situation for the sinh-Gordon integrable system, we call the members of \,$\mathcal{P}_g^{\Liouville}$\, \emph{Liouville potentials}.
However note that in contrast to that case, these polynomial Killing fields minus their ``initial term'' \,$\zeta_1 \lambda$\,
are polynomials in \,$\lambda^{-1}$\,, not in \,$\lambda$\,.
This corresponds to the substitution of \,$\lambda$\, with \,$\lambda^{-1}$\, in the expression for \,$\alpha_\lambda$\, in Equation~\eqref{eq:KdV:alpha}. 
As in the cmc case however, the dependence of the polynomial Killing field \,$\xi = \xi_{z_0}$\, for a solution \,$\omega$\, of Liouville's equation \eqref{eq:KdV:liouville}
on the base point \,$z_0$\, is described by the differential equation
\begin{equation}
  \label{eq:KdV:xi-dgl}
  \diff \xi + [\alpha_\lambda,\xi] = 0 \;,
\end{equation}
where \,$\alpha_\lambda$\, is now the connection form given by Equation~\eqref{eq:KdV:alpha}, and we impose the ``initial condition''
\,$v_1 = \tfrac14 e^{\omega/2}$\, for \,$\xi = \left( \begin{smallmatrix} 0 & v_1 \\ 0 & 0 \end{smallmatrix} \right) \lambda + \dotsc$\,.
The differential equation~\eqref{eq:KdV:xi-dgl} can be expressed as an iterative algorithm for the computation of the coefficients of the polynomial Killing field \,$\xi$\,
which is analogous to the Pinkall-Sterling iteration in the cmc case:

\begin{proposition}
  \label{P:KdV:ps}
  Let a solution \,$(\omega,Q)$\, of Liouville's equation~\eqref{eq:KdV:liouville} with a smooth real-valued function \,$\omega$\, and a constant \,$Q\in \C^{\times} $\,
  be given. We suppose that this solution is of finite type \,$g$\, and write the corresponding polynomial Killing field \,$\xi=\xi_z$\, in the form
  $$ \xi = \sum_{k=1}^{-g} \xi_k\,\lambda^k \quad\text{with}\quad \xi_k = \begin{pmatrix} u_k & \tau_k\,e^{\omega/2} \\ \sigma_k\,e^{\omega/2} & -u_k \end{pmatrix} \;, $$
  where \,$u_k, \tau_k, \sigma_k$\, are smooth, complex-valued functions in \,$z$\,. Then we have \,$\tau_{1} = \tfrac14$\,, \,$u_{1} = \sigma_{1} = 0$\,
  and for every \,$k= 0,\dotsc,-g$\,: 
  \begin{align}
    \label{eq:KdV:ps:ps-tau}
    \tau_{k,z} = -\tfrac{1}{2Q} (u_{k,zz} - \omega_{z} u_{k,z}) \qquad & \qquad \tau_{k,\bar{z}} = 2\bar{Q}\,e^{-\omega}\,u_k \\
    \label{eq:KdV:ps:ps-u}  
    u_{k-1} & = 2(\tau_{k,z}+\omega_{z}\,\tau_k) \\
    \label{eq:KdV:ps:ps-sigma}  
    \sigma_{k-1} & = -\tfrac{1}{\bar{Q}} u_{k-1,\bar{z}} \; . 
  \end{align}
  Moreover every \,$u_k$\, solves the linearisation of Liouville's equation~\eqref{eq:KdV:liouville}:
  \begin{equation}
    \label{eq:KdV:ps:uk-lin-liouville}
    u_{k,z\bar{z}} + 2|Q|^2 e^{-\omega} u_k = 0 \; .
  \end{equation}
\end{proposition}

\begin{proof}
  The proof is analogous to the one of Proposition~\ref{P:cmc:ps}. We write
  $$ \alpha_\lambda = (U_0 + U_1\,\lambda)\mathrm{d}z + V_0\,\mathrm{d}\bar{z} $$
  with
  $$ U_0 = \frac14 \begin{pmatrix} \omega_{z} & 0 \\ -4Qe^{-\omega/2} & -\omega_z \end{pmatrix}\;,\quad
  U_1 = \frac14 \begin{pmatrix} 0 & e^{\omega/2} \\ 0 & 0 \end{pmatrix} \AND
  V_0 = \frac14 \begin{pmatrix} -\omega_{\bar{z}} & 4\bar{Q}e^{-\omega/2} \\ 0 & \omega_{\bar{z}} \end{pmatrix} \; . $$
  
We separate the differential equation \eqref{eq:cmc:xi-dgl} into its \,$\mathrm{d}z$-part and its \,$\mathrm{d}\bar{z}$-part,
and also into the individual powers of \,$\lambda$\, that occur.
In this way we obtain the equations
\begin{align*}
  \xi_{k,z} + [U_0,\xi_{k}] + [U_{1},\xi_{k-1}] & = 0 \\
  \xi_{k,\bar{z}} + [V_0,\xi_k] & = 0
\end{align*}
for all \,$k\in\{1,\dotsc,-g\}$\,. By evaluating the brackets and separating the entries of the matrices, we obtain the following equations:
\begin{align}
  \label{eq:KdV:ps:ps1-ukz}
  u_{k,z} + Q \tau_k + \tfrac14 e^{\omega}\sigma_{k-1} & = 0 \\
  \label{eq:KdV:ps:ps1-ukbz}  
  u_{k,\bar{z}} + \bar{Q} \sigma_k & = 0 \\
  \label{eq:KdV:ps:ps1-taukz}
  \tau_{k,z} + \omega_z\,\tau_k - \tfrac12 u_{k-1} & = 0 \\
  \label{eq:KdV:ps:ps1-taukbz}
  e^{\omega/2}\tau_{k,\bar{z}} - 2\bar{Q}e^{-\omega/2}u_k & = 0 \\
  \label{eq:KdV:ps:ps1-sigmakz}
  e^{\omega/2}\sigma_{k,z} - 2Qe^{-\omega/2}u_k & = 0 \\
  \label{eq:KdV:ps:ps1-sigmakbz}
  \sigma_{k,\bar{z}}+\omega_{\bar{z}} \sigma_k & = 0 \; .
\end{align}
The right-hand equation of \eqref{eq:KdV:ps:ps-tau} follows from Equation~\eqref{eq:KdV:ps:ps1-taukbz}, Equation~\eqref{eq:KdV:ps:ps-u} follows from
Equation~\eqref{eq:KdV:ps:ps1-taukz}, and Equation~\eqref{eq:KdV:ps:ps-sigma} follows from Equation~\eqref{eq:KdV:ps:ps1-ukbz}. 


By differentiating \eqref{eq:KdV:ps:ps1-ukbz} by \,$z$\,, and by \eqref{eq:KdV:ps:ps1-sigmakz}, we get two different expressions for \,$\sigma_{k,z}$\,:
$$ -\tfrac{1}{\bar{Q}} u_{k,z\bar{z}} = \sigma_{k,z} = 2Qe^{-\omega}\,u_k $$
and this equation implies \eqref{eq:KdV:ps:uk-lin-liouville}.

By differentiating Equation~\eqref{eq:KdV:ps:ps1-ukz} with respect to \,$z$\, we get
\begin{equation}
  \label{eq:KdV:ps:ps-taukz-pre}
  u_{k,zz} + Q \tau_{k,z} + \tfrac14 e^{\omega}(\sigma_{k-1,z}+\omega_{z}\,\sigma_{k-1}) = 0 \; .
\end{equation}
Equation~\eqref{eq:KdV:ps:ps1-ukz} implies
\begin{equation}
  \label{eq:KdV:ps:ps-taukz-pre1}
  \tfrac14 e^{\omega}\sigma_{k-1} = -u_{k,z}- Q \tau_k
\end{equation}
and Equations~\eqref{eq:KdV:ps:ps1-sigmakz} (for \,$k-1$\,) and \eqref{eq:KdV:ps:ps1-taukz} imply
\begin{equation}
  \label{eq:KdV:ps:ps-taukz-pre2}
  \sigma_{k-1,z} = 2Q e^{-\omega} u_{k-1} = 4Q e^{-\omega} (\tau_{k,z}+\omega_{z} \,\tau_k) \; .
\end{equation}
By inserting Equations~\eqref{eq:KdV:ps:ps-taukz-pre1} and \eqref{eq:KdV:ps:ps-taukz-pre2} into Equation~\eqref{eq:KdV:ps:ps-taukz-pre}, and solving for \,$\tau_{k,z}$\,,
we obtain the equation on the left-hand side of \eqref{eq:KdV:ps:ps-tau}.
\end{proof}

The iteration described in Proposition~\ref{P:KdV:ps} again permits us to write down explicitly the lowest terms of the polynomial Killing field \,$\xi$\, corresponding to
some solution \,$\omega$\, of Liouville's equation in terms of \,$\omega$\, and its derivatives. There is again an integration constant \,$C_k$\,
associated to every \,$\tau_k$\,, but now one starts with the largest value for \,$k$\, and then uses Equations~\eqref{eq:KdV:ps:ps-tau}--\eqref{eq:KdV:ps:ps-sigma}
to work downwards. In this way, one finds
\begin{equation}
  \label{eq:KdV:xi-terms}
  \xi = \begin{pmatrix} 0 & \tfrac14 e^{\omega/2} \\ 0 & 0 \end{pmatrix}\lambda
  + \begin{pmatrix} \tfrac12 \omega_z & -e^{\omega/2}(\tfrac{1}{4Q}(\omega_{zz}-\tfrac12 \omega_z^2) - C_0) \\ -Qe^{-\omega/2} & -\tfrac12 \omega_z \end{pmatrix}
  + \dotsc \; . 
\end{equation}
By comparison with \eqref{eq:KdV:alpha} we see that for the polynomial Killing field \,$\xi$\, corresponding to some solution \,$\omega$\, of
Liouville's equation \eqref{eq:KdV:liouville} we have \,$\alpha_\lambda = \alpha_\lambda^{\Liouville}(\xi)$\,, where \,$\alpha_\lambda$\, is defined by
Equation~\eqref{eq:KdV:alpha} with this \,$\omega$\,, and \,$\alpha_\lambda^{\Liouville}(\xi)$\, is defined by the linear map
\begin{equation}
  \label{eq:KdV:alpha-pkf}
  \mathcal{P}_g^{\Liouville} \to \Omega^1(\C) \otimes \mathfrak{sl}(2,\C), \; \zeta \mapsto \alpha^{\Liouville}_\lambda(\zeta) := \begin{pmatrix} \tfrac12 u_0 & v_{1}\,\lambda \\ w_0 & -\tfrac12 u_0 \end{pmatrix} \diff z
  - \begin{pmatrix} \tfrac12 \bar{u}_0 & \bar{w}_0 \\ 0 & -\tfrac12 \bar{u}_0 \end{pmatrix} \diff \bar{z} \; .
\end{equation}
Thus we see that similarly to the cmc situation, any polynomial Killing field \,$\xi=\xi_z$\, of degree \,$g$\, solves the
differential equation  
\begin{equation}
  \label{eq:KdV:xi-dgl-2}
  \diff \xi + [\alpha_\lambda^{\Liouville}(\xi),\xi] = 0 
\end{equation}
with regards to the dependence on the base point \,$z$\,.
For such a solution, say \,$\xi = \left( \begin{smallmatrix} 0 & v_1 \\ 0 & 0 \end{smallmatrix} \right)\lambda + \dotsc$\,, the corresponding solution 
of Liouville's equation~\eqref{eq:KdV:liouville} is the real-valued function \,$\omega$\, with \,$\tfrac14 e^{\omega/2} = v_1$\,. 

{\bf Reconstruction of the minimal immersion from the polynomial Killing field.}
It will turn out that by a variation of Symes' method, this kind of polynomial Killing field gives rise to extended frames of minimal surface immersions into \,$\R^3$\,,
from which the immersions themselves can be obtained by the variant of the Sym-Bobenko formula given in Proposition~\ref{P:KdV:sym-bobenko}.

In this process the Iwasawa decomposition is replaced by the following variant of the Birkhoff decomposition.
We consider the following loop subgroups of \,$\Lambda \mathrm{SL}(2,\C)$\, 
\begin{align*}
  \Lambda^+ \mathrm{SL}(2,\C) & = \{ \Phi \in \Lambda \mathrm{SL}(2,\C) \mid \text{\,$\Phi$\, extends analytically to \,$\C$\,}\} \\
  \lambda^{+,\mathrm{SU}(2)} \mathrm{SL}(2,\C) & = \{ \Phi \in \Lambda \mathrm{SL}(2,\C) \mid \Phi|_{\lambda=0} \in \mathrm{SU}(2) \} \\
  \Lambda^- \mathrm{SL}(2,\C) & = \{ \Phi \in \Lambda \mathrm{SL}(2,\C) \mid \text{\,$\Phi$\, extends analytically to \,$\mathbb{P}^1 \setminus \{0\}$\,}\} \\
  \Lambda^{-,\one} \mathrm{SL}(2,\C) & = \{ \Phi \in \Lambda^- \mathrm{SL}(2,\C) \mid \text{\,$\Phi|_{\lambda=\infty} = \one$\,}\} \\
  \Lambda^{-,\R} \mathrm{SL}(2,\C) & = \{ \Phi \in \Lambda^- \mathrm{SL}(2,\C) \mid \text{\,$\Phi|_{\lambda=\infty} = \left( \begin{smallmatrix} \rho & c \\ 0 & \rho^{-1} \end{smallmatrix} \right) $\, with some \,$\rho>0$\, and \,$c\in\C$\,}\} \; . 
\end{align*}

\begin{proposition}[Modified Birkhoff decomposition]
  \label{P:KdV:modified-birkhoff}
  The map
  $$ \Lambda^{+,\mathrm{SU}(2)} \mathrm{SL}(2,\C) \times \Lambda^{-,\R}\mathrm{SL}(2,\C) \to \Lambda \mathrm{SL}(2,\C),\; (F,B) \mapsto F\cdot B $$
  is a real analytic diffeomorphism onto an open and dense subset \,$\mathcal{U}$\, of \,$\Lambda \mathrm{SL}(2,\C)$\,,
  called the \emph{big cell} of \,$\Lambda \mathrm{SL}(2,\C)$\,.
\end{proposition}

\begin{proof}
  The usual loop group Birkhoff decomposition \cite[Chapter 8]{PreS}
  shows the existence of the big cell \,$\mathcal{U}$\, of \,$\Lambda \mathrm{SL}(2,\C)$\,
  along with the complex analytic diffeomorphism
  $$ \Lambda^{+} \mathrm{SL}(2,\C) \times \Lambda^{-,\one}\mathrm{SL}(2,\C) \to \Lambda \mathrm{SL}(2,\C),\; (g_+,g_-) \mapsto g_+\cdot g_- $$
  onto \,$\mathcal{U}$\,. Now let \,$\Phi \in \mathcal{U} \subset \Lambda \mathrm{SL}(2,\C)$\, be given, and let \,$\Phi=g_+\cdot g_-$\, with
  \,$g_+ \in \Lambda^{+} \mathrm{SL}(2,\C)$\, and \,$g_-\in \Lambda^{-,\one}\mathrm{SL}(2,\C)$\, be the usual Birkhoff decomposition of \,$\Phi$\,.
  By the classical Iwasawa decomposition of the complex, semi-simple Lie group \,$\mathrm{SL}(2,\C)$\,, \,$g_+(0) \in \mathrm{SL}(2,\C)$\, can be decomposed as
  \,$g_+(0) = ub$\, with \,$u\in \mathrm{SU}(2)$\, and \,$b = \left( \begin{smallmatrix} \rho & c \\ 0 & \rho^{-1} \end{smallmatrix} \right)$\, where
  \,$\rho>0$\, and \,$c \in \C$\,. Here \,$u$\, and \,$b$\, depend real analytically on \,$g_+(0)$\,. We now define
  \,$F = g_+ \cdot b^{-1}$\, and \,$B = b\cdot g_-$\,. We then have
  $$ F \in \Lambda^+ \mathrm{SL}(2,\C) \quad\text{with}\quad F(0) = g_+(0) \cdot b^{-1} = u\in \mathrm{SU}(2)\;,\quad\text{so}\quad F \in \Lambda^{+,\mathrm{SU}(2)} \mathrm{SL}(2,\C) $$
  and
  $$ B \in \Lambda^- \mathrm{SL}(2,\C) \quad\text{with}\quad B(\infty) = b \cdot g_-(\infty) = b\;,\quad\text{so}\quad B \in \Lambda^{-,\R} \mathrm{SL}(2,\C) $$
  and
  $$ F \cdot B = g_+\,b^{-1}\,b\, g_- = g_+\,g_- = \Phi \; . $$
  This proves the proposition.
\end{proof}

\begin{proposition}
  \label{P:KdV:symes}
  Let a polynomial Killing field \,$\zeta \in \mathcal{P}_g^{\Liouville}$\, be given. We fix a base point \,$z_0$\, and restrict \,$z$\, to the open subset \,$\mathcal{U}_{\C}$\, of \,$\C$\, for which \,$\Phi_\lambda(z) = \exp\bigr( (z-z_0)\zeta \bigr)$\, is in the big cell of the modified Birkhoff decomposition (Proposition~\ref{P:KdV:modified-birkhoff}).
  Then we let \,$\Phi_\lambda(z) = F_\lambda(z)\cdot B_\lambda(z)$\,
  be the modified Birkhoff decomposition of \,$\Phi(z)$\,.
  In this situation we have \,$\diff F_\lambda = F_\lambda\,\alpha_\lambda$\,, and thus the formula
  \eqref{eq:KdV:sym-bobenko} with this \,$F_\lambda$\, gives a conformal minimal immersion \,$f$\, into \,$\R^3$\, with conformal metric \,$e^\omega\,\diff z \,\diff \bar{z}$\,.
  Here \,$\omega$\, is the real-valued function characterised by \,$\xi = \left( \begin{smallmatrix} 0 & \tfrac14 e^{\omega/2} \\ 0 & 0 \end{smallmatrix} \right)\lambda + \dotsc$\,
  for the unique solution \,$\xi: \mathcal{U}_{\C} \to \mathcal{P}_g^{\Liouville}$\, of Equation~\eqref{eq:KdV:xi-dgl-2} with \,$\xi_{z_0}=\zeta$\,, and \,$\alpha_\lambda$\, is defined
  by Equation~\eqref{eq:KdV:alpha} with respect to this \,$\omega$\,. 
\end{proposition}

\begin{proof}
  As in the proof of Proposition~\ref{P:cmc:symes}, it suffices to show \,$F_\lambda^{-1}\,\diff F_\lambda = \alpha_\lambda^{\Liouville}(\xi)$\,.
  Due to the properties of the modified Birkhoff decomposition, \,$F_\lambda$\, extends holomorphically to \,$\lambda=0$\,, and therefore the Laurent series expansion
  of \,$F_\lambda^{-1}\,\diff F_\lambda$\, with respect to \,$\lambda$\, can only contain powers \,$\lambda^k$\, with \,$k \geq 0$\,. On the other hand, we again have the equation 
  \begin{equation}
    \label{eq:KdV:symes:F-regauge}
    F_\lambda^{-1}\,\diff F_\lambda = B_\lambda\,\zeta\,B_\lambda^{-1}\,\diff z - \diff B_\lambda \cdot B_\lambda^{-1} = \xi_\lambda\,\diff z - \diff B_\lambda \cdot B_\lambda^{-1} \; .
  \end{equation}
  Because \,$B_\lambda$\, extends holomorphically to \,$\lambda=\infty$\,, the series expansion of the 
  right hand side of Equation~\eqref{eq:KdV:symes:F-regauge} can only contain powers \,$\lambda^k$\, with \,$k\leq 1$\,. In summary this shows that the series expansion
  of \,$F_\lambda^{-1}\,\diff F_\lambda$\, contains only the powers \,$\lambda^1$\, and \,$\lambda^0$\,.
  We thus write
  $$ F_\lambda^{-1}\,\diff F_\lambda = (A_0'+A_1'\lambda)\diff z + (A_0''+A_1''\lambda)\diff \bar{z} \quad\text{with}\quad
  A_k', A_k'' \in \mathfrak{sl}(2,\C) \; . $$
  Because \,$F_{\lambda=0}$\, maps into \,$\mathrm{SU}(2)$\, we have \,$F_\lambda^{-1}\,\diff F_\lambda \bigr|_{\lambda=0} \in \Omega^1(\mathcal{U}_{\C}) \otimes \mathfrak{su}(2)$\,, and
  therefore we have
  \begin{equation}
    \label{eq:KdV:symes:A0-reality}
    A_0' = -\overline{A_0''}^t \; .
  \end{equation}
  We now write
  $$ B_{\lambda=\infty} = \left( \begin{smallmatrix} \rho & c \\ 0 & \rho^{-1} \end{smallmatrix} \right) \quad\text{and}\quad
  \zeta_\lambda = \sum_{k=1}^{-g} \zeta_{k}\lambda^k \quad\text{with}\quad \zeta_{1} = \left( \begin{smallmatrix} 0 & {v}_{1} \\ 0 & 0 \end{smallmatrix} \right)
  \;, \zeta_{k} = \left( \begin{smallmatrix} {u}_k & {v}_{k} \\ {w}_k & -{u}_k \end{smallmatrix} \right) $$
  with smooth functions \,$\rho: \mathcal{U}_{\C} \to \R_+$\, and \,$c,{u}_k,{v}_k,{w}_k: \mathcal{U}_{\C} \to \C$\,.
  It then follows from Equation~\eqref{eq:KdV:symes:F-regauge} that
  \begin{align}
    \label{eq:KdV:symes:A-1'}
    A_{1}' & = \zeta_{1} = \left( \begin{smallmatrix} 0 & v_1 \\ 0 & 0 \end{smallmatrix} \right) \;, \\
    \label{eq:KdV:symes:A-1''}    
    A_{1}'' & = 0 \;, \\
    \label{eq:KdV:symes:A0'}    
    A_0' & = \zeta_0 - (\tfrac{\diff\ }{\diff z} B_{\lambda=\infty})\,B_{\lambda=\infty}^{-1} = \left( \begin{smallmatrix} {u}_0-\rho^{-1}\,\rho_{z} & {v}_0-\rho\,c_{z}+c\,\rho_{z} \\ {w}_0 & -{u}_0+\rho^{-1}\,\rho_{z} \end{smallmatrix} \right) \; ,  \\
    \label{eq:KdV:symes:A0''}    
    A_0'' & = -(\tfrac{\diff\ }{\diff\bar{z}} B_{\lambda=\infty})\,B_{\lambda=\infty}^{-1} = \left( \begin{smallmatrix} -\rho^{-1}\,\rho_{\bar z} & -\rho\,c_{\bar z}+c\,\rho_{\bar z} \\ 0 & \rho^{-1}\,\rho_{\bar z} \end{smallmatrix} \right) \; . 
  \end{align}
  From Equations~\eqref{eq:KdV:symes:A0-reality} and \eqref{eq:KdV:symes:A0''} we see (using the fact that \,$\rho$\, is real-valued)
  $$ A_0' = \left( \begin{smallmatrix} \rho^{-1}\,\rho_z & 0 \\ \rho \,\bar{c}_z - \bar{c} \, \rho_z & -\rho^{-1}\,\rho_z \end{smallmatrix} \right) $$
  and by comparing this representation of \,$A_0'$\, with Equation~\eqref{eq:KdV:symes:A0'} we obtain
  $$ u_0 = 2\rho^{-1}\,\rho_z \;, \quad v_0 = \rho\,c_z - c\,\rho_z \AND w_0 = \rho\,\bar{c}_z - \bar{c}\,\rho_z \; . $$
  By inserting these equations into Equations~\eqref{eq:KdV:symes:A0'} and \eqref{eq:KdV:symes:A0''} we obtain
  $$ A_0' = \left( \begin{smallmatrix} \tfrac12 u_0 & 0 \\ w_0 & -\tfrac12 u_0 \end{smallmatrix} \right) \AND
  A_0'' = \left( \begin{smallmatrix} -\tfrac12 \bar{u}_0 & -\bar{w}_0 \\ 0 & \tfrac12 \bar{u}_0 \end{smallmatrix} \right) \; . $$
  These equations, together with Equations~\eqref{eq:KdV:symes:A-1'} and \eqref{eq:KdV:symes:A-1''} show the claimed statement
  \,$F_\lambda^{-1}\,\diff F_\lambda = \alpha_\lambda^{\Liouville}(\xi)$\,. 
\end{proof}

\begin{remark}
The integrable system of the Liouville equation that we described in this section is very closely related to the integrable system of the Korteweg-de-Vries (KdV) equation. Indeed the connection form \,$\alpha_\lambda$\, of Equation~\eqref{eq:KdV:alpha} can be transformed by a so-called \emph{gauge transformation} into the usual connection form for the integrable system of the KdV equation.

For a general linear, partial differential equation \,$\mathrm{d}F= F\alpha$\, on \,$z\in \mathbb{C}$\, given by some matrix-valued 1-form \,$\alpha$\,, a gauge is simply a matrix-valued function \,$G$\, on \,$z$\,. The associated gauge transformation is the transformation of the solution \,$F$\, of the differential equation to \,$\tilde{F}=FG$\,. We have
$$ \mathrm{d}\tilde{F} = \mathrm{d}FG = (\mathrm{d}F)G + F(\mathrm{d}G) = F\alpha G + F\mathrm{d}G = (FG)(G^{-1}\alpha G + G^{-1}(\mathrm{d}G)) \;, $$
hence \,$\tilde{F}$\, is a solution of the gauge-transformed differential equation
$$ \mathrm{d}\tilde{F} = \tilde{F}\tilde{\alpha} \quad\text{with}\quad 
\tilde{\alpha} = G.\alpha := G^{-1}\alpha G + G^{-1}(\mathrm{d}G) \; . $$

In our specific situation with \,$\alpha=\alpha_{\lambda}$\, being given by Equation~\eqref{eq:KdV:alpha}, where \,$\omega$\, is a solution of Liouville's equation \eqref{eq:KdV:liouville}, we seek a gauge \,$G$\, so that the \,$\mathrm{d}\bar{z}$-part of \,$G.\alpha$\, vanishes and the \,$\mathrm{d}z$-part corresponds to the Lax operator of the KdV equation. The Lax operator of the KdV equation is the 1-dimensional Schr\"odinger operator
\,$\tfrac{\mathrm{d}^2}{\mathrm{d}z^2} + f$\,, where the holomorphic function \,$f$\, is called the KdV potential. Rewriting the corresponding eigenfunction equation 
$$ \left(\frac{\mathrm{d}}{\mathrm{d}z}\right)^2 \varphi + f\,\varphi = \lambda\,\varphi $$
as a system of first order differential equations we get
$$ \frac{\mathrm{d}}{\mathrm{d}z} \begin{pmatrix} \varphi' \\ \varphi \end{pmatrix} =  \begin{pmatrix} 0 & \lambda-f \\ 1 & 0 \end{pmatrix} \begin{pmatrix} \varphi' \\ \varphi \end{pmatrix} \; . $$
Multiplying the entries of the right-hand matrix with constants corresponds simply to reparameterising the Schr\"odinger operator in \,$z$\,. 

Remarkably one can explicitly write down a gauge \,$G$\, in terms of \,$\omega$\, and \,$\omega_z$\, that has the properties we need. Indeed, we show that we can take
$$ G = \begin{pmatrix} e^{\omega/4} & -\tfrac{1}{2Q}\,\omega_z\,e^{\omega/4} \\ 0 & e^{-\omega/4} \end{pmatrix} \;. $$
We calculate 
$$ G^{-1} = \begin{pmatrix} e^{-\omega/4} & \tfrac{1}{2Q}\,\omega_z\,e^{\omega/4} \\ 0 & e^{\omega/4} \end{pmatrix} $$
and 
$$ \mathrm{d}G 
= \frac{1}{4}\begin{pmatrix} \omega_z\,e^{\omega/4} & -\tfrac{1}{2Q}(4\omega_{zz}+\omega_z^2)e^{\omega/4} \\ 0 & -\omega_z\,e^{-\omega/4} \end{pmatrix} \mathrm{d}z 
+ \frac14 \begin{pmatrix} \omega_{\bar z}\,e^{\omega/4} & -\tfrac{1}{2Q}(4\omega_{z\bar z}+\omega_z\,\omega_{\bar z})e^{\omega/4} \\ 0 & -\omega_{\bar z}\,e^{-\omega/4} \end{pmatrix} \mathrm{d}\bar{z} \;, $$
hence
$$ G^{-1}\partial_z G =  \frac14 \begin{pmatrix} \omega_z & -\frac{1}{Q}(2\omega_{zz} + \omega_z^2) \\ 0 & -\omega_z \end{pmatrix}$$
and 
$$ G^{-1}\partial_{\bar z} G
= 
\frac14 \begin{pmatrix} \omega_{\bar z} & -\frac{1}{Q}(2\omega_{z\bar{z}} + \omega_z\,\omega_{\bar z}) \\ 0 & -\omega_{\bar z} \end{pmatrix} 
= 
\frac14 \begin{pmatrix} \omega_{\bar z} & -4\bar{Q}e^{-\omega} - \tfrac{1}{Q} \omega_z\,\omega_{\bar z} \\ 0 & -\omega_{\bar z} \end{pmatrix}  
$$
(where the second equals sign follows from the fact that \,$-\omega$\, is a solution of Liouville's equation \eqref{eq:KdV:liouville}).
We now express the connection form \eqref{eq:KdV:alpha} as \,$\alpha = \alpha'\,\mathrm{d}z + \alpha''\,\mathrm{d}\bar{z}$\,, and continue to calculate
$$ G^{-1}\,\alpha'\,G = \frac14 \begin{pmatrix} -\omega_z & \lambda \\ -4Q & \omega_z \end{pmatrix}
\quad\text{and}\quad G^{-1}\,\,\alpha''\,G = \frac14 \begin{pmatrix} -\omega_{\bar z} & 4\bar{Q} e^{-\omega} + \tfrac{1}{Q}\omega_z\,\omega_{\bar z} \\ 0 & \omega_{\bar z} \end{pmatrix} \;, $$
therefore
$$ G.\alpha' = G^{-1}\,\alpha'\, G  + G^{-1}\,\partial_z G 
=  \frac14 \begin{pmatrix}  0 & \lambda - \tfrac{1}{Q}(\omega_z^2 + 2\omega_{zz}) \\ -4Q & 0 \end{pmatrix} $$
and
$$ G.\alpha'' = G^{-1}\,\alpha''\, G  + G^{-1}\,\partial_{\bar z} G = 0 \; . $$
This means the regauged \,$\alpha$\, is 
$$ G.\alpha = \frac14 \begin{pmatrix}  0 & \lambda - \tfrac{1}{Q}(\omega_z^2 + 2\omega_{zz}) \\ -4Q & 0 \end{pmatrix} \mathrm{d}z \;, $$
where the function \,$f=\omega_z^2 + 2\omega_{zz}$\, satisfies
$$ \partial_{\bar z}(\omega_z^2 + 2\omega_{zz}) = 2\,\omega_z\,\omega_{z\bar z} + 2(\omega_{z\bar z})_z 
= (2\omega_z)(2|Q|^2\,e^{-\omega}) + 2(2|Q|^2\,e^{-\omega})_z = 0  $$
and is therefore holomorphic. This shows that the regauged \,$\alpha$\, indeed corresponds to the 1-dimensional Schr\"odinger operator with the potential proportional to \,$f$\,. 
The gauge transformation \,$G$\, therefore transforms all data into holomorphic data. The spectral curve does not change under gauge transformation, therefore the spectral curves of corresponding solutions of the Liouville equation and the KdV equation are the same. 

We expect that each solution of the KdV equation corresponds to a real-1-dimensional family of real solutions of the Liouville equation. 
\end{remark}

\begin{example}
	\label{E:KdV:helicoid}
	See \cite[Section~3.5, Example~6]{doCarmo-engl}. 
	Aside from the flat plane, the helicoid is the only ruled minimal surface in \,$\R^3$\,. It has the conformal parameterisation
	\begin{equation}
		\label{eq:KdV:helicoid:f}
		f: \R^2 \to \R^3,\; (x,y) \mapsto \bigr( \sinh(x)\,\cos(y)\,,\, \sinh(x)\,\sin(y)\,,\,y\,\bigr) \; .
	\end{equation}
	Indeed, it is easy to check that
	$$ \langle f_x,f_x \rangle = \langle f_y,f_y \rangle = \cosh(x)^2 \quad\text{and}\quad \langle f_x,f_y \rangle = 0 $$
	holds. Therefore, writing \,$z=x+iy$\,, the metric \,$e^{\omega} \,\mathrm{d}z\,\mathrm{d}\bar{z}$\, of the helicoid is given by \,$e^{\omega} = 2\langle f_z,f_{\bar z}\rangle = \cosh(x)^2$\, and hence \,$\omega = 2\ln(\cosh(x))$\,. A further explicit computation shows that the Hopf differential
	of \,$f$\, is given by \,$\mi Q\,\diff z^2$\, with \,$Q=\tfrac12$\,. 
	The function \,$\omega$\, satisfies
	$$ \omega_z = \omega_{\bar{z}} = \tfrac12 \omega_x = \tanh(x) \quad\text{and}\quad \omega_{z\bar{z}} = \frac{1}{2\cosh(x)^2} = \tfrac12\,e^{-\omega} = \omega_{zz} \; . $$
	The latter equation proves explicitly that \,$\omega$\, is a solution of Liouville's equation \eqref{eq:KdV:liouville}. This implies in particular
	that the conformal immersion \,$f$\, is indeed minimal. The ruling lines of the helicoid are the images under \,$f$\, of the parallels of the \,$x$-axis in the \,$z$-space. 
	Note that the conformal metric \,$\omega$\, varies with \,$x$\, and therefore \,$f$\, does \emph{not} parameterise the ruling lines proportionally to arc length. These ruling lines also are one of the two families of asymptotic lines of \,$f$\,
	(i.e.~of null lines of the second fundamental form of \,$f$\,). Because the Hopf differential of \,$f$\, is a purely imaginary multiple of \,$\diff z^2$\,,
	the other family of asymptotic lines is the image under \,$f$\, of the parallels of the \,$y$-axis in \,$z$-space; they correspond to the helices on the helicoid \,$f$\,.
	
	Via the iteration in Proposition~\ref{P:KdV:ps} we can compute a polynomial Killing field for the helicoid.
	From Equation~\eqref{eq:KdV:xi-terms} we obtain:
	\begin{align*}
		v_1 & = \tfrac14 e^{\omega/2} = \tfrac14 \cosh(x) \\
		u_0 & = \tfrac12 \omega_z = \tfrac12 \tanh(x) \\
		w_0 & = -Qe^{-\omega/2} = -\tfrac12 \cosh(x)^{-1} \\
		v_0 & = -e^{\omega/2}\bigr( \tfrac{1}{4Q}(\omega_{zz}-\tfrac12 \omega_z^2)-C_0 \bigr) = -e^{\omega/2}(2e^{-\omega}-1-C_0) \; .
	\end{align*}
	By choosing the integration constant for \,$v_0$\, as \,$C_0=-1$\,, we obtain \,$v_0=-2e^{-\omega/2}=-2\cosh(x)^{-1}$\,.
	In Proposition~\ref{P:KdV:ps} we have \,$\tau_0=e^{-\omega/2}v_0=-2\cosh(x)^{-2} = -2e^{-\omega}$\, and therefore by Equation~\eqref{eq:KdV:ps:ps-u}
	$$ u_{-1} = 2(\tau_{0,z} + \omega_z\,\tau_0) = 2(2\omega_ze^{-\omega} + \omega_z \cdot (-2)e^{-\omega}) = 0 \; . $$
	By Equations~\eqref{eq:KdV:ps:ps-tau}--\eqref{eq:KdV:ps:ps-sigma} it follows that also all the other lower terms of the polynomial Killing field vanish
	(if one chooses \,$C_k=0$\, for the integration constants for \,$\tau_k$\, where \,$k\leq -1$\,). This shows that
	the helicoid has spectral genus \,$0$\,, and the corresponding polynomial Killing field is given by
	$$ \xi = \begin{pmatrix} 0 & \tfrac14 \cosh(x) \\ 0 & 0 \end{pmatrix} \lambda + \begin{pmatrix} \tfrac12\tanh(x) & -2\cosh(x)^{-1} \\ -\tfrac12 \cosh(x)^{-1} & -\tfrac12 \tanh(x) \end{pmatrix} \; . $$
	
	To use Symes' method and the Sym-Bobenko formula to recover the immersion \,$f$\, from this polynomial Killing field, we need to choose the proper value
	for the constant \,$\varphi$\, from Proposition~\ref{P:KdV:sym-bobenko}. This constant determines the phase of the Hopf differential of the constructed immersion,
	which equals \,$e^{-\mi\varphi}\,Q\,\diff z^2$\,. Because the Hopf differential of \,$f$\, is \,$\mi \,Q\,\diff z^2$\,, 
	one should choose \,$\varphi = -\tfrac{\pi}{2}$\, in Proposition~\ref{P:KdV:sym-bobenko} to recover \,$f$\,. 
	The KdV potentials that correspond to the spectral genus 0 solutions of the Liouville equation are simply the constant functions. Indeed, one can easily confirm by an explicit calculation that for the helicoid, \,$\omega_z^2 + 2\omega_{zz}=1$\, holds. 
\end{example}

\section{Blowing up cmc tori to minimal surfaces}
\label{Se:blowup1}

We now let a sequence \,$(f_n)_{n\in \mathbb{N}}$\, of smooth cmc torus immersions \,$f_n: \C \to \R^3$\, of fixed spectral genus \,$g<\infty$\, be given.
Without loss of generality we may suppose that the parameterisations
are chosen in such a way that the natural coordinate \,$z\in \C$\, is a conformal coordinate for all of them and that the Hopf differential of \,$f_n$\,
is given by \,$Q \diff z^2$\, with a fixed \,$Q \in \C^{\times} $\,. 
Furthermore we suppose that the \,$f_n$\, are scaled in the destination space \,$\R^3$\, such that the value of the mean curvature
of the \,$f_n$\, is some fixed \,$H>0$\,. 
We write the Riemannian metric on \,$\C$\, that is induced by \,$f_n$\, as \,$e^{\omega_n}\,\diff z\,\diff \bar{z}$\, with respect to some
smooth function \,$\omega_n: \mathbb{C}\to\mathbb{R}_+$\,. We choose a base point \,$z_0 \in \C$\, at first arbitrarily (below we will make
restrictions on this choice) and introduce the objects defined in Section~\ref{Se:cmc} for \,$f_n$\, with respect to this base point (where we attach the subscript
\,${}_n$\, to the corresponding symbols). In particular we consider the cmc potential \,$\zeta_{n} \in \mathcal{P}_g$\, for \,$f_n$\, at the base point \,$z_{0}$\,,
the polynomials \,$a_n(\lambda) = -\lambda\,\det(\zeta_n)$\, defining the corresponding spectral curve \,$\Sigma_n$\,,
the zeros \,$\lambda_{n,1},\dotsc,\lambda_{n,g}$\, of \,$a_n(\lambda)$\, with \,$|\lambda_{n,k}|<1$\,, 
and the points \,$(\beta_{n,1},\nu_{n,1}),\dotsc,(\beta_{n,g},\nu_{n,g})$\, in the support of the the spectral divisor defined by \,$\zeta_n$\,.
We will assume that the \,$\lambda_{n,k}$\, 
are ordered as in \eqref{eq:cmc:lambdak-ordering} 
for every \,$n\in \mathbb{N}$\,.

In the case where \,$|\lambda_{n,1}|$\, is bounded away from zero as \,$n\to \infty$\,, all \,$|\lambda_{n,k}|$\, are bounded away from zero by \eqref{eq:cmc:lambdak-ordering}
and all \,$|\beta_{n,k}|$\, are bounded and bounded away from zero by Proposition~\ref{P:cmc:magic-estimates-new}(2). It then follows from Proposition~\ref{P:cmc:xi-trace}
that all of the finitely many coefficients of \,$\zeta_n \in \mathcal{P}_g$\, are bounded as \,$n\to \infty$\,. Because the linear space \,$\mathcal{P}_g$\,
is finite-dimensional, after we pass to a subsequence, 
\,$\zeta_n$\, converges to a \,$\zeta_\infty  \in \mathcal{P}_g$\, as \,$n\to \infty$\,. This limiting polynomial Killing field is non-zero because
\,$-\lambda \det(\zeta_\infty) = \lim_{n\to\infty} a_n(\lambda)$\, is a polynomial of degree \,$2g$\, whose zeros inside the unit disk are exactly the \,$\lim_{n\to\infty} \lambda_{n,k}$\,. As a consequence, the associated connection form \,$\alpha_\lambda(\zeta_n)$\,,
which characterises the integrable system associated to \,$\zeta_n$\,, converges non-trivially to an \,$\alpha_\lambda(\zeta_\infty)$\, that is of the structure given by Equation~\eqref{eq:cmc:alpha}.
Therefore Symes' method applied to \,$\zeta_\infty$\, yields by the Sym-Bobenko formula \eqref{eq:cmc:sym-bobenko}
a cmc immersion \,$f_\infty$\, into \,$\R^3$\,. This immersion is the limit of the corresponding subsequence of the \,$f_n$\,.
This case is not relevant to us because we are interested in obtaining a solution of a \emph{different} integrable system by a limit process. So we will assume  in the sequel that \,$\lambda_{n,1} \to 0$\, as \,$n\to \infty$\,. 


To obtain convergence in the latter setting we first construct a blow-up of the coordinate \,$z$\, around the base point \,$z_{0}$\, and of the immersion \,$f_n$\, around the image of the base point \,$f_n(z_0)$\,. More explicitly, for sequences \,$(r_n)_{n\in \bbN}$\, and \,$(h_n)_{n\in \bbN}$\, of positive real numbers, we introduce a rescaled coordinate \,$\tilde{z}$\, on the domain \,$\C$\, and rescaled immersions \,$\tilde{f}_n$\, by
\begin{equation}
	z = z_{0}+r_n\,\tilde{z} \AND \tilde{f}_n(\tilde{z}) = h_n^{-1}\bigr( f_n(z_{0}+r_n\,\tilde{z})-f_n(z_{0})\bigr) \;. 
\end{equation}
A blown up spectral parameter \,$\tilde{\lambda}$\, and a blown up polynomial Killing field \,$\tilde{\zeta}_n$\, are obtained by choosing
more sequences \,$(\ell_n)_{n\in \bbN}$\, and \,$(s_n)_{n\in \bbN}$\, of positive real numbers
and setting
\begin{equation}
	\label{eq:blowup1:blowup-spectral}
	\lambda = \ell_n\,\tilde{\lambda} \AND \tilde{\zeta}_n(\tilde{\lambda}) = s_n\,\zeta_n(\ell_n\,\tilde{\lambda}) \; .
\end{equation}
We will see that the spectral data \,$\tilde{\lambda}$\, and \,$\tilde{\zeta}_n$\, correspond to the immersion \,$\tilde{f}_n(\tilde{z})$\, if we choose the sequences of positive real numbers according to Equations~\eqref{eq:blowup1:blowup1:factors} below.
Note that the relationship between the new coordinates \,$(\tilde{z},\tilde{\lambda})$\, and the original coordinates \,$(z,\lambda)$\, depends on \,$n$\,.
For the construction of the blow-up we will consider the new coordinates \,$(\tilde{z},\tilde{\lambda})$\, as being independent of \,$n$\,
and consequently let the corresponding values of \,$(z,\lambda)$\, vary with \,$n$\, (although we do not indicate this dependency with a subscript \,${}_n$\,
for the sake of notational sanity).
Also note that \,$\tilde{z}$\, is a conformal coordinate for \,$\tilde{f}_n$\,, that the induced Riemannian metric of \,$\tilde{f}_n$\, is given by
\,$e^{\tilde{\omega}_n(\tilde{z})}\,\diff \tilde{z}\,\diff \bar{\tilde{z}}$\, with the conformal factor \,$e^{\tilde{\omega}_n(\tilde{z})} = h_n^{-2}\,r_n^2\,e^{\omega_n(z_{0}+r_n\tilde{z})}$\,, and that
\,$\tilde{f}_n$\, has the constant mean curvature \,$h_n\,H$\, and the Hopf differential \,$h_n^{-1}\,r_n^2\,Q\,\diff \tilde{z}^2$\,. We also express the \,$\lambda_{n,k}$\,
and the \,$\beta_{n,k}$\, in the \,$\tilde{\lambda}$-coordinates by
$$ \tilde{\lambda}_{n,k} = \ell_n^{-1} \,\lambda_{n,k} \AND \tilde{\beta}_{n,k} = \ell_n^{-1} \,\beta_{n,k} \; . $$

Our first objective in the blow-up construction is to choose the sequences determining the blow-ups in such a way that both the blown-up
cmc potentials \,$\tilde{\zeta}_n$\, and the associated connection forms \,$\alpha(\tilde{\zeta}_n)$\, converge in the blown-up coordinates
to non-trivial quantities. 
Here we want to make our choice in such a way that \,$\tilde{\lambda}_{n,1}$\, (corresponding to the branch point of the blown up spectral curve with the shortest \,$\tilde{\lambda}$-coordinate) is bounded and bounded away from zero. This will give
us a KdV spectral curve and the situation described in Section~\ref{Se:KdV} for the blow-up. We will effect this by making the choice \,$\ell_n = |\lambda_{n,1}|$\,. 
We will then see that with the proper choice of \,$h_n$\, also the \,$\tilde{f}_n$\, converge to a surface immersion into \,$\R^3$\,
and that the limiting immersion is minimal.

We begin by exploring necessary conditions for the convergence of the blow-up we just described.

\begin{proposition}
\label{P:blowup1:necessary}
In the situation described above (in particular, \,$\ell_n = |\lambda_{n,1}|\to 0$\,) suppose that the sequence of blown-up cmc immersions converges locally uniformly to a non-umbilical, conformal surface immersion \,$\tilde{f}_\infty: \bbC \to \bbR^3$\,. Then the following holds:
\begin{enumerate}
	\item 
	Up to multiplication with sequences in \,$\bbR_+$\, that are bounded and bounded away from zero, we have
	$$ s_n= \ell_n^{1/2}\;,\quad r_n = \ell_n^{1/2} \quad\text{and}\quad h_n = \ell_n\; . $$
	\item
	\,$|\lambda_{n,1}|^{-1} \cdot e^{\omega_n(z_{0})}$\, is bounded and bounded away from zero.
	\item \,$\tilde{f}_\infty$\, is minimal.
\end{enumerate}
\end{proposition}

\begin{proof}
Because \,$\tilde{f}_n$\, converges locally uniformly to \,$\tilde{f}_\infty$\,, the conformal factors \,$\tilde{\omega}_n$\, and all their derivatives also converge locally uniformly. Because the coefficients in the Pinkall-Sterling iteration are differential polynomials in \,$\tilde{\omega}_n$\,, it follows that all coefficients of the potentials \,$\tilde{\zeta}_n$\, converge. Here the limits \,$\tilde{v}_{\infty}(\tilde{\lambda})$\, and \,$\tilde{w}_\infty(\tilde{\lambda})$\,
of the off-diagonal entries \,$\tilde{v}_n(\tilde{\lambda})$\, and \,$\tilde{w}_n(\tilde{\lambda})$\, of \,$\tilde{\zeta}_n(\tilde{\lambda})$\, cannot vanish identically because \,$\tilde{f}_\infty$\, is a non-umbilical immersion.
By Proposition~\ref{P:cmc:xi-trace} and \eqref{eq:blowup1:blowup-spectral} we have
\begin{align*}
	\tilde{v}_n(\tilde{\lambda}) & = s_n \cdot \tfrac12 H e^{\omega_n(z_{0})/2} \ell_n^{-1}\tilde{\lambda}^{-1}\cdot \prod_{k=1}^g(1-\tilde{\beta}_k^{-1}\,\tilde{\lambda}) \\
	\tilde{w}_n(\tilde{\lambda}) & = s_n \cdot (-Q)e^{-\omega_n(z_{0})/2}\cdot \prod_{k=1}^g (1-\ell_n^2 \bar{\tilde{\beta}}_k\tilde{\lambda}) \; . 
\end{align*}
It follows that the factors
$$ \tfrac12 H e^{\omega_n(z_{0})/2} s_n \ell_n^{-1} \quad\text{and}\quad -Qe^{-\omega_n(z_{0})/2} s_n $$
are bounded and bounded away from zero. Therefore the sequences \,$s_n^2 \cdot \ell_n^{-1}$\, and \,$e^{\omega_n(z_{0})}\cdot \ell_n^{-1}$\, are also bounded and bounded away from zero. Because \,$s_n^2 \cdot \ell_n^{-1}$\, is bounded and bounded away from zero, we have \,$s_n = \ell_n^{1/2}$\, (up to a sequence bounded and bounded away from zero, which we will assume without loss of generality to be equal to \,$1$\, in the sequel). The fact that \,$e^{\omega_n(z_{0})}\cdot \ell_n^{-1}$\, is bounded and bounded away from zero implies (2).

The Riemannian metric, mean curvature and Hopf differential of \,$\tilde{f}_n$\, are given by
$$ e^{\tilde{\omega}_n(0)}\,\mathrm{d}\tilde{z} \,\mathrm{d}\bar{\tilde{z}} = h_n^{-2}\,r_n^2\,e^{\omega_n(z_{0})}
\,\mathrm{d}\tilde{z} \, \mathrm{d}\bar{\tilde{z}} \;,\quad
\tilde{H}_n=h_n\,H \quad\text{and}\quad
\tilde{Q}_n\,\mathrm{d}\tilde{z}^2 =h_n^{-1} r_n^2\,Q\,\mathrm{d}\tilde{z}^2 $$
respectively. Because these quantities converge to the corresponding quantities of the non-umbilical immersion \,$\tilde{f}_\infty$\,, it follows that the sequences \,$h_n^{-2} \, r_n^2 \, e^{\omega_n(z_{0})}$\, and \,$h_n^{-1}\,r_n^2$\, are bounded and bounded away from zero. Therefore the product and the quotient of these two sequences, \,$h_n^{-3}\,r_n^4\,e^{\omega_n(z_{0})}$\, and \,$h_n^{-1}\,e^{\omega_n(z_{0})}$\,, are also bounded and bounded away from zero. Because \,$e^{\omega_n(z_{0})}\cdot \ell_n^{-1}$\, is bounded and bounded away from zero, this implies \,$h_n=\ell_n$\, and \,$r_n=\ell_n^{1/2}$\, up to multiplication with a sequence that is bounded and bounded away from zero, which completes the proof of (1). Finally we note that \,$\tilde{H}_n = h_n\,H \to 0$\,, showing that \,$\tilde{f}_\infty$\, is minimal. 
%
\end{proof}

\begin{theorem}
  \label{T:blowup1:blowup1}
  In the situation described above, suppose that the following conditions are satisfied (see Remark~\ref{R:blowup:motivation} below for the motivation):
  \begin{itemize}
  \item[(a)]
    \,$\lim_{n\to\infty} \lambda_{n,1} = 0$\,.
  \item[(b)]
    \,$|\lambda_{n,1}|^{-1}\cdot e^{\omega_n(z_{0})}$\, is bounded and bounded away from zero.
  \item[(c)]
For all \,$k\neq k'$\,, the sequences \,$|\lambda_{n,1}|^{-1}\cdot(\beta_{n,k}-\beta_{n,k'})$\, and \,$|\lambda_{n,1}|^{-1}\cdot({\beta}_{n,k}^{-1}-{\beta}_{n,k'}^{-1})$\, are bounded away from zero. 
  \end{itemize}
Choose
  \begin{equation}
    \label{eq:blowup1:blowup1:factors}
    \ell_n = |\lambda_{n,1}| \;,\quad r_n = s_n = \ell_n^{1/2} \AND h_n=\ell_n \; .
  \end{equation}
  
  Then there exists a subsequence of the \,$(f_n)$\,, again denoted by \,$(f_n)$\,, so that the rescaled potentials
  defined in \eqref{eq:blowup1:blowup-spectral} converge to some non-zero \,$\tilde{\zeta}_\infty(\tilde{\lambda})$\,.
  With respect to this subsequence, all objects associated to the blow-up converge to the corresponding objects of a minimal immersion \,$\tilde{f}_\infty: \mathbb{C} \to \mathbb{R}^3$\,. 
  More specifically, the following holds: 
  \begin{enumerate}
  \item
    \textbf{The limiting divisor has the correct degree.}
    There exists
    \,$1 \leq \tilde{d} \leq g$\, so that \,$\tilde{\lambda}_{n,k}$\, converges to some \,$\tilde{\lambda}_{\infty,k} \in \C \setminus \mathbb{D}$\, for \,$k\leq \tilde{d}$\,
    and \,$\tilde{\lambda}_{n,k} \to \infty$\, for \,$k > \tilde{d}$\,. The spectral curve \,$\tilde{\Sigma}_\infty$\, corresponding to \,$\tilde{\zeta}_\infty$\,
    has geometric genus \,$\tilde{g} = \lfloor \tfrac12 \tilde{d} \rfloor$\,. Among the \,$k \in \{1,\dotsc,g\}$\, there are exactly \,$\tilde{g}$-many so that
    \,$\tilde{\beta}_{n,k}$\, converges to some \,$\tilde{\beta}_{\infty,k} \in \C \setminus \mathbb{D}$\,, whereas for the other values of \,$k$\,
    we have \,$\tilde{\beta}_{n,k}\to \infty$\,. 
  \item
	\textbf{The limiting potential is of the type of the Liouville integrable system.}
    Let \,$\lambda^{\Liouville} = \tilde{\lambda}^{-1}$\, and
    \,$\zeta^{\Liouville}_\infty(\lambda^{\Liouville}) = \tilde{\zeta}_\infty((\lambda^{\Liouville})^{-1})$\,, then \,$\zeta^{\Liouville}_\infty\in \mathcal{P}_{\tilde{d}}^{\Liouville}$\, holds.

    Moreover write
    $$ \alpha_{\ell_n\tilde{\lambda}}(\zeta_n) = \tilde{U}_n\diff \tilde{z} + \tilde{V}_n\diff \bar{\tilde{z}} \AND
    \alpha^{\Liouville}_{\lambda^{\Liouville}}(\zeta_\infty^{\Liouville}) = U^{\Liouville}_\infty\,\diff \tilde{z} + V^{\Liouville}_\infty\,\diff \bar{\tilde{z}} \; . $$
    Then \,$\tilde{U}_n$\, and \,$\tilde{V}_n$\, converge for \,$n\to\infty$\, to \,$U^{\Liouville}_\infty$\, and \,$V^{\Liouville}_\infty$\,, respectively. 

  \item
  \textbf{The blown-up extended frames and polynomial Killing fields converge.} 
    Let \,$F_n(z,\lambda)$\, be the extended frame corresponding to \,$f_n$\, and let
    \,$\tilde{F}_n(\tilde{z},\tilde{\lambda}) = F_n(z_{0}+r_n\tilde{z},\ell_n \tilde{\lambda})$\,. Then \,$\tilde{F}_n$\, converges for \,$n\to\infty$\, locally uniformly in \,$\tilde{z}\in \C$\,
    and \,$\tilde{\lambda}\in \C^{\times}  \cup \{\infty\}$\, to some \,$\tilde{F}_\infty(\tilde{z},\tilde{\lambda})$\,. \,$\tilde{F}_\infty$\, depends holomorphically on \,$\tilde{\lambda}$\, and
    \,$\tilde{F}_\infty(\tilde{z},\tilde{\lambda}=\infty) \in \mathrm{SU}(2)$\, holds.

    Let \,$F^{\Liouville}_\infty(\tilde{z},\lambda^{\Liouville}) = \tilde{F}_\infty(\tilde{z},(\lambda^{\Liouville})^{-1})$\, and let \,$\xi^{\Liouville}_\infty: \tilde{z}\ni \C \to \mathcal{P}_g^{\Liouville}$\, be the
    solution of \eqref{eq:KdV:xi-dgl-2} with \,$\xi^{\Liouville}_{\infty,\tilde{z}=0} = \zeta^{\Liouville}_\infty$\,. Then
    \,$\diff F_\infty^{\Liouville} = F_\infty^{\Liouville}\,\alpha^{\Liouville}(\xi^{\Liouville}_\infty)$\, holds.

  \item
  \textbf{The blown-up conformal factors converge.}
    The functions \,$\tilde{\omega}_n$\, converge locally uniformly to a real-valued, smooth function \,$\tilde{\omega}_\infty$\,. 
    
  \item
  \textbf{The blown-up immersions converge to a minimal immersion.}
    The sequence of blown-up cmc immersions \,$\tilde{f}_n$\,
    converges locally uniformly to the conformal minimal surface immersion \,$\tilde{f}_\infty: \C \to \R^3$\, defined by Equation~\eqref{P:KdV:sym-bobenko}
    with the extended frame \,$F=F^{\Liouville}_\infty$\,. The immersion \,$\tilde{f}_\infty$\, induces the metric \,$e^{\tilde{\omega}_\infty}\,\diff \tilde{z}\,\diff\bar{\tilde{z}}$\,,
    and has the Hopf differential \,$Q \,\diff \tilde{z}^2$\,. 
  \end{enumerate}
\end{theorem}

\begin{remark}
	\label{R:blowup:motivation}
  As we explained at the beginning of the section, condition (a) of Theorem~\ref{T:blowup1:blowup1} is required so that the blown-up sequence of solutions of the sinh-Gordon equation can converge to a solution of a different integrable system.
  
  It was shown in Proposition~\ref{P:blowup1:necessary} that condition (b) is a necessary condition for the convergence of the blown-up conformal factor; this condition compares the rate of convergence of the conformal factor \,$e^{\omega_n(z_0)}$\, to the rate of convergence of the shortest root \,$\lambda_{n,1}$\,. By Equation~\eqref{eq:cmc:xi-terms}, this condition can also be interpreted as a condition on the rate of convergence of the lowest term of the polynomial Killing fields, or by Equation~\eqref{eq:cmc:xi-trace:omega}, of the \,$\beta_{n,k}$\, as \,$n\to\infty$\,. 
  
  Condition (c) is the only one of the conditions in the theorem that is stronger than the necessary conditions of Proposition~\ref{P:blowup1:necessary}. It implies in particular that the divisors corresponding to the spectral line bundle \,$\Lambda_n$\, and to its dual bundle converge to non-special divisors on the limiting spectral curve in the blow-up, which is necessary for that divisor to correspond to an immersion into \,$\mathbb{R}^3$\,. However, condition (c) is more restrictive than this non-specialness condition, because it excludes divisor points in singularities of the spectral curve, and also multiple points in the divisor with equal projections on the \,$\lambda$-plane. We impose these more restrictive conditions to avoid the need to work with singular algebraic curves and generalised divisors (in the sense of Hartshorne, also compare \cite{KLSS}).
  \end{remark}

\begin{proof}[Proof of Theorem~\ref{T:blowup1:blowup1}.]
Due to our choice of \,$\ell_n$\,, \,$r_n$\, and \,$h_n$\, in \eqref{eq:blowup1:blowup1:factors}
we have \,$e^{\tilde{\omega}_n(0)} = h_n^{-2}\,r_n^2\,e^{\omega_n(z_{0})} = |\lambda_{n,1}|^{-1}\,e^{\omega_n(z_{0})}$\,,
thus it follows from condition (b) that we can choose a subsequence of \,$(f_n)$\, so that \,$\tilde{\omega}_n(0)$\, converges to some real number \,$\tilde{\omega}_\infty(0)$\,.
We will see in the proof of (3) below that we will need to pass to a further subsequence so that  \,$\tilde{\omega}_n(\tilde{z})$\, also converges for \,$\tilde{z}\neq 0$\,. 
  
To prove (1), we investigate how the spectral data behave under the blow-up. By spectral data, we mean the polynomial \,$a_n(\lambda)$\, which defines the spectral curve \,$\Sigma_n$\,, the anti-symmetric function \,$\nu$\, on \,$\Sigma_n$\,, and the potential \,$\zeta_n$\,. The roots \,$\lambda_{n,1},\dotsc,\lambda_{n,g}$\, of \,$a_n(\lambda)$\,
inside the unit disk correspond to \,$\tilde{\lambda}_{n,k} = \ell_n^{-1}\,\lambda_{n,k}$\,. Due to our choice \eqref{eq:blowup1:blowup1:factors} and the ordering
\eqref{eq:cmc:lambdak-ordering}, we have
$$ 1 = |\tilde{\lambda}_{n,1}| \leq |\tilde{\lambda}_{n,2}| \leq \dotsc \leq |\tilde{\lambda}_{n,g}| \; . $$
Note that for \,$k\geq 2$\, it is possible for the sequence \,$(|\tilde{\lambda}_{n,k}|)_{n\in \N}$\, to be unbounded. By passing to a subsequence, we can arrange
that for every \,$k \geq 1$\, the sequence \,$(\tilde{\lambda}_{n,k})_{n\in \N}$\, converges to some \,$\tilde{\lambda}_{\infty,k} \in \mathbb{P}^1 \setminus \mathbb{D}$\,.
There exists \,$1\leq \tilde{d} \leq g$\, so that \,$\tilde{\lambda}_{\infty,k} \in \C$\, for \,$k\leq \tilde{d}$\, and \,$\tilde{\lambda}_{\infty,k}=\infty$\, for \,$k>\tilde{d}$\,.
By Proposition~\ref{P:cmc:lambdak} we have
$$ a_n(\ell_n\tilde{\lambda}) = -\frac12 HQ \prod_{k=1}^g (1-\tilde{\lambda}_{n,k}^{-1}\tilde{\lambda})\cdot (1-\ell_n\,\bar{\lambda}_{n,k}\,\tilde{\lambda}) \; . $$
Due to \,$|\lambda_{n,k}| < 1$\, for all \,$n,k$\,, we here have \,$\ell_n\,\bar{\lambda}_{n,k} \to 0$\, for any \,$k$\,; and moreover \,$\tilde{\lambda}_{n,k}^{-1} \to 0$\,
for \,$k>\tilde{d}$\,. This shows that the polynomials \,$\tilde{a}_n(\tilde{\lambda}) := a_n(\ell_n\,\tilde{\lambda})$\, converge for \,$n\to\infty$\, to the polynomial
$$ \tilde{a}_\infty(\tilde{\lambda}) = -\frac12 HQ \prod_{k=1}^{\tilde{d}} (1-\tilde{\lambda}_{\infty,k}^{-1}\tilde{\lambda}) \; . $$
Clearly, the polynomial \,$\tilde{a}_\infty(\tilde{\lambda})$\, has degree \,$\tilde{d}$\, and its zeros are exactly
\,$\tilde{\lambda}_{\infty,1},\dotsc,\tilde{\lambda}_{\infty,\tilde{d}}$\,. The spectral curve corresponding to \,$a_n(\lambda)$\, is given by
$$ \Sigma_n = \{ (\lambda,\nu) \mid \nu^2 = \lambda\,a_n(\lambda) \} \; . $$
On \,$\Sigma_n$\, we thus have
$$ \nu^2 = \lambda\,a_n(\lambda) = \ell_n\,\tilde{\lambda}\,a_n(\ell_n\,\tilde{\lambda}) = \ell_n\,\tilde{\lambda}\,\tilde{a}_n(\tilde{\lambda}) \; . $$
Thus by blowing up the holomorphic function \,$\nu$\, on \,$\Sigma_n$\, by
$$ \nu = \ell_n^{1/2}\,\tilde{\nu}\;,\quad\text{we have } \tilde{\nu}^2 = \tilde{\lambda}\,\tilde{a}_n(\tilde{\lambda}) \; . $$
The function \,$\tilde{\nu}$\, defines the rescaled (blown-up) spectral curve
$$ \tilde{\Sigma}_n = \{ (\tilde{\lambda},\tilde{\nu}) \mid \tilde{\nu}^2 = \tilde{\lambda}\,\tilde{a}_n(\tilde{\lambda}) \} \;, $$
which is naturally biholomorphic to \,$\Sigma_n$\, via the biholomorphic map
$$ \Theta_n: \tilde{\Sigma}_n \to \Sigma_n,\; (\tilde{\lambda},\tilde{\nu}) \mapsto (\ell_n\,\tilde{\lambda}, \ell_n^{1/2}\,\tilde{\nu}) \; . $$
It follows that if we view \,$\tilde{\nu} = \tilde{\nu}(\tilde{\lambda})$\, as a two-valued, holomorphic function on \,$\C \ni \tilde{\lambda}$\,, then
\,$\tilde{\nu}$\, converges for \,$n\to \infty$\, to a holomorphic function on the hyperelliptic complex curve
$$ \tilde{\Sigma}_\infty = \{ (\tilde{\lambda},\tilde{\nu}) \mid \tilde{\nu}^2 = \tilde{\lambda}\,\tilde{a}_\infty(\tilde{\lambda}) \} $$
which we may regard as the limit of the curves \,$\tilde{\Sigma}_n$\,. Because \,$\tilde{\lambda}\tilde{a}_\infty(\tilde{\lambda})$\, is a polynomial of degree \,$\tilde{d}+1$\,,
the complex curve \,$\tilde{\Sigma}_\infty$\, has genus \,$\tilde{g} = \lfloor \tfrac12 \tilde{d} \rfloor$\,. 

Further we let \,$(\beta_{n,1},\nu_{n,1}),\dotsc,(\beta_{n,g},\nu_{n,g}) \in \Sigma_n$\, be the points of the spectral divisor corresponding to the line bundle \,$\Lambda_n$\,.
Then \,$(\tilde{\beta}_{n,k},\tilde{\nu}_{n,k}) = (\ell_n^{-1}\,\beta_{n,k},\ell_n^{-1/2}\,\nu_{n,k}) \in \tilde{\Sigma}_n$\, are the corresponding points for the
line bundle \,$\tilde{\Lambda}_n = \Theta_n^* \Lambda_n$\, on \,$\tilde{\Sigma}_n$\,. By Proposition~\ref{P:cmc:magic-estimates-new}(2) we have
$$ 1 = |\tilde{\lambda}_{n,1}| \leq |\tilde{\beta}_{n,k}| \leq \ell_n^{-2}\,|\tilde{\lambda}_{n,1}|^{-1} = \ell_n^{-2} \; . $$
By again passing to a subsequence, we may assume that \,$\tilde{\beta}_{n,k}$\, converges for \,$n\to \infty$\, to some \,$\tilde{\beta}_{\infty,k} \in \mathbb{P}^1$\,.
Note that \,$|\tilde{\beta}_{\infty,k}| \geq 1$\,, but \,$\tilde{\beta}_{\infty,k}=\infty$\, is possible for all \,$k$\,. 
We define 
$$ S := \bigr\{ k \in \{1,\dotsc,g\} \,|\, \tilde{\beta}_{\infty,k} \neq \infty \bigr\}\; . $$
Because of \,$\tilde{\nu}_{n,k}^2 = \tilde{\beta}_{n,k}\,\tilde{a}_n(\tilde{\beta}_{n,k})$\, we can also achieve that \,$\tilde{\nu}_{n,k}$\, converges to some \,$\tilde{\nu}_{\infty,k}$\,
so that \,$(\tilde{\beta}_{\infty,k},\tilde{\nu}_{\infty,k})\in \tilde{\Sigma}_\infty$\,. Due to condition (c), the finite \,$\tilde{\beta}_{\infty,k}$\, are pairwise unequal
and in particular the divisor \,$\sum_k (\tilde{\beta}_{\infty,k}, \tilde{\nu}_{\infty,k})$\, on \,$\tilde{\Sigma}_\infty$\, is non-special. This completes the proof of (1). 

To prove (2), we will show that for the subsequence of the \,$(\zeta_n)$\, we have now chosen, the rescaled cmc potentials \,$\tilde{\zeta}_n$\, defined by Equation~\eqref{eq:blowup1:blowup-spectral} converge to some \,$\tilde{\zeta}_\infty$\, as \,$n\to\infty$\,,
and that the additional convergence statements in (3), (4) hold for \,$\tilde{z}=0$\,. For this purpose we write
$$ \zeta_n(\lambda) = \left( \begin{matrix} u_n(\lambda) & v_n(\lambda) \\ w_n(\lambda) & -u_n(\lambda) \end{matrix} \right)
\AND \tilde{\zeta}_n(\tilde{\lambda}) = \left( \begin{matrix} \tilde{u}_n(\tilde{\lambda}) & \tilde{v}_n(\tilde{\lambda}) \\ \tilde{w}_n(\tilde{\lambda}) & -\tilde{u}_n(\tilde{\lambda}) \end{matrix} \right) \; . $$

We have by Equation~\eqref{eq:cmc:xi-trace:v}
$$ \tilde{\lambda} \tilde{v}_n(\tilde{\lambda}) = s_n\,\ell_n^{-1} v_n(\ell_n^{-1}\,\tilde{\lambda})
= \frac12 H \ell_n^{-1/2}\,e^{\omega_n(z_{0})/2} \prod_{k=1}^g (1-\beta_{n,k}^{-1}\,\ell_n^{-1}\,\tilde{\lambda})
= \frac12 H e^{\tilde{\omega}_n(0)/2} \prod_{k=1}^g (1-\tilde{\beta}_{n,k}^{-1}\,\tilde{\lambda}) \; . $$
 Hence \,$\tilde{\lambda}\,\tilde{v}_n(\tilde{\lambda})$\, converges for \,$n\to\infty$\, to the polynomial 
$$ \tilde{\lambda}\,\tilde{v}_\infty(\tilde{\lambda}) = \frac12 H e^{\tilde{\omega}_\infty(0)/2} \prod_{k\in S} (1-\tilde{\beta}_{\infty,k}^{-1}\,\tilde{\lambda}) \; . $$
We define \,$\chi_{n,k}(\lambda) = \prod_{k'\neq k} \frac{\lambda-\beta_{n,k'}}{\beta_{n,k}-\beta_{n,k'}}$\,, compare Equation~\eqref{eq:cmc:xi-trace:chi},
then we have
$$ \chi_{n,k}(\ell_n\tilde{\lambda}) = \prod_{k'\neq k} \frac{\ell_n \tilde{\lambda}-\beta_{n,k'}}{\beta_{n,k}-\beta_{n,k'}}
= \prod_{k'\neq k} \frac{\tilde{\lambda}-\tilde{\beta}_{n,k'}}{\tilde{\beta}_{n,k}-\tilde{\beta}_{n,k'}} =: \tilde{\chi}_{n,k}(\tilde{\lambda}) \; . $$
If \,$k\in S$\,, then \,$\tilde{\chi}_{n,k}(\tilde{\lambda})$\, converges for \,$n\to\infty$\, to
\,$\tilde{\chi}_{\infty,k}(\tilde{\lambda}) = \prod_{k' \in S \setminus \{k\}} \frac{\tilde{\lambda}-\tilde{\beta}_{\infty,k'}}{\tilde{\beta}_{\infty,k}-\tilde{\beta}_{\infty,k'}}$\,;
note that the denominator cannot become zero because of (c). If \,$k\not\in S$\,, then \,$\tilde{\chi}_{n,k}(\tilde{\lambda})
= \tilde{\beta}_{n,k}^{-1} \prod_{k'\neq k} \tfrac{\tilde{\beta}_{n,k'}^{-1}\tilde{\lambda}-1}{\tilde{\beta}_{n,k'}^{-1}-\tilde{\beta}_{n,k}^{-1}}$\,
converges to zero; note that condition (c) again implies that the denominator is bounded away from zero. It follows by Equation~\eqref{eq:cmc:xi-trace:u}
that
$$ \tilde{u}_n(\tilde{\lambda}) = s_n\, u_n(\ell_n\,\tilde{\lambda}) = -\ell_n^{1/2} \sum_{k=1}^g \beta_{n,k}^{-1}\,\nu_{n,k}\,\chi_{n,k}(\lambda) = -\sum_{k=1}^g \tilde{\beta}_{n,k}^{-1}\,\tilde{\nu}_{n,k}\,\tilde{\chi}_{n,k}(\tilde{\lambda}) $$
converges for \,$n\to\infty$\, to the polynomial
$$ \tilde{u}_\infty(\tilde{\lambda}) = -\sum_{k\in S} \tilde{\beta}_{\infty,k}^{-1}\,\tilde{\nu}_{\infty,k}\,\tilde{\chi}_{\infty,k}(\tilde{\lambda}) \; .$$

For the proof of the convergence of \,$\tilde{\zeta}_n$\,, it now only remains to control the \,$w_n(\lambda)$\,. From the equation 
$$ a_n(\lambda) = -\lambda\,\det(\zeta_n(\lambda)) = -\lambda \bigr( u_n(\lambda)^2 - v_n(\lambda)\,w_n(\lambda) \bigr) $$
we get because of \,$\ell_n\,s_n^{-2} = 1$\, 
\begin{align*}
\tilde{a}_n(\tilde{\lambda}) & = a_n(\ell_n\,\tilde{\lambda}) = -\ell_n\tilde{\lambda}\,\bigr(u_n(\ell_n\tilde{\lambda})^2 - v_n(\ell_n\tilde{\lambda})w_n(\ell_n\tilde{\lambda})\bigr) \\
& = -\tilde{\lambda}\,\bigr(\tilde{u}_n(\tilde{\lambda})^2 - \tilde{v}_n(\tilde{\lambda})\tilde{w}_n(\tilde{\lambda})\bigr)
\end{align*}
and therefore
$$ \tilde{w}_n(\tilde{\lambda})
= \frac{\tilde{a}_n(\tilde{\lambda})+\tilde{\lambda}\tilde{u}_n(\tilde{\lambda})^2}{\tilde{\lambda}\,\tilde{v}_n(\tilde{\lambda})} \; . $$
Because the term \,$\tilde{a}_n(\tilde{\lambda}) + \tilde{\lambda}\tilde{u}_n(\tilde{\lambda})^2$\, vanishes at the divisor points \,$\tilde{\beta}_{n,k}$\,, which are also the zeros of the polynomial \,$\tilde{\lambda}\,\tilde{v}_n(\tilde{\lambda})$\,, we see that \,$\tilde{w}_n(\tilde{\lambda})$\, is a polynomial for every \,$n$\,. The convergence of all terms on the right hand side of the above equation for \,$n\to \infty$\, now implies that also the polynomials \,$\tilde{w}_n(\tilde{\lambda})$\, converge to a polynomial \,$\tilde{w}_\infty(\tilde{\lambda})$\, for \,$n\to \infty$\,. This completes the proof of the convergence of the \,$\tilde{\zeta}_n$\,. We postpone showing that the degree of \,$\tilde{\zeta}_\infty$\, equals \,$\tilde{g}$\, until a later part of the proof.

We have \,$\diff z = r_n\,\diff \tilde{z}$\, and \,$\tilde{\zeta}_{n,k} = s_n\,\ell_n^k\,\zeta_{n,k}$\, and therefore due to our choice \,$r_n=s_n$\,
\begin{align*}
  \alpha_{\ell_n \tilde{\lambda}}(\zeta_n)
  & = \begin{pmatrix} u_{n,0} & v_{n,-1}\,\ell_n^{-1}\,\tilde{\lambda}^{-1} \\ w_{n,0} & -u_{n,0} \end{pmatrix} r_n \diff\tilde{z}
  - \begin{pmatrix} \bar{u}_{n,0} & \bar{w}_{n,0} \\ \bar{v}_{n,-1}\,\ell_n\,\tilde{\lambda} & -\bar{u}_{n,0} \end{pmatrix} r_n \diff\bar{\tilde{z}} \\
  & = \underbrace{\begin{pmatrix} \tilde{u}_{n,0} & \tilde{v}_{n,-1}\,\tilde{\lambda}^{-1} \\ \tilde{w}_{n,0} & -\tilde{u}_{n,0} \end{pmatrix}}_{= \tilde{U}_n} \diff\tilde{z}
  \underbrace{- \begin{pmatrix} \bar{\tilde{u}}_{n,0} & \bar{\tilde{w}}_{n,0} \\ \ell_n^2 \cdot \bar{\tilde{v}}_{n,-1}\,\tilde{\lambda} & -\bar{\tilde{u}}_{n,0} \end{pmatrix}}_{= \tilde{V}_n} \diff\bar{\tilde{z}} \; . 
\end{align*}
By our preceding results on the convergence of \,$\tilde{\zeta}_n$\, and \,$\ell_n\to 0$\,, the second claim of (2) follows.


To prove the remaining statements of (3)--(5),
we use the fact that by Symes' method for the cmc integrable system (Proposition~\ref{P:cmc:symes})
the extended frame \,$F_n(z,\lambda)$\, corresponding to \,$f_n$\, occurs in the Iwasawa decomposition (Proposition~\ref{P:cmc:iwasawa})
\,$\Phi_n(z,\lambda) = F_n(z,\lambda)\cdot B_n(z,\lambda)$\, of \,$\Phi_n(z,\lambda) = \exp\bigr( (z-z_{0}) \zeta_n\bigr)$\,. Similarly,
by Symes' method for the integrable system of Liouville's equation (Proposition~\ref{P:KdV:symes}), the factor \,$F_\infty^{\Liouville}$\, in the
modified Birkhoff decomposition (Proposition~\ref{P:KdV:modified-birkhoff}) \,$\Phi_\infty^{\Liouville}(\tilde{z},\lambda^{\Liouville}) = F_\infty^{\Liouville}(\tilde{z},\lambda^{\Liouville}) \cdot
B_\infty^{\Liouville}(\tilde{z},\lambda^{\Liouville})$\, of \,$\Phi_\infty^{\Liouville}(\tilde{z},\lambda^{\Liouville}) = \exp(\tilde{z}\,\zeta_\infty^{\Liouville})$\, is
the extended frame of a minimal surface immersion \,$\tilde{f}_\infty$\,, which is obtained by the Sym-Bobenko formula of Proposition~\ref{P:KdV:sym-bobenko}.

We now also consider \,$\tilde{\Phi}_n(\tilde{z},\tilde{\lambda}) = \exp( \tilde{z}\,\tilde{\zeta}_n)$\,. For \,$n\to\infty$\,, \,$\tilde{\zeta}_n(\tilde{\lambda})$\, converges
to \,$\zeta_\infty^{\Liouville}(\tilde{\lambda}^{-1})$\, and therefore \,$\tilde{\Phi}_n(\tilde{z},\tilde{\lambda})$\, converges to
\,$\Phi_\infty^{\Liouville}(\tilde{z},\tilde{\lambda}^{-1})$\,. On the other hand we have
\begin{align}
\label{eq:blowup1:blowup1:tildePhin-iwasawa}
\tilde{\Phi}_n (\tilde{z},\tilde{\lambda}) & = \exp\bigr( \tilde{z}\,\tilde{\zeta}_n(\tilde{\lambda})\bigr) = \exp\bigr( \tilde{z}\cdot s_n\,\zeta_n(\ell_n\tilde{\lambda}) \bigr)
\overset{r_n=s_n}{=} \exp\bigr( r_n\tilde{z}\cdot \zeta_n(\ell_n\tilde{\lambda}) \bigr) = \Phi_n(z_{n,0}+r_n\tilde{z},\ell_n\tilde{\lambda}) \\
\notag
& = \tilde{F}_n(\tilde{z},\tilde{\lambda}) \cdot \tilde{B}_n(\tilde{z},\tilde{\lambda})
\quad\text{with}\quad  \tilde{F}_n(\tilde{z},\tilde{\lambda}) = F_n(z_{n,0}+r_n\tilde{z},\ell_n\tilde{\lambda})\;,\;\;
\tilde{B}_n(\tilde{z},\tilde{\lambda}) = B_n(z_{n,0}+r_n\tilde{z},\ell_n\tilde{\lambda})\;. 
\end{align}
By the properties of the extended frame \,$F_n$\,,
\,$\tilde{F}_n(\tilde{z},\tilde{\lambda})$\, is holomorphic for \,$\tilde{\lambda} \in \C^{\times} $\, and we have \,$\tilde{F}_n(\tilde{z},\tilde{\lambda}) \in \mathrm{SU}(2)$\,
for \,$|\tilde{\lambda}| = \ell_n^{-1}$\,, i.e.~\,$|\lambda^{\Liouville}| = \ell_n$\,. Moreover, by the properties of the Iwasawa decomposition,
\,$\tilde{B}_n(\tilde{z},\tilde{\lambda})$\, is holomorphic for \,$\tilde{\lambda} \in B(0,\ell_n^{-1})$\,, i.e.~for \,$\lambda^{\Liouville} \in B(\infty,\ell_n^{-1})$\,,
and is equal to an upper triangular matrix with real diagonal entries for \,$\tilde{\lambda}=0$\,, i.e.~for \,$\lambda^{\Liouville}=\infty$\,.
This shows that the \,$(\ell_n^{-1})$-Iwasawa decomposition \eqref{eq:blowup1:blowup1:tildePhin-iwasawa} of \,$\tilde{\Phi}_n(\tilde{z},\tilde{\lambda})$\, converges to the
modified Birkhoff decomposition of \,$\Phi_\infty^{\Liouville}(\tilde{z},\lambda^{\Liouville}=\tilde{\lambda}^{-1})$\,. This implies the claim about the extended frames in (3).

We now consider also the polynomial Killing fields 
\,$\tilde{\xi}_n(\tilde{z},\tilde{\lambda})$\, with
\,$\tilde{\xi}_n(0,\tilde{\lambda}) = \tilde{\zeta}_n(\tilde{\lambda})$\,. By Equation~\eqref{eq:cmc:xi-basepointchange}
we have
$$ \tilde{\xi}_n(\tilde{z}, \tilde{\lambda}) = \tilde{F}_n(\tilde{z},\tilde{\lambda})^{-1}\cdot \tilde{\zeta}_n(\tilde{\lambda}) \cdot \tilde{F}_n(\tilde{z},\tilde{\lambda}) \; .  $$
Because of \,$\tilde{F}_n(\tilde{z},\tilde{\lambda}) \to \tilde{F}_\infty(\tilde{z},\tilde{\lambda})$\, and
\,$\tilde{\zeta}_n(\tilde{\lambda}) \to \tilde{\zeta}_\infty(\tilde{\lambda})$\,, we see that \,$\tilde{\xi}_n(\tilde{z},\tilde{\lambda})$\, converges locally uniformly in
\,$\tilde{z}\in \C$\, to a function \,$\tilde{\xi}_\infty(\tilde{z},\tilde{\lambda})$\,. Then \,$\xi^{\Liouville}(\tilde{z},\lambda^{\Liouville})
= \tilde{\xi}_\infty(\tilde{z},(\lambda^{\Liouville})^{-1})$\,
is a solution of the differential equation \,$\diff \xi^{\Liouville}_\infty + [\alpha^{\Liouville}_{\lambda^{\Liouville}}(\xi^{\Liouville}),\xi^{\Liouville}] = 0$\, with
\,$\xi^{\Liouville}_\infty(0,\lambda^{\Liouville}) = \zeta^{\Liouville}_\infty(\lambda^{\Liouville}) \in \mathcal{P}^{\Liouville}_{\tilde{d}}$\,. Therefore we have \,$\xi^{\Liouville}_\infty(\tilde{z},\lambda^{\Liouville}) \in \mathcal{P}^{\Liouville}_{\tilde{d}}$\, for all \,$\tilde{z}\in \bbC$\, by the analogue of the Pinkall-Sterling iteration in Proposition~\ref{P:KdV:ps}. This completes the proof of (3). 

The lower left entry of the \,$\tilde{\lambda}^0$-term of \,$\tilde{\xi}_n$\,
is equal to \,$-Q\,e^{-\tilde{\omega}_n/2}$\, by Equation~\eqref{eq:cmc:xi-vm1w0}. Therefore \,$\tilde{\omega}_n$\, converges locally uniformly to the smooth, real-valued
function \,$\tilde{\omega}_\infty$\, such that the lower left entry of the \,$(\lambda^{\Liouville})^0$-term of \,$\xi_\infty^{\Liouville}$\, is equal to \,$-Q\,e^{-\tilde{\omega}_\infty/2}$\,, giving (4).
Due to Equation~\eqref{eq:KdV:xi-terms} this function \,$\tilde{\omega}_\infty$\, will indeed be the conformal factor of the induced metric of \,$\tilde{f}_\infty$\,.
We also note that the sequence \,$(\alpha_{\tilde{\lambda}}(\tilde{\xi}_n))_{n\in \N}$\, converges to \,$\alpha^{\Liouville}_{\lambda^{\Liouville}=\tilde{\lambda}^{-1}}(\xi^{\Liouville}_\infty)$\,.

We now show that the degree of \,$\tilde{\xi}_\infty$\, is indeed \,$\tilde{g}$\,. Due to the equation 
$$ \tilde{a}_\infty(\tilde{\lambda}) = -\tilde{\lambda}\,\det(\tilde{\xi}_\infty) = \tilde{\lambda}\tilde{v}_\infty(\tilde{\lambda})\tilde{w}_\infty(\tilde{\lambda}) - \tilde{\lambda}\,\tilde{u}(\tilde{\lambda})^2 $$
this is equivalent to the condition that in the lowest order term of \,$\xi^{\Liouville}_\infty$\, with respect to \,$\lambda^{\Liouville}$\,, both off-diagonal entries \,$\tau_k$\, and \,$\sigma_k$\, are non-zero. The latter statement follows from the fact that \,$\xi^{\Liouville}_\infty$\, is a Liouville polynomial Killing field and therefore satisfies the Pinkall-Sterling iteration in the form of Proposition~\ref{P:KdV:ps}. Indeed, if \,$\tau_k$\, vanishes for some \,$k$\,, then \,$u_k$\, vanishes by the right-hand equation in \eqref{eq:KdV:ps:ps-tau} and therefore \,$\sigma_k$\, vanishes by Equation~\eqref{eq:KdV:ps:ps-sigma} applied to \,$k$\, in the place of \,$k-1$\,. If \,$\sigma_k$\, vanishes for some \,$k$\,, then \,$u_k$\, is holomorphic by \eqref{eq:KdV:ps:ps-sigma}, therefore \,$u_k$\, vanishes by \eqref{eq:KdV:ps:uk-lin-liouville}, and hence \,$\tau_k$\, is constant by Equations~\eqref{eq:KdV:ps:ps-tau}. By choosing the integration constant \,$C_k=0$\, we can ensure \,$\tau_k=0$\,. This argument also shows that among the \,$\tilde{\beta}_{n,k}$\, there are exactly \,$\tilde{g}$-many which have a finite limit as \,$n\to \infty$\, (\,$\# S = \tilde{g}$\,), because the finite \,$\tilde{\beta}_{\infty,k}$\, are precisely the roots of the polynomial \,$\tilde{\lambda}\tilde{v}_\infty(\tilde{\lambda})$\,. 

Finally we show the convergence of the sequence \,$(\tilde{f}_n)$\, of blown up cmc immersions to a minimal immersion \,$\tilde{f}_\infty$\,.  For this purpose we express the cmc immersions \,$f_n$\, by the
Sym-Bobenko formula \eqref{eq:cmc:sym-bobenko} for cmc surfaces with respect to some Sym point \,$\lambda_s = e^{\mi\varphi} \in S^1$\,. We shall blow up this formula with respect to the spatial parameter \,$z$\,, while keeping the spectral parameter \,$\lambda$\, fixed at the Sym point \,$\lambda_s$\,. We will prove that all terms in this formula converge under this blow-up, and that the limits satisfy the Sym-Bobenko formula for minimal surfaces (Proposition~\ref{P:KdV:sym-bobenko}).

After applying suitable translations in \,$\R^3$\,
so that \,$f_n(z_{0})=0$\,, we have
$$ f_n(z) = G_n \cdot F_n^{-1}\biggr|_{\lambda=e^{\mi\varphi}} \quad\text{with}\quad G_n(z,\lambda) := -\frac{1}{H}\,\mi\,\lambda\,\frac{\partial F_n}{\partial \lambda} $$
and therefore
$$ \tilde{f}_n(\tilde{z}) = h_n^{-1} \, f_n(z_{0}+r_n\tilde{z}) = h_n^{-1}\,G_n(z_{0}+r_n\tilde{z},\lambda)\,F_n^{-1}(z_{0}+r_n\tilde{z},\lambda) \biggr|_{\lambda=e^{\mi\varphi}}
\!\!\!\!\!\!\!\!= \tilde{G}_n(\tilde{z},\lambda=e^{\mi\varphi})\cdot \tilde{F}_n(\tilde{z},\tilde{\lambda}=\ell_n^{-1} e^{\mi\varphi})^{-1}, $$
where we
$$ \text{recall}\quad \tilde{F}_n(\tilde{z},\tilde{\lambda}) = F_n(z_{0}+r_n\tilde{z},\ell_n\tilde{\lambda}) \quad\text{and define}\quad
\tilde{G}_n(\tilde{z},\lambda) = h_n^{-1}\cdot G_n(z_{0}+r_n\tilde{z},\lambda) $$
(note that we do not blow up the parameter \,$\lambda$\, of \,$\tilde{G}_n$\,). Here we already know that the blown-up extended frame \,$\tilde{F}_n$\, converges for \,$n \to \infty$\,, whereas we still need to investigate the limit behaviour of \,$\tilde{G}_n$\,.
We have 
$$ \tilde{G}_n(\tilde{z},\lambda)
= -h_n^{-1}\,\frac{\mi}{H} \lambda \frac{\partial F_n(\tilde{z}_{0}+r_n\tilde{z},\lambda)}{\partial \lambda} \;. $$
Now \,$\diff F_n = F_n\,\alpha_n$\,, where \,$\alpha_{n,\lambda} = \alpha_\lambda(\xi_n)$\, is given by Equation~\eqref{eq:cmc:alpha} with respect to
the solution \,$\omega=\omega_n$\, of the sinh-Gordon equation, and therefore \,$\diff \tfrac{\partial F_n}{\partial \lambda} = \tfrac{\partial F_n}{\partial \lambda} \alpha_n + F_n
\tfrac{\partial \alpha_n}{\partial \lambda}$\,, whence 
$$ \diff \tilde{G}_n = \tilde{G}_n\,\alpha_n + \tilde{F}_n \,\beta_n $$
follows with
$$ \beta_n = - h_n^{-1}\frac{\mi}{H}\lambda \frac{\partial \alpha_n}{\partial \lambda}
\overset{\eqref{eq:cmc:sym-bobenko:beta}}{=} h_n^{-1} \frac{\mi}{2}e^{\omega_n/2}\begin{pmatrix} 0 & \lambda^{-1}\,\diff z \\ \lambda \diff \bar{z} & 0 \end{pmatrix}
= e^{\tilde{\omega}_n/2}\,\frac{\mi}{2} \begin{pmatrix} 0 & \lambda^{-1}\,\diff\tilde{z} \\ \lambda \diff\bar{\tilde{z}} & 0 \end{pmatrix} \; . $$
Here we again used the fact that due to our choices of blow-up factors \eqref{eq:blowup1:blowup1:factors} we have \,$h_n^{-1} e^{\omega_n/2}\diff z
= e^{\tilde{\omega}_n/2}\diff \tilde{z}$\,. We now consider these objects at the Sym point \,$\lambda=e^{\mi\varphi}$\,, which corresponds to \,$\tilde{\lambda}=\ell_n^{-1}
e^{\mi\varphi}$\,, and take the limit as \,$n\to \infty$\,. Then \,$\tilde{F}_n(\tilde{z},\tilde{\lambda})$\, converges to \,$\tilde{F}_\infty(\tilde{z},\infty)
= F_\infty^{\Liouville}(\tilde{z},0)$\,, and the factors of \,$\diff \tilde{z}$\, and of \,$\diff \bar{\tilde{z}}$\, in \,$\alpha_{n,\lambda}$\, converge to the corresponding
factors in \,$\alpha^{\Liouville}_{\lambda^{\Liouville}=0}(\xi^{\Liouville}_\infty)$\,. Moreover the \,$\diff \tilde{z}$-\, and \,$\diff \overline{\tilde{z}}$\,-factors of \,$\beta_n$\, converge to the corresponding factors of
$$ \beta_\infty = e^{\tilde{\omega}_\infty/2}\,\frac{\mi}{2} \begin{pmatrix} 0 & e^{-\mi\varphi}\,\diff\tilde{z} \\ e^{\mi\varphi} \diff\bar{\tilde{z}} & 0 \end{pmatrix} \; . $$
Hence \,$\tilde{G}_n(\tilde{z},\lambda)$\, converges at least for \,$\lambda=e^{\mi\varphi}$\, to the solution of
\,$\diff \tilde{G}_\infty = \tilde{G}_\infty \alpha^{\Liouville}_{\infty,\lambda^{\Liouville}=0} + \tilde{F}_\infty \beta_\infty$\,. Then \,$\tilde{f}_n = \tilde{G}_n\cdot \tilde{F}_n^{-1}$\,
converges to \,$\tilde{G}_\infty \cdot \tilde{F}_\infty^{-1} =: \tilde{f}_\infty$\,, which by Proposition~\ref{P:KdV:sym-bobenko} is a minimal conformal immersion
with induced metric \,$e^{\tilde{\omega}_\infty}\,\diff \tilde{z}\,\diff \bar{\tilde{z}}$\, and Hopf differential \,$Q\,e^{\mi\varphi}\,\diff \tilde{z}^2$\,.
Thus (5) is shown when we choose the Sym point \,$\lambda_s = e^{\mi\varphi}=1$\,. 
\end{proof}

\begin{example}
	\label{E:blowup1:example}
	We describe examples of situations where the conditions (a)--(c) from Theorem~\ref{T:blowup1:blowup1} are fulfilled.
	
	Let an \underline{odd} number \,$g = 2d+1 \geq 1$\, and a sequence \,$(a_n)_{n\in \bbN}$\, of polynomials \,$a_n \in \mathcal{S}^g_1$\, be given. Let \,$\lambda_{n,1},\dotsc,\lambda_{n,g}$\, be the zeros of \,$a_n$\, inside the unit disk, again ordered as in Equation~\eqref{eq:cmc:lambdak-ordering}. 
	We make the assumption that all \,$\lambda_{n,k}$\, go to zero at the same rate, meaning that \,$\ell_n := |\lambda_{n,1}| \to 0$\, and that for all \,$k$\,, \,$\tilde{\lambda}_{\infty,k} = \lim_{n\to\infty} (\ell_n^{-1}\,\lambda_{n,k})$\, converges in \,$\bbC^{\times} $\,. In particular condition (a) from Theorem~\ref{T:blowup1:blowup1} holds. 
	We moreover assume that the \,$\tilde{\lambda}_{\infty,k}$\, are pairwise unequal. Note that we do not prove here that these assumptions can be satisfied by spectral curves of cmc tori, but we conjecture that this is true.
	
	For \,$k\in \{1,\dotsc,g\}$\, we now choose
	\begin{equation}
		\label{eq:blowup1:example:beta-choice}
		\beta_{n,k} = \begin{cases} \lambda_{n,k} & \text{for \,$k$\, odd} \\ \bar{\lambda}_{n,k}^{-1} & \text{for \,$k$\, even} \end{cases} \; .
	\end{equation}
	Then we define (see Equations~\eqref{eq:cmc:xi-trace:omega} and \eqref{eq:cmc:xi-trace:v}) the ``number'' \,$e^{\omega_n(z_{0})}$\, and the polynomials \,$v_n(\lambda)$\, and \,$w_n(\lambda)$\, 
	by
	$$ e^{\omega_n(z_{0})} = 2\frac{|Q|}{H}\prod_{k=1}^g |\beta_{n,k}|\;,\quad 
	v_n(\lambda) = \frac12 H e^{\omega_n(z_{0})/2}\,\lambda^{-1} \prod_{k=1}^g \beta_{n,k}
	\quad\text{and}\quad w_n(\lambda) = -\lambda^{g-1}\,\overline{v_n(\bar{\lambda}^{-1})} \; . $$
	With these definitions we have 
	$$ |\lambda_{n,1}|^{-1} \cdot e^{\omega_n(z_{0})} = 2\frac{|Q|}{H} \,\frac{|\beta_{n,1}|}{|\lambda_{n,1}|} \prod_{j=1}^d |\beta_{n,2j}|\cdot|\beta_{n,2j+1}| 
	= 2\frac{|Q|}{H} \prod_{j=1}^d \frac{|\lambda_{n,2j+1}|}{|\lambda_{n,2j}|} \; . $$
	By our assumptions, this expression converges for \,$n\to\infty$\, to \,$2\tfrac{|Q|}{H} \prod_{j=1}^d \tfrac{|\tilde{\lambda}_{\infty,2j+1}|}{|\tilde{\lambda}_{\infty,2j}|} \in \bbR^+$\,, hence condition (b) from Theorem~\ref{T:blowup1:blowup1} is also satisfied. 
	
	Due to our construction,
	$$ \zeta_n(\lambda) = \begin{pmatrix} 0 & v_n(\lambda) \\ w_n(\lambda) & 0 \end{pmatrix} \in \calP_g $$
	is a cmc potential of spectral genus \,$g$\, that is in the isospectral set of \,$a_n(\lambda)$\,. The corresponding cmc immersion \,$f_n$\, with base point \,$z_{0}$\,, mean curvature \,$H$\, and constant Hopf differential \,$Q\,\mathrm{d}z^2$\, indeed has the conformal factor \,$e^{\omega_n(z_{0})}$\, that was defined above. We finally note that due to our choice of \,$\beta_{n,k}$\, and our assumption that the \,$\tilde{\lambda}_{\infty,k}$\, are pairwise unequal, the condition of Theorem~\ref{T:blowup1:blowup1}(c) is also satisfied. Therefore Theorem~\ref{T:blowup1:blowup1} is applicable in this situation.
\end{example}


\begin{example}
In the situation investigated in this section, one special case is of particular interest to us. We choose \,$H=1$\, and \,$Q=\tfrac12$\,, and the Sym point
\,$\lambda_s = -\mi$\, in Proposition~\ref{P:cmc:sym-bobenko} resp.~\,$\varphi = -\tfrac\pi2$\, in Proposition~\ref{P:KdV:sym-bobenko}. 
Suppose that the spectral genus \,$g=2d+1$\, is odd, that 
\,$\lambda_{n,1}$\, goes to zero as \,$n\to\infty$\,, and that for \,$k\geq 2$\, we have \,$\tfrac{|\lambda_{n,k}|}{|\lambda_{n,1}|} \to \infty$\,. Moreover we suppose that \,$\beta_{n,1}$\, goes to zero at the same rate as \,$\lambda_{n,1}$\, does,
i.e.~\,$\tilde{\beta}_{n,1} = |\lambda_{n,1}|^{-1}\,\beta_{n,1}$\, has a limit \,$\tilde{\beta}_{\infty,1} \in \C \setminus \mathbb{D}$\,, that for \,$k\geq 2$\, we have
\,$\tfrac{|\beta_{n,k}|}{|\lambda_{n,1}|} \to \infty$\,, and that these divisor points are pairwise conjugate to each other in the sense that
\begin{equation}
\label{eq:blowup:helicoid:beta-cancel}
|\beta_{n,2j}| < 1 \AND \beta_{n,2j+1} = \bar{\beta}_{n,2j}^{-1} \quad\text{holds for \,$j=1,\dotsc,d$\,.}
\end{equation}
For example, choosing \,$\beta_{n,k}$\, as in Equation~\eqref{eq:blowup1:example:beta-choice} satisfies these assumptions. 

Equation~\eqref{eq:cmc:xi-trace:omega} shows that the hypotheses (a)--(c) of Theorem~\ref{T:blowup1:blowup1} are satisfied in the present situation.
In Theorem~\ref{T:blowup1:blowup1} we have \,$\tilde{d}=1$\, and therefore the blown up minimal immersion \,$\tilde{f}_\infty$\, has spectral genus
\,$\lfloor \tfrac{\tilde{d}}{2} \rfloor = 0$\,. The limiting polynomial Killing field \,$\zeta^{\Liouville}_\infty$\, at the base point \,$\tilde{z}=0$\,
is therefore of the form
$$ \zeta^{\Liouville}_\infty = \begin{pmatrix} 0 & v_1 \\ 0 & 0 \end{pmatrix} \lambda^{\Liouville} + \begin{pmatrix} u_0 & v_0 \\ w_0 & -u_0 \end{pmatrix} \in \mathcal{P}^{\Liouville}_0 \; . $$
Here we have by Equations~\eqref{eq:KdV:xi-terms}, \eqref{eq:cmc:xi-trace:omega} and \eqref{eq:blowup:helicoid:beta-cancel}
\begin{align*}
v_1 & = \tfrac14 e^{\tilde{\omega}_\infty(0)} = \tfrac14 \lim_{n\to\infty} \bigr( \ell_n^{-1} e^{\omega_n(z_0)} \bigr)
= \tfrac14 \lim_{n\to\infty} \bigr( \tfrac{2|Q|}{H} \ell_n^{-1} \prod_{k=1}^g |\beta_{n,k}| \bigr) \\
& = \tfrac14 \lim_{n\to\infty} \bigr( \tfrac{2|Q|}{H} \ell_n^{-1} |\beta_{n,1}| \bigr) 
= \tfrac14 |\tilde{\beta}_{\infty,1}| \geq \tfrac14 \; .
\end{align*}
Therefore there exists \,$x\in \R$\, with \,$v_1 = \tfrac14\cosh(x)$\,. We then have by Equation~\eqref{eq:KdV:xi-terms}
$$ w_0 = -\tfrac14\,Q\,v_1^{-1} = -\tfrac12 \cosh(x)^{-1} \; . $$
By a finite rescaling of the coordinate \,$\tilde{z}$\, we can obtain that \,$u_0 = \tfrac12 (\tilde{\omega}_\infty)_{\tilde{z}} = \tfrac12 \tanh(x)$\, holds,
and by choosing the integration constant for \,$v_0$\, in the Pinkall-Sterling iteration of Proposition~\ref{P:KdV:ps} appropriately, we can moreover
obtain \,$v_0 = -2\cosh(x)^{-1}$\,.
To sum up, after these transformations we have
$$ \zeta^{\Liouville}_\infty = \begin{pmatrix} 0 & \tfrac14 \cosh(x) \\ 0 & 0 \end{pmatrix} \lambda^{\Liouville} + \begin{pmatrix} \tfrac12 \tanh(x) & -2 \cosh(x)^{-1} \\ -\tfrac12 \cosh(x)^{-1} & -\tfrac12 \tanh(x) \end{pmatrix} \in \mathcal{P}^{\Liouville}_0 \; . $$
This is the polynomial Killing field of a helicoid in \,$\R^3$\,, as we saw in Example~\ref{E:KdV:helicoid}. Under the stated circumstances, the blowup of a sequence
of cmc tori at the fastest possible rate (\,$\ell_n = |\lambda_{n,1}|$\,) therefore produces a helicoid in \,$\R^3$\,.
\end{example}

\bibliography{ref}

\end{document}